\documentclass{amsart}

\usepackage[hmargin=3.5cm,top=3cm,bottom=3cm,includehead]{geometry}

\usepackage{amsmath, amssymb, amsthm, mathrsfs}
\usepackage{graphicx,color,enumerate} 

\usepackage[pdfborder={0 0 0}]{hyperref}

\newcommand\N{{\mathbb N}}
\newcommand\R{{\mathbb R}}
\newcommand\T{{\mathbb T}}

\def\AA{{\mathcal A}}
\def\BB{{\mathcal B}}

\def\EE{{\mathcal E}}
\def\FF{{\mathcal F}}

\def\HH{{\mathcal H}}

\def\LL{{\mathcal L}}

\def\OO{{\mathcal O}}

\def\RR{{\mathcal R}}
\def\SS{{\mathcal S}}

\def\XX{{\mathcal X}}
\def\YY{{\mathcal Y}}
\def\ZZ{{\mathcal Z}}

\def\BBB{{\mathscr B}}

\def\e{\epsilon}
\def\eps{{\varepsilon}}

\newcommand{\la}{\langle}
\newcommand{\ra}{\rangle}

\newcommand{\Nt}{|\hskip-0.04cm|\hskip-0.04cm|}

\newtheorem{thm}{Theorem}[section]

\newtheorem{prop}[thm]{Proposition}
\newtheorem{lem}[thm]{Lemma}
\newtheorem{cor}[thm]{Corollary}
\newtheorem*{thm*}{Theorem}
\theoremstyle{remark}
\newtheorem{rem}[thm]{Remark}

\theoremstyle{definition}

\numberwithin{equation}{section}

\newcommand{\be}{\begin{equation}}
\newcommand{\ee}{\end{equation}}
\newcommand{\ba}{\begin{aligned}}
\newcommand{\ea}{\end{aligned}}

\newcommand{\beqn}{\begin{equation}}
\newcommand{\eeqn}{\end{equation}}
\newcommand{\bear}{\begin{eqnarray}}
\newcommand{\eear}{\end{eqnarray}}
\newcommand{\bean}{\begin{eqnarray*}}
\newcommand{\eean}{\end{eqnarray*}}

\title[Landau equation near Maxwellians]
{Landau equation for very soft and Coulomb potentials near Maxwellians}
\author{K. Carrapatoso}
\author{S. Mischler}

\address[Kleber Carrapatoso]{CMLA, ENS Cachan, CNRS, Universit\'e Paris-Saclay, 94235 Cachan, France.}
\email{carrapatoso@cmla.ens-cachan.fr}

\address[St\'ephane Mischler]{Universit\'e Paris-Dauphine \& Institut Universitaire de France (IUF), PSL Research University,
CNRS, UMR [7534], CEREMADE, 
Place du Mar\'echal de Lattre de Tassigny
75775 Paris Cedex 16, 
France.}
\email{mischler@ceremade.dauphine.fr}


\begin{document}

\begin{abstract}
This work deals with the Landau equation for very soft and Coulomb potentials near the associated Maxwellian equilibrium. 
We first investigate the corresponding linearized operator and develop a method to prove 
 strong asymptotical (but not uniformly exponential)  stability estimates of its associated semigroup in large functional spaces. 
We then deduce existence, uniqueness and fast decay of the solutions to the nonlinear equation in a close-to-equilibrium framework. Our result drastically improves the set of initial data compared to the one considered by Guo and Strain who established similar results in \cite{Guo,GS1,GS2}. Our functional framework is compatible with the non perturbative frameworks developed by Villani, Desvillettes and co-authors \cite{Vi2,DV-boltzmann,D,CDH}, 
and our main result then makes possible to improve the speed of convergence to the equilibrium established therein. 
\end{abstract}

\maketitle

\begin{center}
{\bf \today}
\end{center}

\tableofcontents
  
\section{Introduction}

\subsection{The Landau equation}
The Landau equation is a fundamental equation in kinetic theory modeling the evolution of a dilute plasma interacting through binary collisions. 
We consider here a plasma confined in a torus $\T^3$ and described by the distribution function $F=F(t,x,v) \ge 0$ of particles which at time $t \ge 0$ and at position $x\in\T^3$, move with the velocity  $v\in\R^3$. 
The evolution of $F$ is governed by the \emph{spatially inhomogeneous} Landau equation
\be\label{eq:landau-inhom}
\left\{
\ba
& \partial_t F + v \cdot \nabla_x F =  Q(F,F)  \\
& F(0,x,v) = F_0(x,v).
\ea
\right.
\ee
For a spatially homogeneous plasma, namely when $F=F(t,v)$, the equation simplifies into the \emph{spatially homogeneous} Landau equation
\be\label{eq:landau}
\left\{
\ba
& \partial_t F  =  Q(F,F)  \\
& F(0,v) = F_0(v).
\ea
\right.
\ee
The Landau collision operator $Q$ is a bilinear operator acting only on the velocity variable and it is given by
\be\label{eq:oplandau0}
Q(g,f) (v) = \partial_{i} \int_{\R^3} a_{ij}(v-v_*) \left\{ g_* \partial_j f - f \partial_{j}g_*\right \} \, dv_*,
\ee
where here and below we use the convention of implicit summation over repeated indices and the usual shorthand $g_* = g(v_*)$, $\partial_{j} g_* = \partial_{v_{*j}} g(v_*)$, $f=f(v)$ and $\partial_j f = \partial_{v_j} f(v)$. 
The matrix-valued function $a$ is nonnegative, symmetric and depends on the interaction between particles. When particles interact by an inverse power law potential, $a$ is given by 
\be\label{eq:aij}
a_{ij}(z) = |z|^{\gamma+2} \left( \delta_{ij} - \frac{z_i z_j}{|z|^2}\right), \quad -3 \le \gamma \le 1.
\ee
In the present article, we shall consider the cases of very soft potentials $\gamma \in (-3,-2)$ and Coulomb potential $\gamma=-3$.  It is worth mentioning that the Coulomb potential is the most physically interesting case, and also the most difficult because of the strong singularity in \eqref{eq:aij}.

\smallskip
The Landau equation \eqref{eq:landau-inhom} (or \eqref{eq:landau}) possesses two fundamental properties (which hold at least formally). On the one hand, it conserves mass, momentum and energy, more precisely 
%Indeed, for any test function $\varphi$, we have
%$$
%\int_{\R^3} Q(f,f) \varphi \, dv =\frac12 \int_{\R^3 \times \R^3} a_{ij}(v-v_*) f f_*  \left(  \frac{\partial_i f}{f} -  \frac{\partial_{i} f_*}{f_*} \right)   \left( \partial_j \varphi - \partial_{j} \varphi_* \right) \, dv \, dv_*,
%$$
%from which we deduce that
\beqn \label{eq:conserv}
\frac{d}{dt} \int_{\T^3 \times \R^3} F \varphi \, dx \, dv=
\int_{\T^3 \times \R^3} \left\{  Q(F,F) - v \cdot \nabla_x f \right\} \varphi \, dx \, dv = 0 \quad \text{for} \quad \varphi(v) = 1,v,|v|^2.
\eeqn
On the other hand, the Landau version of the celebrated Boltzmann $H$-theorem holds: the entropy $H(F) := \int F \, \log F \, dx \, dv $ is non-increasing and the global equilibria are global Maxwellian distributions that are independent of time and position.
Hereafter, we normalize the initial data 
$$
\int_{\T^3 \times \R^3} F_0 \, dx \, dv = 1, \quad
\int_{\T^3 \times \R^3} F_0 \, v \, dx \, dv = 0, \quad
\int_{\T^3 \times \R^3} F_0 \, |v|^2 \, dx \, dv = 3,
$$
and therefore we consider the associated global Maxwellian equilibrium
$$
\mu(v)=(2 \pi)^{-3/2} e^{-|v|^2/2},
$$
with same mass, momentum and energy of the initial data (normalizing the volume of the torus to $|\T^3_x| = 1$).

\subsection{Main results}
Our aim in this work is to study the Landau equation in a close-to-equilibrium framework  (or perturbative regime) in large functional spaces and to establish new well-posedness and trend to the equilibrium results. 

\smallskip
Let us then introduce the functional framework we will work with. For a given velocity weight function $m=m(v) : \R^3 \to \R_+$ and exponent $1 \le p \le \infty$, we define the associated weighted Lebesgue space $L^p_v(m)$ 
and weighted Sobolev space $W^{1,p}_v(m)$,  through their norms 
\be\label{eq:defLpm}
\| f \|_{L^p_v(m)} := \| m f  \|_{L^p_v}, \quad  
\| f \|_{W^{1,p}_v(m)} := \| m f  \|_{W^{1,p}_v}.
\ee
Similarly, we define the weighted Sobolev space $W^{n,p}_x L^{p}_v(m)$, $n \in \N$, associated to the norm
\be\label{eq:defWxWvm}
\| f \|_{W^{n,p}_x L^{p}_v(m)}^p :=  \| m f \|_{W^{n,p}_x L^{p}_v}^p
:=\sum_{0 \le j \le n } \|  \nabla^j_x   (mf) \|_{L^p_{x,v}}^p,
\ee
and we adopt the usual notation $H^n = W^{n,2}$.

\smallskip

We make the following assumption on the weight function $m$ :
\be\label{eq:m}
\ba
 m &= \la v \ra^k := (1 + |v|^2)^{k/2}  \text{ with }   k > 2 + 3/2 ; \\
 m &= \exp(\kappa \la v \ra^s)  \text{ with }  s \in (0,2) \text{ and } \kappa > 0, 
\text{ or } s=2 \text{ and }  \kappa \in (0,1/2);
\ea
\ee
and through the paper we denote $\sigma=0$ when $m$ is a polynomial weight, and $\sigma=s$ when $m$ is an exponential weight. We associate the decay functions
\be\label{eq:Theta_m}
\Theta_{m}(t) =
\left\{
\ba
&  C \,  \langle t \rangle^{ -\frac{k - \ell}{|\gamma|}} ,   & \quad   \text{if } m = \la v \ra^{k}, \\
& C e^{ -\lambda \, t^{ \frac{s}{|\gamma|} }  },   & \quad\text{if } m = e^{\kappa \la v \ra^s},  \\
\ea
\right.
\ee
for any constant $ \ell \in (2+3/2,k)$ and some constants $C, \lambda \in (0,\infty)$. It is worth emphasizing that in the polynomial case 
$ m = \la v \ra^{k}$, the notation $\Theta_m$ refers to a class of functions (with increasing rate of decay as $\ell$ tends to $ 2+3/2$), while in the exponential case $m = e^{\kappa \la v \ra^s}$, the notation $\Theta_m$ stands for a fixed function. 
We finally introduce the projection operator $P_v$ on the $v$-direction for any given $v \in \R^3 \backslash \{ 0 \}$ defined by 
\beqn\label{eq:def:Pv}
P_v \xi = \left( \xi \cdot \frac{v}{|v|}   \right) \frac{v}{|v|}, \quad \forall \, \xi \in \R^3, 
\eeqn
as well as the anisotropic gradient $\widetilde \nabla_v f$ of a function $f$ 
defined by 
\beqn\label{eq:def:anisoG}
\widetilde \nabla_v f = P_v \nabla_v f + \la v \ra (I-P_v) \nabla_v f. 
\eeqn

Our main result reads as follows. 

\begin{thm}\label{thm:stabNL-inhom}
For any weight function $m$ satisfying \eqref{eq:m}, there exist $C > 0$ and $\eps_0 >0$, small enough, so that, if $\| F_0 - \mu \|_{H^2_x L^2_v(m)} < \eps_0$,   there exists a unique global weak solution $F$ to \eqref{eq:landau-inhom} such that
\be\label{eq:bound-g-inhom} 
\ba
\sup_{t \ge 0} \| F(t) - \mu \|_{H^2_x L^2_v(m)}^2 
&+ \int_0^\infty  \| \la v \ra^{\frac{\gamma+\sigma}{2}} (F(t) - \mu)   \|_{H^2_x L^2_v(m)}^2 \, dt \\
&+ \int_0^\infty \| \la v \ra^{\frac{\gamma}{2}}  \widetilde \nabla_v \{ m (F(t) - \mu ) \} \|_{H^2_x L^2_v}^2 \, dt \le C \eps_0^2.
\ea
\ee
This solution verifies the decay estimate 
\be\label{eq:decay-g-inhom}
\| F(t) - \mu \|_{H^2_x L^2_v} \le \Theta_{m}(t) \, \| F_0 - \mu \|_{H^2_x L^2_v (m)}, \quad \forall \, t \ge 0.
\ee
\end{thm}

\begin{rem} For a spatially homogeneous initial datum $F_0 \in L^2_v(m)$, the associated solution $F(t)$ is also a spatially homogeneous function, and thus satisfies the spatially homogeneous Landau equation~\eqref{eq:landau}. In that spatially homogeneous framework, the $H^2_x$ regularity is automatically fulfilled, it can be then removed of the corresponding version of Theorem~\ref{thm:stabNL-inhom} which statement thus simplifies accordingly. 
%In the spatially homogeneous case \eqref{eq:landau}, Theorem~\ref{thm:stabNL-inhom} also holds in a simplified way (without the $H^2_x$ space), and we refer to Theorem~\ref{thm:stabNL} for a precise statement.
\end{rem}

\smallskip

Let us briefly comment on known results on the existence, uniqueness and long-time behaviour of solutions to the Landau equation when $-3 \le \gamma < - 2$. For the other cases $-2 \le \gamma \le 1$, we refer the reader to the recent work \cite{CTW} and the references therein.

\smallskip
In the space homogeneous case, existence of solutions has been first addressed by Arsenev-Penskov \cite{Arsenev}, and next by Villani \cite{Vi2} and Desvillettes \cite{D} who establish existence of global solutions for any initial datum with finite mass, energy and entropy. Uniqueness of strong solutions (which do exist locally in time) has been proved by Fournier-Gu\'erin \cite{FG} and Fournier \cite{Fournier}. In a similar framework and for bounded (after regularisation) collision  kernel $a$ with $- 3 < \gamma < - 2$, polynomial convergence to the equilibrium has been obtained by Toscani and Villani~\cite{TosVi-slow} by entropy dissipation method. That last result has been recently improved by Desvillettes, He and the first author \cite{CDH}, who prove convergence to equilibrium with algebraic or stretched exponential rate removing the boundedness (unphysical) assumption on the collision kernel $a$ and also considering the Coulomb potential $\gamma=-3$. 
The space homogeneous version of the results by Guo and Stain presented below also 
provides well-posedness and accurate rate of convergence to the equilibrium in a perturbative regime in  $H^8_{v}(\mu^{-\theta})$, $\theta \in (1/2,1)$. It is worth emphasising that even in that simple space homogeneous case, it was the only known result of existence and uniqueness of global (in time) solutions. 

\smallskip
In the space inhomogeneous case, existence of global (renormalized with a defect measure) solutions has been established by Alexandre-Villani \cite{AV} for any initial datum with finite mass, energy and entropy. Under an additional (unverified) strong uniform in time boundedness assumption, Desvillettes and Villani \cite{DV-boltzmann} proved polynomial convergence of the solutions to the equilibrium. On the other hand, in a perturbative regime, Guo \cite{Guo} proved well-posedness in the high-order Sobolev space with fast decay in velocity $H^8_{x,v}(\mu^{-1/2})$, and Guo and Strain \cite{GS1,GS2} proved stretched exponential convergence to equilibrium in $H^8_{x,v}(\mu^{-\theta})$, $\theta \in (1/2,1)$. 

\smallskip
Our result thus improves the well-posedness theory of Guo \cite{Guo} to larger spaces $H^2_x L^2_v(m)$ as well as the convergence to equilibrium of Guo and Strain \cite{GS1,GS2} to larger spaces and with more accurate rate. It is worth emphasising that in the space homogeneous case, our results only require that initial data are bounded (and close) in the Lebesgue space $L^2_v(m)$ (and thus do not require any control on derivatives).

\smallskip

Our result makes possible to improve the speed of convergence to the equilibrium results available in a non perturbative framework in the following way. 
%by just  combining them with our convergence result stated in Theorem~\ref{thm:stabNL-inhom}

\begin{cor}[Spatially homogeneous framework] \label{cor:SpeedHomogeneous} Consider a nonnegative normalized initial datum $F_0 = F_0(v)$ with finite entropy such that furthermore $F_0 \in L^1(m)$ for an exponential weight function $m$ satisfying \eqref{eq:m} with $s \in (0,1/2)$. There exists a global weak solution $F$ to the spatially homogenous Landau equation \eqref{eq:landau} associated to $F_0$ satisfying
\beqn\label{eq:CorHomogeneous}
\| F(t) - \mu \|_{L^2_v} \lesssim \Theta_{m}(t), \quad \forall \, t \ge 0.
\eeqn 
\end{cor}

Estimate \eqref{eq:CorHomogeneous} improves the rate of convergence of order  $e^{ -\lambda \,  t^{ \frac{s}{s+|\gamma|} }  }$ established in \cite{CDH}, thanks to an entropy method, for the global weak solutions built in \cite{Vi2,D}. Corollary~\ref{cor:SpeedHomogeneous} has to be compared with \cite{Mouhot} where the optimal speed of convergence to the equilibrium for the spatially homogeneous Boltzmann equation for hard spheres has been established 
and with \cite{MR3260265} where the optimal speed of convergence to the equilibrium for the spatially homogeneous Boltzmann equation for hard potentials has been proved.

\begin{cor}[Spatially inhomogeneous framework with a priori bounds] \label{cor:SpeedInhomogeneous} Let $F$ be a nonnegative normalized global strong solution to the spatially inhomogeneous Landau equation \eqref{eq:landau-inhom} such that 
\beqn\label{hyp:SpeedInhomo1}
\sup_{t \ge 0} \left( \| F(t) \|_{H_{x,v}^\ell} + \| F(t)  \|_{L^1_{x,v}(m)} \right) < +\infty,
\eeqn
for some explicit $\ell \ge 3$ large enough and some exponential weight function $m$ satisfying \eqref{eq:m}, and such that the spatial density is uniformly positive on the torus, namely 
\beqn\label{hyp:SpeedInhomo2}
\forall \,  t \ge0, \,\, x \in \T^3, \quad\rho(t,x) = \int_{\R^d} f(t,x,v) \, d v \ge \alpha > 0.
\eeqn
Then this solution satisfies
\beqn\label{eq:CorInhomogeneous}
\| F(t) - \mu \|_{H^2_{x} L^2_v} \lesssim \Theta_{m}(t), \quad \forall \, t \ge 0.
\eeqn
\end{cor}

Estimate \eqref{eq:CorInhomogeneous} improves the polynomial (of any order) rate of convergence established in \cite[Theorem 2]{DV-boltzmann} under stronger (of any order) uniform Sobolev norm estimates but weaker (polynomial of any order) velocity moment uniform estimates. Corollary~\ref{cor:SpeedInhomogeneous} has to be compared with \cite{GMM} where the optimal speed of convergence to the equilibrium for the spatially inhomogeneous Boltzmann equation for hard spheres has been established.

\subsection{Overview of the proof}
%Our main theorem is based on  stability estimates (which are not uniformly exponential) for the semigroup corresponding to the \Red{\it weakly dissipative}  associated linearized operator in large functional spaces, by taking advantage of well-known (or simple extension of well-known) weak coercivity estimates in some small space and using an enlargement trick.  We then conclude to our main result by  combining these stability estimates (at the linear level) together with some nonlinear estimates for the Landau operator $Q$ and a trapping argument. It is worth mentioning that our method is mostly based on the semigroup stability estimates, what is quite different of the nonlinear energy method of \cite{Guo,GS1,GS2}. 

Our main theorem is based on stability estimates (which are however not uniformly exponential) for the semigroup corresponding to the associated linearized operator in large functional spaces, by taking advantage of a {\it weak coercivity} estimate in one small space and using an {\it enlargement trick for weakly dissipative operators} that we introduce here.  
We then conclude to our main result by  combining these stability estimates (at the linear level) together with some nonlinear estimates for the Landau operator $Q$ and a trapping argument. It is worth mentioning that our method is mostly based on these semigroup stability estimates, what is drastically different from  the nonlinear energy method of \cite{Guo,GS1,GS2}.

\smallskip
 Let us explain this enlargement trick in more details, and we restrict ourselves to the Hilbert framework to make the discussion simpler (and because it is the only case we will consider in the all paper). We begin with the simpler hypodissipative case.
Let $\Lambda$ be a linear operator acting on two Hilbert spaces $E \subset \EE$ and suppose that $\Lambda$ has a spectral gap in the small space $E$, and more precisely
\beqn\label{eq:introLambdaStrongDissipE}
\forall \, f \in E^\Lambda_1,  \quad\langle \Lambda f , f \rangle_E \lesssim - \| \Pi f \|^2_E, 
\eeqn
where $E^\Lambda_1$ stands for the domain of $\Lambda$ when acting on the space $E$ and $\Pi$ denotes the projector onto the orthogonal of $\mathrm{ker}(\Lambda)$. It is worth recalling that this estimate is equivalent to an exponential rate decay for the associated semigroup $S_\Lambda(t) \Pi$ in $E$.  
The extension theory recently introduced in an abstract Banach framework in \cite{Mouhot} and developed in 
\cite{GMM, MM, Mbook} (see also \cite{MMcmp2009,MR3452720,MiSch} for other developments of the factorization approach for the spectral analysis of semigroups in large Banach spaces) establishes that if we can factorise $\Lambda = \AA + \BB$ where $\BB$ is hypodissipative (with respect to $\EE$), 
$\AA$ is bounded and some convolution product of $\AA S_\BB$ enjoys suitable regularity property, then $\Lambda$ generates a $C_0$-semigroup $S_\Lambda (t)$ on the large space $\EE$   and $S_\Lambda(t) \Pi$ enjoys in $\EE$  the same exponential rate decay as  in $E$. 
This method has been successfully applied to many evolution equations, and in particular to the Landau equation with hard and moderately soft potentials in \cite{KC1,KC2,CTW}.  
%
%We observe here that saying that $\BB$ is hypodissipative in $\EE$ means (assuming $\EE$ is a Hilbert space to simplify)
%$$
%\forall\, f \in \EE^\BB_1, \quad
%\la \BB f , f \ra_X \lesssim - \lambda \| f \|_X^2,
%$$
%where $\EE^\BB_1$ is the domain of $\BB$ when acting on $\EE$, and this estimate immediately implies an exponential decay for the associated semigroup $S_\BB(t)$.

\smallskip
In our case (of very soft and Coulomb potentials $\gamma \in [-3,-2)$), the linearized Landau operator $\Lambda$ does not satisfy any spectral gap inequality but only a weak coercivity estimate on a small space $E$. We are however able to generalize the extension theory presented above and prove that $\Lambda$ generates a uniformly bounded continuous semigroup $S_\Lambda (t)$ on small and large Hilbert spaces $X$, which is now only strongly stable but not uniformly exponentially stable. 

More precisely, on the one hand, the linearized version of the $H$-Theorem states that (at least) in one Hilbert space $E$, the linearized Landau operator $\Lambda$ enjoys a weak spectral gap estimate 
\beqn\label{eq:introLambdaDissipE}
\forall \, f \in E^\Lambda_1,  \quad\langle \Lambda f , f \rangle_E \lesssim - \| \Pi f \|^2_{E_*}, \quad  E_* \,\,\hbox{not included into} \,\, E, 
\eeqn
where here %$E^\Lambda_1$ stands for the domain of $\Lambda$ when acting on the space $E$, $\Pi$ denotes the projector onto the orthogonal of $\mathrm{ker}(\Lambda)$,  and 
$E_*$ is a second Hilbert space (in the norm of which we express the weak dissipativity property of $\Lambda$ in $E$). 

On the other hand, in many Hilbert spaces $X$, the linearized Landau operator $\Lambda$ splits as $\Lambda = \AA + \BB$ where $\AA$ is a bounded operator in $X$ and $\BB$ is weakly dissipative 
\beqn\label{eq:introBBDissipX}
\forall \, f \in X^\Lambda_1,  \quad\langle \BB f , f \rangle_X \lesssim - \| f \|^2_{X_*}, \quad  X_* \,\,\hbox{not included into} \,\, X, 
\eeqn
where again $X^\Lambda_1$ stands for the domain of $\Lambda$ when acting on the space $X$ and $X_*$ is a second Hilbert space (in the norm of which we express the weak dissipativity property of $\BB$ in $X$).

It is worth emphasizing that this weakly dissipative case is much more tricky than the previous classical dissipative case, because one cannot deduce any decay estimate on $\Pi S_\Lambda$ (resp. $S_\BB$)  just from inequality \eqref{eq:introLambdaDissipE} (resp. inequality \eqref{eq:introBBDissipX}).  

%
%It is worth emphasizing that in the weakly dissipative framework is much more difficult to obtain estimates on the semigroup than in the classical dissipative setting, where estimates on the semigroup follow from only one estimate on the operator.
%More precisely, from \eqref{eq:introLambdaDissipE} or \eqref{eq:introBBDissipX} we cannot deduce any uniform exponential stability of the 
%semigroups $S_\BB$ or $\Pi S_\Lambda$. 
%
However, by using \eqref{eq:introBBDissipX} with several choices of spaces $X$ and using an interpolation argument,
we first obtain that $S_\BB$ is strongly asymptotically stable (but not uniformly exponentially stable). 
Next, by using an extension trick, we deduce that the same holds for $\Pi S_\Lambda$. More precisely, for several choices of Hilbert spaces $X \subsetneq X_0$, we have first
\beqn\label{eq:introSLambdaDecayXX0}
\| \Pi S_\Lambda (t) \|_{X \to X_0} \le \Theta (t) \to 0, \quad \hbox{as}\,\, t \to \infty, 
\eeqn
for some polynomial or stretched exponential decay function $\Theta = \Theta_{X, X_0}$,  as well as the regularization estimate
\beqn\label{eq:introSLambdaRegXX'*} 
 \| \Pi S_\Lambda (t) \|_{X'_* \to X_0}  \le   (t\wedge1)^{-1/2} \,  \Theta_* (t), %\frac{\Theta_* (t)}{t^{1/2} \wedge 1} , 
\eeqn
 for some polynomial decay function $\Theta_* = \Theta_{X'_*,X_0}$ (such that $(t\wedge1)^{-1/2} \, \Theta(t) \,  \Theta_* (t) \in L^1(\R_+)$) 
and where $X'_*$ is the dual of $X_*$  for some suitable duality product. 
Next, for some convenient choice of $\eta,K>0$, the norm
\be\label{def:Nt}
{ \forall\, f \in \Pi X}, \quad
\Nt f \Nt_{X}^2 := \eta \| f \|_{X}^2 + \int_0^\infty \| S_\Lambda(\tau) f \|_{X_0}^2 \, d\tau
\ee
is an equivalent norm in $ \Pi X$ and $\Lambda$ satisfies the weak dissipativity estimate 
\be\label{def:introLambdaDissipX}
\forall \, f \in X^\Lambda_1,  \quad \la\!\la \Lambda  f, f \ra\!\ra_{X}  \le - K \| \Pi f \|_{X_*}^2,
\eeqn
where $ \la\!\la \cdot, \cdot \ra\!\ra_{X}$ stands for the duality bracket associated to the  $\Nt \cdot \Nt_X$ norm. 

\smallskip
By choosing $X$ and $X_*$ well adapted for the quadratic Landau operator, we may then establish that for any solution $f = F-\mu$ to the Landau equation, the following a priori estimate holds (for some constant $C>0$)
$$
\frac{d}{dt} \| \Pi f \|_{X}^2 \le \| \Pi f \|_{X_*}^2 (-K + C \| \Pi f \|_{X}).
$$
Our existence, uniqueness and asymptotic stability results are then immediate consequences of that last differential inequality and of the estimates it provides.

\medskip  

Let us finally discuss the decay issue for non-uniformly exponentially stable semigroups which naturally arises in many contexts. 
It arises first in statistical physics when involved coefficients  are suitably decaying. 
In  \cite{Caflisch1,Caflisch2}, for the Boltzmann equation with soft potential of interaction under Grad's cutoff assumption, Caflisch had exhibited the explicit semigroup solution to the associated linearized equation and had deduced well-posedness and stability for the nonlinear Boltzmann equation in a perturbative regime. In  \cite{Liggett}, a similar result is obtained for the critical case of an attractive reversible nearest particle system. More recently, for the Fokker-Planck equation with weak confinement potential and for the spatial homogeneous Landau equation with soft interaction some polynomial and stretch exponential rate of convergence to the equilibrium have been established in  \cite{RockWang,TosVi-slow}. The proofs are based on entropy methods, moments estimates and interpolation arguments. These  results for the Fokker-Planck equation are improved in \cite{KM} where a similar factorization approach, as introduced in the present paper, is developed. 

\smallskip

Independently, inspired by scattering and control  theory \cite{LaxPhil,BLR}, many results on the decay rate of the energy for damped wave type equations have been established, see for instance \cite{Lebeau,LebeauRobb,LebeauZZ,Burq}. These results are based on the analysis of the absence of poles (resonances) in the neighbourhood of the real axis for the resolvent of the associated operator. They have inspired an abstract theory for non-uniformly exponentially stable semigroups, and we refer the interested reader to \cite{BD,BEPS,BCT} and the references therein.

\smallskip
It is worth emphasizing that in that last abstract theory, one typically obtains some estimate on the semigroup by allowing the lost of (part of) a domain in the control of the trajectory and, roughly speaking, that is related to the absence of pole in bounded neighbourhoods of the real axis and to the control of how the spectrum approaches the imaginary axis at $\pm i\infty$. That is slightly different from the picture arising in the present statistical physics framework, where the estimates do not involve domains norms but norms controlling the confinement of the distribution function and where the continuous spectrum extends up to the origin.

\subsection{Notations and definitions}
If $\Lambda$ is a closed linear operator on a Banach space $X$ that generates a semigroup on $X$, we denote by $S_\Lambda(t)$ its associated semigroup. Moreover, for Banach spaces $X$ and $Y$, we denote $\BBB(X,Y)$ the space of bounded linear operators from $X$ to $Y$, with the associated operator norm $\| \cdot \|_{X \to Y}$. We say that the generator $\Lambda$ of a semigroup in a Banach space $X$ is dissipative if 
$$
\forall \, f \in X^1_\Lambda, \,\, \exists \, f^* \in J_f, \quad \langle f^*, \Lambda f \rangle_{X',X} \le 0
$$
where $X^1_\Lambda = D(\Lambda)$ is the domain of $\Lambda$ and $J_f$ is the dual set $J_f := \{ g \in X'; \, \| g \|^2_{X'} = \| f \|^2_{X} = \langle g,  f \rangle_{X',X} \}$. We say that the generator $\Lambda$ is hypodissipative if it is dissipative for an equivalent norm. 

\subsection{Structure of the paper}
For the sake of clarity, we shall first consider the spatially homogeneous case through Sections~\ref{sec:LinOp} to \ref{sec:stabNL}, and in the last Section 6 we show how our method can be adapted to the spatially inhomogeneous equation. In Section~\ref{sec:LinOp} we introduce a factorization of the (homogeneous) linearized Landau operator $\LL = \AA + \BB$ and prove several properties of the operators $\AA$ and $\BB$. 
Section~\ref{sec:semigroup} is devoted to the proof of (non exponential) decay estimates in large functional spaces of the semigroup associated to $\LL$ (see Theorem~\ref{thm:S_L}) as well as weak
dissipative properties for $\LL$ (see Corollary~\ref{cor:S_L-homBIS}), using the method presented above. 
In Section~\ref{sec:nonL} we prove nonlinear estimates for the Landau operator $Q$, and then in Section~\ref{sec:stabNL} we prove the spatially homogeneous version of Theorem 1.1. Finally, in Section~\ref{sec:nonH}, we deal with the inhomogeneous case and prove Theorem~\ref{thm:stabNL-inhom}, by following the same program as for the homogeneous case above.
% (from Section~\ref{sec:LinOp} to Section~\ref{sec:stabNL}).

\bigskip\noindent
{\bf Acknowledgments.} The first author is supported by the Fondation Math\'ematiques Jacques Hadamard.
The second author gratefully acknowledges the support of the STAB ANR project (ANR-12-BS01-0019).

%%%%%%%%%%%%%%%%%%%%%%%%%%%%
\section{Linearized operator}\label{sec:LinOp}

We define the following quantities
\be\label{eq:bc}
\ba
& b_i(z) = \partial_j a_{ij}(z) = - 2 \, |z|^\gamma \, z_i, \\
& c(z) =  \partial_{ij} a_{ij}(z)  = - 2 (\gamma+3) \, |z|^\gamma \quad \text{if } \gamma \in (-3,-2) , \\
& c(z) =  \partial_{ij} a_{ij}(z) = -8\pi \delta_0 \quad\text{if } \gamma=-3,
\ea
\ee
from which we are able to rewrite the Landau operator~\eqref{eq:oplandau0} into two other forms
\be\label{eq:oplandau}
\ba
Q(g,f) 
& = \partial_i  \{ (a_{ij}*g) \partial_j f - (b_i*g) f \}\\
&= ( a_{ij}*g)\partial_{ij} f - (c*g) f .
\ea
\ee
Consider now the variation $f := F- \mu$ and the linearized (homogeneous) Landau operator 
\be\label{eq:lin-op}
\ba
\LL f &:= Q(\mu,f) + Q(f,\mu) .
%\\ & = (a_{ij}*\mu)\partial_{ij}g - (c*\mu) g + (a_{ij}*g)\partial_{ij} \mu - (c*g) \mu.
\ea
\ee
We denote
\be\label{eq:barabc}
\bar a_{ij} = a_{ij}*\mu, \quad
\bar b_i = b_i * \mu, \quad
\bar c = c*\mu,
\ee
and remark that %when $\gamma \in(-3,-2)$, we have
\bean
\bar c(v) &=& -2(\gamma+3) \int_{v_*} |v-v_*|^\gamma \, \mu_* \quad\hbox{when} \quad \gamma \in(-3,-2), 
\\
\bar c(v) &=& -8\pi \mu(v) \quad\hbox{when} \quad \gamma = - 3.
\eean

\subsection{Known results}
On the space $E_0 := L^2_v(\mu^{-1/2})$, we classically observe that $\LL$ is self-adjoint and verifies $\la \LL f , f \ra_{E_0} \le 0$, so that its spectrum satisfies $\Sigma(\LL) \subset \R_-$. Moreover, thanks to the conservation laws,  there holds
$$
\mathrm{ker}(\LL) = \mathrm{span} \{ \mu, v_1 \mu, v_2 \mu, v_3 \mu, |v|^2 \mu   \}, 
$$
and the projection $\Pi_0$ onto $\mathrm{ker}(\LL)$ is given by
\beqn\label{def:Pi0}
\Pi_0(f) = \left(\int f \, dv \right) \mu + \sum_{j=1}^3 \left( \int v_j f \, dv \right) v_j \mu + \left(  \int \frac{|v|^2 - 3}{6} f \, dv \right) \frac{|v|^2 - 3}{6} \mu .
\eeqn

Several authors have studied weak coercivity estimates for $\LL$ on $E_0$. Summarising results from \cite{DL,BM, Guo, M, MS}, for all $-3 \le \gamma \le 1$,  we have
\be\label{eq:LLsg}
\la   \LL f , f \ra_{E_0}  \lesssim - \| \la v \ra^{\frac{\gamma+2}{2}} \Pi f \|_{E_0}^2 
-  \| \la v \ra^{\frac{\gamma}{2}} \widetilde \nabla_v \Pi ( \mu^{-1/2} f ) \|_{L^2}^2, \quad \forall \, f \in E_0, 
\ee
where we define the projection $\Pi := I - \Pi_0$ onto the orthogonal of $\mathrm{ker}(\LL)$ and we recall that the anisotropic gradient $\widetilde \nabla_v$ has been defined in \eqref{eq:def:anisoG}. 
Observe that \eqref{eq:LLsg} does not provide any spectral gap for the operator $\LL$ in $E_0$ in the very soft and Coulomb potential  case  $-3 \le \gamma < -2$ we are concerned with in the present work, contrarily to the moderately soft and hard potentials case $-2 \le \gamma \le 1$.

\subsection{Factorization of the operator}

Using the form \eqref{eq:oplandau} of the operator $Q$, we decompose the linearized Landau operator as $\LL =  \AA_0 + \BB_0$, where we define
\be\label{eq:A0B0}
\ba
\AA_0 f &:= Q(f,\mu) = \partial_i \{  (a_{ij}* f)\partial_{j}\mu + ( b_i * f)\mu \} 
= (a_{ij}* f)\partial_{ij}\mu - (c * f)\mu, \\
\BB_0 f &:= Q(\mu,f) = \partial_i \{  (a_{ij}* \mu)\partial_{j}f + ( b_i * \mu)f \}
=(a_{ij}* \mu)\partial_{ij}f - (c * \mu)f.
\ea
\ee
Consider a smooth nonnegative function $\chi \in C^\infty_c(\R^3)$ such that $0\leq \chi(v) \leq 1$, $\chi(v) \equiv 1$ for $|v|\leq 1$ and $\chi(v) \equiv 0$ for $|v|>2$. 
For any $R\geq 1$, we define $\chi_R(v) := \chi(R^{-1}v)$. 
Then, we make the final decomposition of the operator $\LL$ as $\LL = \AA + \BB$, with
\be\label{eq:AB}
\AA  := \AA_0 + M \chi_R ,
\qquad
\BB  := \BB_0 - M\chi_R,
\ee
where $M > 0 $ and $R \ge 1$ will be chosen later.

\subsection{Preliminaries}\label{subsec:prelim}
We introduce some convenient classes of weight functions and we
state some preliminaries results that will be useful in the sequel.

\smallskip
We say that a weight function $m :\R^3 \to \R_+$ is admissible if 

\begin{enumerate}[$(i)$]

\item it is a polynomial function, and we write $m = \la v \ra^k$, $k \ge 0$;
 
\item or if it is an exponential function, that is $m = e^{\kappa \la v \ra^s}$ with $\kappa>0$ and $s\in(0,2)$, or with $0< \kappa < 1/2$ and $s=2$.

\end{enumerate}

We denote $\sigma = 0$ when $m= \la v \ra^k$ and $\sigma=s$ when $m = e^{\kappa \la v \ra^s}$. For two admissible weight functions $m_0$ and $m_1$, we write $m_0 \prec m_1$ (or $m_1 \succ m_0$) if  $\lim_{|v| \to \infty} \frac{m_0}{m_1}(v) = 0$. Similarly, we write $m_0 \preceq m_1$ (or $m_1 \succeq m_0$) if $m_0 \prec m_1$ or $m_0 = m_1$  (up to a constant).

\medskip

We finally define the following functions:
\be\label{eq:def-varphi}
\ba
\zeta_{m} (v) &:= \frac{1}{2} \frac{1}{m^2} \partial_{ij}(\bar a_{ij} m^2) - \bar c 
= \bar a_{ij} \frac{\partial_{ij} m }{m} 
+ \bar a_{ij} \frac{\partial_i m}{m} \frac{\partial_j m}{m}
+ 2 \bar b_i  \frac{\partial_i m}{m} 
- \frac12  \bar c ,
\ea
\ee
\be\label{eq:def-tildevarphi}
\ba
\tilde\zeta_{m} (v) &:=  \bar a_{ij} \frac{\partial_i m}{m} \frac{\partial_j m}{m}  +\bar b_i  \frac{\partial_i m}{m} - \frac12  \bar c ,
\ea
\ee
and also 
\be\label{eq:Phimw} 
\zeta_{m,\omega}(v) := \bar a_{ij} \frac{\partial_{ij} \omega}{\omega} 
+ \bar a_{ij} \frac{\partial_{i} \omega}{\omega} \frac{\partial_j \omega}{\omega} 
- 2\bar a_{ij} \frac{\partial_{i} \omega}{\omega} \frac{\partial_j m}{m}.
\ee

\smallskip  

We start stating some estimates on the matrix $\bar a_{ij}$. To that purpose, we define 
 $$
\ba
\ell_1(v) &= \int_{\R^3} \left(1 - \left(\frac{v}{|v|}\cdot\frac{w}{|w|}   \right)^2   \right) |w|^{\gamma+2} \mu(v-w)\, dw, \\
\ell_2(v) &= \int_{\R^3} \left(1 - \frac12 \left| \frac{v}{|v|}\times\frac{w}{|w|}  \right|^2   \right) |w|^{\gamma+2} \mu(v-w)\, dw, 
\ea
$$ 
where $\times$ stands for the vector product in $\R^3$, and, for $-3 <  \beta <0$, we define
$$
J_\beta(v) := \int_{\R^3} |v-w|^\beta \mu(w)\, dw. 
$$

\begin{lem}\label{lem:bar-aij}
The following properties hold:

\begin{enumerate}[(a)]

\item The matrix $\bar a(v)$ has a simple eigenvalue $\ell_1(v)>0$ associated with the eigenvector $v$ and a double eigenvalue $\ell_2(v)>0$ associated with the eigenspace $v^{\perp}$. 
Moreover, when $|v|\to +\infty$,  we have
$$
\ba
\ell_1(v) \sim  2 \la v \ra^\gamma, \quad
\ell_2(v) \sim  \la v \ra^{\gamma+2} .
\ea
$$

\item The function $\bar a_{ij}$ is smooth, more precisely for any multi-index $\beta\in \N^3$, 
$$
|\partial^\beta \bar a_{ij}(v)| 
\leq C_\beta \la v \ra^{\gamma+2-|\beta|}. 
$$
Moreover, there exists a constant $K > 0$ such that 
$$
\ba
\bar a_{ij}(v) \xi_i \xi_j &= \ell_1(v) |P_v \xi|^2 + \ell_2(v)|(I-P_v)\xi|^2 \\
&\ge K \{  \la v \ra^{\gamma} |P_v \xi|^2 + \la v \ra^{\gamma+2} |(I-P_v)\xi|^2  \}.
\ea
$$

\item We have
$$
 \mathrm{tr} (\bar a (v) )  =\ell_1 (v)  + 2 \ell_2 (v)
= 2  J_{\gamma+2} (v)
\qquad\text{and}\qquad
\bar b_i(v) = - \ell_1(v)\, v_i.
$$

\item If $|v|>1$, we have
$$
|\partial^\beta \ell_{1}(v)|\leq C_\beta \la v \ra^{\gamma-|\beta|}\qquad\text{and}\qquad |\partial^\beta \ell_{2}(v)|\leq C_\beta \la v \ra^{\gamma+2-|\beta|}.
$$

\item For any $\beta \in (-3,0)$, there exists some constant $C_\beta>0$ such that 
$$
|J_\beta(v) - \langle v \rangle^\beta | \le C_\beta \la v \ra^{\beta-1/2}, \quad \forall \, v \in \R^3.
$$

\end{enumerate}

\end{lem}

\begin{proof}
Item (a) comes from \cite[Propositions 2.3 and 2.4]{DL}, (b) is \cite[Lemma 3]{Guo}, (c) is evident and (d) is \cite[Lemma 2.3]{CTW}.

We just then present the proof of (e). 
On the one hand, for any $v \in \R^3$, we have
\bear\label{eq:bdJa1}
J_{\beta}(v) 
&=& \int_{|v_*|\le 1} |v_*|^{\beta} \mu(v_*- v)\, dv_* + \int_{|v_*|\ge 1} |v_*|^{\beta} \mu(v-v_*)\, dv_* \\
\nonumber
&\leq& \sup_{|v_*|\le 1} \mu(v-v_*) \int_{|v_*| \le 1} |v_*|^{\beta} \, dv_* 
+  \int_{|v_*|\ge 1}   \mu(v-v_*)\, dv_* \le C_1,
\eear
since the two terms are clearly bounded uniformly in $v \in \R^3$. 

\smallskip
On the other hand, for any $v \in \R^3$, $|v| \ge 1$, and for any $R > 0$, we write 
$$
\ba
J_{\beta}(v) &= \int_{|v_*|\le R} |v_*- v|^{\beta} \mu(v_*)\, dv_* + \int_{|v_*|\ge R} |v_*-v|^{\beta} \mu(v_*)\, dv_* = T_1 + T_2. 
\ea
$$
For the second term,  we have   
$$
|T_2| \le \sqrt{\mu(R)} \int_{|v_*|\ge R} |v_*-v|^{\beta} \sqrt{\mu(v_*)}\, dv_* \le C_2 \, e^{-R^2/4},
$$
where we have used an estimate very similar to \eqref{eq:bdJa1} in order to bound the integral term. 
For the first term and for $|v| > R$, we have   
\bean
T_1 
&\ge& \int_{|v_*|\le R}( |v| +|v_*|)^{\beta} \mu(v_*)\, dv_*
\\
&\ge& \int_{|v_*|\le R}( |v| +R)^{\beta} \mu(v_*)\, dv_*
\ge  ( |v| +R)^{\beta} ( 1- C_3 e^{-R^2/4}),
\eean
and in a similar way, we have 
$$
T_1 \le   | |v| -R |^{\beta} .
$$
We conclude by making the choice $R := |v|^{1/2}$. 
\end{proof}

\begin{lem}\label{lem:varphi}
Let $m$ be an admissible weight function such that $m \succ \la v \ra^{(\gamma+3)/2}$.

\smallskip\noindent
(1) If $\sigma = 0$ and $\omega = \la v \ra^\alpha$ is a polynomial weight function such that $\omega \prec m \la v \ra^{-(\gamma+3)/2}$, then
$$
\ba
\limsup_{|v|\to \infty} \zeta_{m}(v) \la v \ra^{-\gamma} =  
\limsup_{|v|\to \infty} \tilde\zeta_{m}(v) \la v \ra^{-\gamma} &\le 2 \{ (\gamma+3)/2 - k \} ,\\
\limsup_{|v|\to \infty} \left[ \tilde\zeta_{m}(v) + \zeta_{m,\omega}(v) \right] \la v \ra^{-\gamma} &\le 2 \{ (\gamma+3)/2 + \alpha - k \} .
\ea
$$

\smallskip\noindent
(2) If $\sigma \in (0,2)$, then
$$
\ba
 \limsup_{|v|\to \infty} \zeta_{m}(v) \la v \ra^{-\sigma-\gamma} = 
 \limsup_{|v|\to \infty} \tilde\zeta_{m}(v) \la v \ra^{-\sigma-\gamma} &\le  -2 \kappa s.
\ea
$$

\smallskip\noindent
(3) If $\sigma = 2$, then
$$
\ba
 \limsup_{|v|\to + \infty} \zeta_{m}(v) \la v \ra^{-2-\gamma} &\le  4\kappa (4\kappa-1) ,\\
\limsup_{|v|\to + \infty} \tilde\zeta_{m}(v) \la v \ra^{-2-\gamma}  &\le 4\kappa(2\kappa-1) .
\ea
$$
\end{lem}

\begin{proof}
We introduce the notation
$$
\tilde J_\gamma (v) =  
\left\{
\ba
& (\gamma+3) J_\gamma(v) \quad &\text{if } \gamma \in (-3,-2),\\ 
& 4 \pi \mu(v) \quad &\text{if } \gamma=-3, 
\ea
\right.
$$
so that $ \bar c = - 2 \tilde J_\gamma$. 
We observe from Lemma~\ref{lem:bar-aij} that, when $|v| \to + \infty$, we have 
\beqn\label{eq:AsymBehavEll&J}
{1 \over 2} \ell_1(v) \sim  
\ell_2(v) |v|^{-2} \sim \la v \ra^{\gamma}
\quad\hbox{and}\quad
 \tilde J_\gamma (v) = (\gamma+3) \, \langle v \rangle^{\gamma} +  \OO(\langle v \rangle^{\gamma-1/2} ).
\eeqn

\smallskip
{\noindent \it Step 1. Polynomial weight.}
Consider $m= \la v \ra^k$. From definition \eqref{eq:bc}-\eqref{eq:barabc} and Lemma~\ref{lem:bar-aij},
 we obtain
$$
\ba
\bar a_{ij} \,\frac{\partial_{ij} m}{m} 
&= (\delta_{ij}\bar a_{ij}) \,  k \la v \ra^{-2} +  (\bar a_{ij} v_i v_j)  \, k(k-2)\la v \ra^{-4} \\
&= 2\ell_2(v) \,  k \la v \ra^{-2} + \ell_1(v) \,  k \la v \ra^{-2} +  \ell_1(v)  \, k(k-2) |v|^2 \la v \ra^{-4},
\ea
$$
Moreover,
$$
\ba
\bar a_{ij} \,\frac{\partial_i m}{m}\, \frac{\partial_j m}{m} = (\bar a_{ij} v_i v_j)  \, k^2\la v \ra^{-4}
=   \ell_1(v)  \, k^2 |v|^2 \la v \ra^{-4},
\ea
$$
and also, using the fact that $\bar b_{i}(v) = - \ell_1(v) v_i$ from Lemma~\ref{lem:bar-aij},
$$
\ba
\bar b_{i} \,\frac{\partial_i m}{m} 
=   - \ell_1(v)  \, k |v|^2 \la v \ra^{-2}.
\ea
$$
It follows  that 
$$
\ba
\zeta_{m}(v)
&=
  2 k \ell_2(v) \la v \ra^{-2} + k \ell_1(v) \la v \ra^{-2}
+  k(k-2) \, \ell_1(v) \, |v|^2 \la v \ra^{-4} \\
&\quad
+  k^2  \, \ell_1(v) \, |v|^2 \la v \ra^{-4} 
 - 2   k \, \ell_1(v) \, |v|^2 \la v \ra^{-2}
+\tilde  J_{\gamma}(v),
\ea
$$
as well as
$$
\tilde \zeta_m (v) = k^2 \ell_1(v) |v|^2 \la v \ra^{-4} - k \ell_1(v) |v|^2 \la v \ra^{-2} + \tilde J_{\gamma}(v).
$$
Thanks to \eqref{eq:AsymBehavEll&J}, the dominant terms are of order $\la v \ra^{\gamma}$.
We then obtain
$$
{ \limsup_{|v|\to + \infty}  \zeta_{m}(v) \la v \ra^{-\gamma} }
= \limsup_{|v|\to + \infty} \tilde \zeta_{m}(v) \la v \ra^{-\gamma} 
\le  2 \{  (\gamma+3)/2 - k \},
$$
from which we conclude the proof of the first part of point (1). The estimate of $\zeta_{m,\omega}$ is similar as above, and thus we omit it.

\medskip
{\noindent \it Step 2. Exponential weight.}
For $m=e^{\kappa \la v \ra^s}$, we have 
$$
\ba
\zeta_{m}(v) &=
2 \kappa s \,  \ell_2(v)  \la v\ra^{s-2}
+  \kappa s \,  \ell_1(v)  \la v\ra^{s-2}
+ \kappa s(s-2) \,  \ell_1(v)   |v|^2\la v\ra^{s-4}\\
&\quad
+ 2 \kappa^2 s^2 \,  \ell_1(v)   |v|^2\la v\ra^{2s-4}
 - 2 \kappa s \,  \ell_1(v) |v|^2\la v\ra^{s-2}
+ \tilde J_{\gamma}(v)
\ea
$$
and also
$$
\tilde\zeta_{m} (v) 
 =   - \kappa s \ell_1(v)|v|^2 \la v \ra^{s-2} + \kappa^2 s^2 \ell_1(v) |v|^2 \la v \ra^{2s-4} 
+ \tilde J_\gamma(v).
$$
In any cases $0< s \le2$,   the dominant terms are of order $\la v \ra^{\gamma+s}$, and we easily conclude. 
\end{proof}

We conclude this section with a remark about the weighted spaces we have defined in \eqref{eq:defLpm}. For any admissible weight function $m$ we easily obtain
\be\label{eq:Equiv-H1(m)}
\| \la v \ra^{(\sigma - 1)_+}\, m f  \|_{L^2}^2 + \| \nabla_v (mf) \|_{L^2}^2 \sim 
\|  m f \|_{L^2}^2 + \| m \nabla_v f \|_{L^2}^2,
\ee
so that in particular $\| f \|_{H^1(m)}^2 \sim \| f \|_{L^2(m)}^2 + \| \nabla_v f \|_{L^2(m)}^2 $ when $\sigma \in [0,1]$.

%%%%%%%%%%%%%%%%%%%%%%%%
\subsection{Dissipative properties of $\BB$}

We prove in this section weakly dissipative properties for the operator $\BB$. These estimates are similar to the estimates established in \cite{KC2,CTW} for $-2 \le \gamma \le 1$, in which case it is proven that the operator $\BB-\alpha$ is  dissipative for some $\alpha < 0$. 
%For the sake of clarity and completeness, we shall give the proof in detail of the new estimates in our case in Lemmas~\ref{lem:BB} and \ref{lem:BB2}.

\begin{lem}\label{lem:BB}
Let $m$ be an admissible weight function such that $m \succ \la v \ra^{(\gamma+3)/2}$ and 
we recall that we have defined $\sigma=0$ when $m$ is polynomial and $\sigma=s$ when $m$ is exponential. 
There exist $M,R>0$ large enough such that $\BB$ is weakly dissipative in $L^2(m)$ in the sense:

\medskip

$\bullet$ If $m \prec \mu^{-1/2}$, there holds
\be\label{eq:BL2}
 \la \BB f , f \ra_{L^2(m)} \lesssim 
-   \| \la v \ra^{\frac{\gamma}{2}} \,   \widetilde \nabla_v  f \|_{L^2(m)}^2
-   \| \la v \ra^{\frac{\gamma}{2}} \,   \widetilde \nabla_v (m f) \|_{L^2}^2
-  \| \la v \ra^{\frac{\gamma+\sigma}{2}} f  \|_{L^2(m)}^2.
\ee

\medskip

$\bullet$ If $\mu^{-1/2} \preceq m \prec \mu^{-1}$, there holds
\beqn\label{eq:BL2bis}
 \la \BB f , f \ra_{L^2(m)} \lesssim 
-  \| \la v \ra^{\frac{\gamma}{2}} \,   \widetilde \nabla_v (m f) \|_{L^2}^2
-   \| \la v \ra^{\frac{\gamma+\sigma}{2}} f  \|_{L^2(m)}^2,
\ee

\end{lem}

\begin{proof}
From the definition \eqref{eq:A0B0}-\eqref{eq:AB} of $\BB$, we have 
$$
\ba
\int (\BB f) \, f \, m^2 &=  
 \int \bar a_{ij} \partial_{ij} f  \, f \, m^2
-\int \bar c \, f^2 \, m^2  
-\int M\chi_R \, f^2 \, m^2 \\
&=: T_1 + T_2 + T_3 .
\ea
$$
Let us compute the term $T_1$. 
Writing $g = mf$ and thus $\partial_{ij} f \, f m^2 = \partial_{ij} (m^{-1} g) \, g m$, an integration by parts yields
$$
T_1 = - \int  \left\{ \bar b_j g m + \bar a_{ij} \partial_i g m + \bar a_{ij} g \partial_i m   \right\} \partial_j (m^{-1} g).
$$
Using that $\partial_j (m^{-1} g ) = m^{-1} \partial_j g - m^{-2} \partial_j m g$ in the last equation, we first get
$$
T_1 = - \int \bar a_{ij} \partial_i g \partial_j g
+ \int \Big\{ \bar a_{ij}\frac{\partial_i m}{m}\frac{\partial_j m}{m} + \bar b_j \frac{\partial_j m}{m} \Big\} g^2 - \int \bar b_j g \partial_j g ,
$$
and thanks to another integration by parts for the last term, we finally obtain
$$
\int (\BB f) \, f \, m^2
=  -  \int  \bar a_{ij} \partial_i (mf) \partial_j (mf)  
+\int \{  \tilde\zeta_{m} - M\chi_R \}  f^2 \, m^2.
$$
In a similar (and even simpler) way,  we can also obtain
$$
\int (\BB f) \, f \, m^2
=  -  \int  \bar a_{ij} \partial_i f \partial_j f \, m^2  
+\int \{ \zeta_{m} - M\chi_R \}  f^2 \, m^2.
$$
Thanks to Lemma~\ref{lem:varphi}, we may choose $M,R>0$ large enough such that
$$
\zeta_{m}(v) - M\chi_R(v) \lesssim - \la v \ra^{\gamma+\sigma} , \quad
\tilde \zeta_{m}(v) - M\chi_R(v) \lesssim - \la v \ra^{\gamma+\sigma}, \quad \text{if } m \prec \mu^{-1/2},
$$
and
$$
 \tilde\zeta_{m}(v) - M\chi_R(v) \lesssim - \la v \ra^{\gamma+\sigma}, \quad \text{if } \mu^{-1/2} \preceq m \prec \mu^{-1},
$$
and  we then conclude using the coercivity of $\bar a_{ij}$ from Lemma~\ref{lem:bar-aij}. 
\end{proof}

For any admissible weight function $m$, we define the operator $\BB_m g = m \BB(m^{-1} g)$, which writes
\be\label{eq:BBm}
\ba
\BB_m g &= \bar a_{ij} \partial_{ij} g -2\bar a_{ij}\frac{\partial_i m}{m}\, \partial_j g 
+ \left\{ 2 \bar a_{ij}\frac{\partial_i m}{m} \frac{\partial_j m}{m} - \bar a_{ij}\frac{\partial_{ij} m}{m}  -\bar c  - M \chi_R \right\} g \\
&= : \bar a_{ij} \partial_{ij} g + \beta_j \partial_j g + (\delta - M\chi_R) g.
\ea
\ee
We then define its formal adjoint operator $\BB^*_m$ that verifies
\be\label{def:B*m}
\BB^*_m \phi = \bar a_{ij} \partial_{ij} \phi + 2\Big\{ \bar a_{ij}\frac{\partial_i m}{m} + \bar b_j  \Big\} \partial_j \phi + \Big\{ \bar a_{ij}\frac{\partial_{ij} m}{m} + 2 \bar b_i \frac{\partial_i m}{m}  - M \chi_R \Big\} \phi.
\ee
Observe that if $f$ satisfies the equation $\partial_t f = \BB f$ then $g= mf$ satisfies $\partial_t g = \BB_m g$, and also that $\la \BB f , f \ra_{L^2(m)} = \la \BB_m g , g \ra_{L^2}$.  Moreover there holds by duality
$$
\forall \, t \ge 0, \quad
\la S_{\BB_m}(t) g , \phi \ra_{L^2} = \la g , S_{\BB^*_m}(t) \phi \ra_{L^2},
$$
where we recall that $S_{\BB_m}(t)$ is the semigroup generated by $\BB_m$ and $S_{\BB^*_m}(t)$ the semigroup generated by $\BB^*_m$.

\medskip

We now prove weakly dissipative properties of the adjoint $\BB_m^*$.    Here, we restrict ourselves to the case of a %(low degree) 
polynomial weight function in order to simplify the presentation and because it will be sufficient for our purpose. Indeed, the final estimates we will deduce of the analysis we are starting here will be used on ``perturbation terms" and we will not destroy the possible faster rate of decay we get for stronger weight functions. 
%perturbed the term and a somewhat low (polynomial) decay will not perturbated the possibly

\begin{lem}\label{lem:BB2}
Let $m$ and $\omega$ be two admissible polynomial weight functions such that $m \succ \langle v \rangle^{(\gamma+3)/2}$ and $1 \preceq \omega \prec m \, \la v \ra^{-(\gamma+3)/2}$.  

\smallskip
(1) We can choose $M,R > 0$, large enough, such that $\BB^*_m$ is weakly dissipative in $L^2(\omega)$ in the following sense:
\be\label{eq:B*mL2}
\la \BB^*_m \phi, \phi \ra_{L^2(\omega)} \lesssim -  \| \phi \|_{L^2(\omega \la v \ra^{\gamma/2})}^2 
-  \|  \widetilde \nabla_v \phi \|_{L^2(\omega \la v \ra^{\gamma/2})}^2.
\ee

\smallskip
(2) For any $\eta > 0$, we define the equivalent norm  $\| \cdot \|_{\tilde H^1(\omega)}$ on $H^1(\omega)$, and the associated scalar product $\la \cdot, \cdot \ra_{\tilde H^1(\omega)}$, by 
$$
\| \phi \|_{\tilde H^1(\omega)}^2 := \| \phi \|_{L^2(\omega)}^2 + \eta \| \nabla_v \phi \|_{L^2(\omega)}^2.
$$
We can choose $M,R,\eta > 0$, such that $\BB^*_m$ is weakly dissipative in $H^1(\omega)$ in the following sense:
\be\label{eq:B*mH1}
\la \BB^*_m \phi, \phi \ra_{\tilde H^1(\omega)} \lesssim 
-  \|  \phi \|_{\tilde H^1(\omega \la v \ra^{\gamma/2})}^2 
-  \|  \widetilde \nabla_v \phi \|_{L^2(\omega \la v \ra^{\gamma/2})}^2
- \eta  \|  \widetilde \nabla_v (\nabla_v \phi) \|_{L^2(\omega \la v \ra^{\gamma/2})}^2.
\ee

\end{lem}

\begin{proof}
We split the proof into three steps. In what follows we shall use the equivalence \eqref{eq:Equiv-H1(m)} since $\omega$ is a polynomial weight function.

\medskip\noindent
\textit{Step 1.} We have
$$
\begin{aligned}
 \int (\BB_m^* \phi) \, \phi \, \omega^2
&= \int \left( \bar a_{ij} \frac{\partial_{ij} m}{m}  + 2 \bar b_j \frac{\partial_j m}{m}  - M \, \chi_R\right) \, \phi^2 \, \omega^2\\
&\quad +  \int \left(  \bar a_{ij} \frac{\partial_j m}{m} + \bar b_i \right) \partial_i (\phi^2)  \, \omega^2   + \int \bar a_{ij} \partial_{ij} \phi \, \phi   \, \omega^2\\
&=: I_1 + I_2 + I_3 . 
\end{aligned} 
$$
Performing one integration by parts, we obtain
$$
\ba
I_2 &= -\int \partial_i \Big( \bar a_{ij} \frac{\partial_j m}{m}  + \bar b_i    \Big) \, \phi^2 \, \omega^2
- \int \Big( \bar a_{ij} \frac{\partial_j m}{m}  + \bar b_i    \Big) 2 \omega \partial_i \omega \, \phi^2 \\
&=\int \bigg\{ - \bar a_{ij} \frac{\partial_{ij} m}{m} + \bar a_{ij} \frac{\partial_i m}{m} \frac{\partial_j m}{m}  - \bar b_j \frac{\partial_j m}{m}   -\bar c \bigg\} \, \phi^2\, \omega^2 \\
&\quad
- \int  2 \bigg\{  \bar a_{ij} \frac{\partial_j m}{m} \frac{\partial_i \omega}{\omega} + \bar b_i \frac{\partial_j \omega}{\omega}  \bigg\} \,  \phi^2 \, \omega^2.
\ea
$$
Using that $ \partial_{ij} \phi \, \phi = \frac{1}{2} \partial_{ij}(\phi^2) - \partial_i \phi \partial_j \phi $, it follows
$$
\ba
I_3
& = - \int \bar a_{ij} \partial_i \phi \partial_j \phi \, \omega^2  + \frac12 \int \partial_{ij}( \bar a_{ij} \omega^2)  \phi^2 \\
& = - \int \bar a_{ij} \partial_i \phi \partial_j \phi \, \omega^2  + \frac12 \int \partial_{i}( \bar b_{i} \omega^2 + \bar a_{ij} 2 \omega \partial_j \omega)  \phi^2 \\
& = - \int \bar a_{ij} \partial_i \phi \partial_j \phi \, \omega^2  + \frac12 \int \bigg\{ \bar c  + 4 \bar  b_i \frac{\partial_i \omega}{\omega}  + 2 \bar a_{ij} \frac{\partial_i \omega}{\omega} \frac{\partial_j \omega}{\omega}   + 2 \bar a_{ij}  \frac{\partial_{ij} \omega}{\omega} \bigg\}  \,  \phi^2 \, \omega^2.
\ea
$$
Finally, we get
\be\label{eq:B*m}
\ba
\int (\BB_m^* \phi) \, \phi \, \omega^2 &=
- \int \bar a_{ij} \partial_i \phi \partial_j \phi \, \omega^2 +
\int \{ \tilde\zeta_{m} + \zeta_{m,\omega} - M \chi_R \} \, \phi^2 \, \omega^2\\
&\lesssim - \| \la v \ra^{\frac{\gamma}{2}} \widetilde \nabla_v \phi \|_{L^2(\omega)}^2 -  \| \la v \ra^{\frac{\gamma}{2}} \, \phi \|_{L^2(\omega)}^2 
\ea
\ee
by choosing $M,R>0$ large enough and using that $ \tilde\zeta_{m}(v) + \zeta_{m,\omega}(v)-M \chi_R(v) \lesssim - \la v \ra^{\gamma}$ thanks to Lemma~\ref{lem:varphi}. That completes the proof of point (1).

\medskip\noindent
\textit{Step 2.} Now, we introduce the notation $\phi_\alpha := \partial^\alpha_v \phi$ where $\alpha \in \N^3$ and $|\alpha|=1$. There holds 
$$
\ba
\partial^\alpha_v (\BB_m^* \phi) &= \BB_m^* \phi_\alpha + 
\partial^\alpha_v \left\{ \bar a_{ij} \frac{\partial_{ij} m}{m} + 2 \bar b_j \frac{\partial_j m}{m}  - M \, \chi_R \right\} \phi  
+ 2 \partial^\alpha_v \left\{  \bar a_{ij} \frac{\partial_j m}{m} + \bar b_i  \right\} \partial_i \phi \\
&\quad  + \partial^\alpha_v \bar a_{ij} \partial_{ij} \phi,
\ea
$$
which implies that 
$$
\ba
\int \partial^\alpha_v (\BB_m^* \phi) \, \phi_\alpha \, \omega^2
&= \int (\BB_m^*\phi_\alpha )\, \phi_\alpha \, \omega^2
+  \int \partial^\alpha_v \left\{ \bar a_{ij} \frac{\partial_{ij} m}{m} + 2 \bar b_j \frac{\partial_j m}{m}  - M \, \chi_R \right\} \phi \, \phi_\alpha \, \omega^2 \\
&\quad +  2\int  \partial^\alpha_v \left\{ \bar a_{ij} \frac{\partial_j m}{m} + \bar b_i \right\} \partial_i \phi \, \phi_\alpha \, \omega^2
 +  \int (\partial^\alpha_v \bar a_{ij}) (\partial_{ij} \phi) \, \phi_\alpha \, \omega^2\\
&=: T_1+T_2+T_3+T_4. 
\ea
$$
Using Step 1 of the proof, we have, for some constant $\lambda>0$,
$$
T_1 \le - \lambda \|  \la v \ra^{\frac{\gamma}{2}} \widetilde \nabla_v \phi_\alpha \|_{L^2(\omega)}^2  + \int \{ \tilde\zeta_{m} + \zeta_{m,\omega}-M \chi_R \} \, \phi_\alpha^2 \, \omega^2. 
$$
For the term $T_2$, we have straightforwardly from Lemma~\ref{lem:bar-aij}
$$
T_2 \lesssim \int \la v \ra^{\gamma-1} |\phi| \, |\nabla_v \phi | \, \omega^2 \le \| \phi \|_{L^2(\omega \, \la v \ra^{(\gamma-1)/2})}
 \| \nabla_v\phi \|_{L^2(\omega \, \la v \ra^{(\gamma-1)/2})}, 
$$
and similarly 
$$
T_3 \lesssim \int \la v \ra^{\gamma}  |\nabla_v \phi |^2 \, \omega^2 = \| \nabla_v\phi \|_{L^2(\omega \, \la v \ra^{\gamma/2})}^2.
$$
For the last term, we use one first integration by part, in order to get 
\bean
T_4 &=&
 - \int (\partial^\alpha_v \bar b_{i}) (\partial_i \phi) \, \phi_\alpha \, \omega^2
- \int (\partial^\alpha_v \bar a_{ij}) (\partial_i \phi) \, \partial_j \phi_\alpha \, \omega^2
\\ 
 &&-\int (\partial^\alpha_v \bar a_{ij}) (\partial_{i} \phi) \, \phi_\alpha \,\partial_j  \omega^2 = U_1 + U_2 + U_3.
 \eean
In the above expression, the first term and last term can be bounded exactly as $T_3$. For the middle term, we perform one more integration with respect to the $\partial_\alpha$ derivative, and we get 
$$
U_2 = \int (\Delta_v \bar a_{ij}) \partial_i \phi \, \partial_j \phi \, \omega^2
+ \int (\partial^\alpha_v \bar a_{ij}) (\partial_i \phi_\alpha) \, \partial_j \phi \, \omega^2
+ \int (\partial^\alpha_v \bar a_{ij}) (\partial_i \phi) \, \partial_j \phi \, { \partial^\alpha_v }\omega^2.
$$
We recognize the middle term as $-U_2$, from what we deduce 
\bean
U_2 &=& \frac12 \int (\Delta_v \bar a_{ij}) \partial_i \phi \, \partial_j \phi \, \omega^2
+ \frac12 \int (\partial^\alpha_v \bar a_{ij}) (\partial_i \phi) \, \partial_j \phi \, { \partial^\alpha_v }\omega^2
\\
&\lesssim& \| \nabla_v \phi \|^2_{L^2(\omega \la v \ra^{\gamma/2})}.
\eean
All the estimates together, we have established, { for some constants $\lambda,C>0$},
\beqn\label{eq:nablaB*m}
\la \nabla_v (\BB^*_m \phi) , \nabla_v \phi \ra_{L^2(\omega)} 
\le - \lambda \|  \widetilde \nabla_v (\nabla_v \phi) \|_{L^2(\omega  \la v \ra^{\gamma/2} )}^2 
+ C \|   \phi  \|_{H^1(\omega  \la v \ra^{\gamma/2} )}^2.
\eeqn

\medskip\noindent
\textit{Step 3.} 
We gather estimates \eqref{eq:B*m} and \eqref{eq:nablaB*m}, 
we observe that 
$$
 \|   \phi  \|^2_{H^1(\omega  \la v \ra^{\gamma/2} )} \lesssim  \|   \phi  \|^2_{L^2(\omega  \la v \ra^{\gamma/2} )}
 +  \|   \widetilde \nabla_v \phi  \|^2_{L^2(\omega  \la v \ra^{\gamma/2} )}
 $$ and we conclude choosing $\eta >0$ small enough.
\end{proof}

\subsection{Estimates on the operator $\AA$}
We prove boundedness properties for the operator $\AA$.

\begin{lem}\label{lem:AA}
For any $\theta\in(0,1)$, $\ell = 0,1$ and $p \in [1,\infty]$, there holds $ \AA \in \BBB (W^{\ell,p} , W^{\ell,p}(\mu^{-\theta}) )$. 
\end{lem}

\begin{proof}
We only prove the case $\ell = 0$, the case $\ell = 1$ being similar. We only investigate $\AA_0$ since $\AA = \AA_0 + M\chi_R$, and we recall that $\AA_0 g = (a_{ij}*g)\partial_{ij}\mu + (c*g)\mu$.
 We decompose $a$ and $c$ into a bounded part and a singular part. More precisely, we split  $a_{ij}(z) = a_{ij}(z) {\mathbf 1}_{|z|>1} + a_{ij}(z) {\mathbf 1}_{|z| \le 1} =: a_{ij}^+(z) + a_{ij}^-(z)$, and similarly for $c(z)$.

\smallskip
Assume first $\gamma \in (-3,-2)$. For the bounded parts $a^+$ and $c^+$, we easily have 
$$
|(a_{ij}^+*g)(v)| + |(c^+*g)(v)| \lesssim \| g \|_{L^1},
$$
and therefore  
$$
\| (a_{ij}^+ * g) { \partial_{ij}\mu }\|_{L^p(\mu^{-\theta})}
+\| (c^+ * g) \mu \|_{L^p(\mu^{-\theta})} \lesssim \| g \|_{L^1}.
$$
We now turn to the singular terms. We first have
$$
\ba
\| (a_{ij}^- * g) { \partial_{ij}\mu } \|_{L^1(\mu^{-\theta})}
 \lesssim \int_{v_*} |g(v_*)|  \left( \int_{v} |v-v_*|^{(\gamma+2)} \, {\mathbf 1}_{|v-v_*| \le 1} \, \mu^{1-\theta}(v)  \right)  
\lesssim \| g \|_{L^1}
\ea
$$
and similarly, 
$$
\ba
\| (c^- * g) \mu \|_{L^{1}(\mu^{-\theta})} 
 \lesssim \int_{v_*} |g(v_*)|  \left( \int_{v} |v-v_*|^{\gamma } \, {\mathbf 1}_{|v-v_*| \le 1} \, \mu^{1-\theta}(v)  \right) 
\lesssim \| g \|_{L^{1}}.
\ea
$$
As a consequence, we already obtain that $\AA$ is a bounded operator from $L^1 \to L^1(\mu^{-\theta})$. Moreover, we can estimate
$$
|(a_{ij}^- * g)(v)| 
\lesssim \| g \|_{L^\infty} \left(\int |v-v_*|^{(\gamma+2)} \, {\mathbf 1}_{|v-v_*| \le 1} \, dv_* \right)    
\lesssim \| g \|_{L^\infty}
$$
and in a similar way
$$
|(c^- * g)(v)| 
\lesssim \| g \|_{L^\infty} \left(\int |v-v_*|^{\gamma } \, {\mathbf 1}_{|v-v_*| \le 1} \, dv_* \right)  \\
\lesssim \| g \|_{L^\infty},
$$
which imply
$$
\| (a_{ij}^- * g) { \partial_{ij}\mu } \|_{L^\infty(\mu^{-\theta})} \lesssim \| g \|_{L^\infty}, \quad
\| (c^- * g) \mu \|_{L^\infty(\mu^{-\theta})} \lesssim \| g \|_{L^\infty}.
$$
These estimates prove that $\AA$ is bounded from $L^\infty \to L^\infty(\mu^{-\theta})$. We can then conclude to the boundedness of $\AA$ for any $p\in[1,\infty]$ by Riesz-Thorin interpolation theorem. 

\smallskip
Assume now $\gamma=-3$. In that case the term $(a_{ij}*g){ \partial_{ij}\mu }$ can be treated exactly in the same way as above, but now we have $c= - \delta_0$ and then $c*g = - g$. Therefore, for any $p \in [1,\infty]$,
$$
\| (c*g) \mu \|_{L^p(\mu^{-\theta})} = \| g \mu^{1-\theta} \|_{L^p} \lesssim \| g \|_{L^p},
$$
which completes the proof.
\end{proof}

%%%%%%%%%%%%%%%%%%%%%%%%%%%%%%%%%
\section{Semigroup decay}\label{sec:semigroup}

This section is devoted to the proof of decay and regularity estimates for the linearized semigroup $S_\LL$. 
Given two admissible weight functions $m_0 \prec m_1$, we define
%$$
%\Theta_{m_1,m_0}(t) = {(\log (1+t))^{ 2 (k_1-k_0)/|\gamma|} \over \la t \ra^{ (k_1-k_0)/|\gamma|}},  \quad\text{if } m_1= \la v \ra^{k_1} \text{ and }  m_0 = \la v \ra^{k_0},
%$$
$$
\Theta_{m_1,m_0}(t) =  \la t \ra^{ - \frac{(k_1-k_*)}{|\gamma|}}, \, \text{for any } k_* \in (k_0,k_1), \quad\text{if } m_1= \la v \ra^{k_1} \text{ and }  m_0 = \la v \ra^{k_0},
$$
and
$$
\Theta_{m_1,m_0}(t) = e^{ - \lambda \, t^{ \frac{s}{|\gamma|} }}, \ \text{for some} \ \lambda >0, \quad\text{if } m_1 = e^{\kappa \la v \ra^s} .
$$
 
In order to avoid misleading, it is worth emphasizing that  when $m_1$ is a polynomial weight, $\Theta_{m_1,m_0}$ refers to a class of functions, whereas for $m_1$ an exponential weight, $\Theta_{m_1,m_0}$ stands for a fixed function. That somehow usual convention greatly shorten notations and simplify the exposition. As a consequence, we also emphasize that  in both cases, for any $0 < s < t$, we have
$$
\Theta_{m_1,m_0}^{-1} (t) \lesssim \Theta_{m_1,m_0}^{-1} (t-s) \, \Theta_{m_1,m_0}^{-1} (s).
$$

Here and below, we define the time convolution product $S_1 * S_2$ of two functions $S_i$ defined on the half real line $\R_+$ by 
$$
(S_1 * S_2)(t) = \int_0^t S_1(t-s)  S_2(s) \, ds,
$$
and we also define $S^{0} = I$ and $S^{(*n)} = S * S^{(*(n-1))}$ for any $n \ge 1$.

%--------------- Decay estimates for $S_\BB$ -------------------------------
\subsection{Decay estimates for $S_\BB$}

We first prove decay estimates for the semigroup $S_\BB$.

\smallskip
For any admissible weight function $m$, we define the space $H^1_{*}(m)$ associated to the norm
\be\label{def:H1*}
 \| f \|_{H^1_{*}(m)}^2 := \| f \|_{L^2(m \la v \ra^{(\gamma+\sigma)/2})}^2 + \| \widetilde \nabla_v ( mf ) \|_{L^2(\la v \ra^{\gamma/2})}^2,
\ee
and we easily observe that $ H^1_*(m \la v \ra^{|\gamma|/2}) \subset H^1(m) \subset H^1_*( m)$.  When furthermore $m$ is a polynomial weight function,    we define the negative Sobolev space $H^{-1}_{*} (m)$ in duality with $H^1_{*}(m)$ with respect to the duality product on $L^2(m)$, more precisely
\be\label{def:H-1*}
\| f \|_{H^{-1}_{*} (m)} := \sup_{ \| \phi \|_{H^{1}_{*} (m)} \le 1} \la f , \phi \ra_{L^2(m)}
= \sup_{ \| \phi \|_{H^{1}_{*} (m)} \le 1} \la m f , m \phi \ra_{L^2},
\ee
and observe that $\| f \|_{H^{-1}_{*} (m)} = \| m f \|_{H^{-1}_{*}}$.

\begin{lem}\label{lem:SBdecay}
Let $m_0,m_1$ be two admissible weight functions such that $m_1 \succ m_0 \succ \la v \ra^{(\gamma+3)/2}$. For any $t \ge 0$, there holds
\be\label{eq:SB-L2}
\| S_{\BB}(t) \|_{L^2(m_1) \to L^2(m_0)} \lesssim \Theta_{m_1,m_0}(t) .
\ee
Let $m_0,m_1, m$ be admissible polynomial weight functions such that $m \succeq m_1 \succ m_0 \succ \la v \ra^{(\gamma+3)/2}$. For any $t \ge 0$, there holds
\be\label{eq:SB*-L2}
\| S_{\BB^*_m}(t) \|_{L^2(\omega_1) \to L^2(\omega_0)} \lesssim \Theta_{m_1,m_0}(t) ,
\ee
where $\omega_1 := m/m_0$ and $\omega_0 := m/m_1$.

\end{lem}

\begin{proof}
We denote $X(m) = L^2(m)$.
We observe that for  $\tilde m_0 := m_0 \la v \ra^{(\gamma + \sigma)/2} \prec m_0 \prec m_1$ (where we recall that $\sigma = 0$ if $m_0$ is a polynomial function and $\sigma = s$ if $m_0$ is an exponential function),  there is a positive constant $C=C(m_0,m_1)$ such that for any $R \in (0,\infty)$ we have  
$$
\frac{\tilde m_0^2}{m_0^2}(R) \| f \|_{X(m_0)}^2 \le  \| f \|_{X(\tilde m_0)}^2 + C\frac{\tilde m_0^2}{m_1^2}(R) \| f \|_{X(m_1)}^2, 
$$
where we also denote by $m$ the function $R \mapsto m(v)$ for $|v|=R$. 
We write that estimate as 
\be\label{eq:interpolation}
\eps_R \| f \|_{X(m_0)}^2 \le \| f \|_{X(\tilde m_0)}^2 + C \theta_R \| f \|_{X(m_1)}^2,
\ee
with
$$
\eps_R := \frac{\tilde m_0^2}{m_0^2}(R), \quad \theta_R := \frac{\tilde m_0^2}{m_1^2}(R), \quad \eps_R,\; \frac{\theta_R}{\eps_R} \to 0 \quad\text{as } R \to \infty. 
$$
Let us denote $f_\BB (t) = S_{\BB}(t) f_0$ for any $t \ge 0$. Thanks to \eqref{eq:BL2} for the weight $m_1$, we have 
$$
\| f_{\BB}(t)\|_{X(m_1)} \le \| f_0 \|_{X(m_1)}, \quad \forall \,  t\ge 0.
$$ 
Writing now \eqref{eq:BL2} for $m_0$, using the interpolation \eqref{eq:interpolation} and the above estimate, for any $R>0$, we get ({ for some positive constants $\lambda,C>0$})
$$
\ba
\frac{d}{dt} \| f_{\BB} \|_{X(m_0)}^2 
&\le - \lambda \| f_{\BB} \|_{X(m_0 \la v \ra^{(\gamma + \sigma)/2})}^2 \\
&\le - \lambda \eps_R \| f_{\BB} \|_{X(m_0)}^2 + C \theta_R \| f_\BB \|_{X(m_1)}^2  
\\
&\le - \lambda \eps_R \| f_{\BB} \|_{X(m_0)}^2 + C \theta_R \| f_0 \|_{X(m_1)}^2  ,
\ea 
$$
with $\eps_R = \la R \ra^{\gamma+\sigma}$ and $\theta_R / \eps_R =  m_0^2(R) / m_1^2(R)$. 
Integrating that last differential inequality, we obtain 
$$
\ba
\| f_\BB(t) \|_{X(m_0)}^2 &\lesssim e^{- \lambda  \eps_R t} \| f_0 \|_{X(m_0)}^2 
+\frac{\theta_R}{\eps_R} \| f_0 \|_{X(m_1)}^2  \\
&\lesssim \Gamma_{m_1,m_0}^2 (t) \, \| f_0 \|_{X(m_1)}^2, 
\ea
$$
with
$$
\Gamma_{m_1,m_0}^2 (t) := \inf_{R>0} \left( e^{- \lambda \eps_R t} +\frac{\theta_R}{\eps_R} \right).
$$
We can complete the proof of \eqref{eq:SB-L2} by establishing $\Gamma_{m_1,m_0}(t) \lesssim \Theta_{m_1,m_0}(t)$ for the different choices of weight functions $m_0 \prec m_1$.

\medskip\noindent
{\it Case 1: $m_0 = \la v \ra^{k_0}$ and $m_1 = \la v \ra^{k_1}$ with $k_0 < k_1$.}
We have
$$
\Gamma_{m_1,m_0}^2 (t) = \inf_{R>0} \left( e^{- \lambda \la R \ra^{\gamma} t} + \la R \ra^{2(k_0-k_1)} \right).
$$
We take $\la R \ra = ( \la t \ra \theta(t))^{1 / |\gamma|}$ with $\theta(t) := [\log (1+t)]^{-2}$ and we get 
$$
\Gamma_{m_1,m_0}^2 (t) \le  e^{- \lambda \theta(t)^{-1}} + [\log(1+t)]^{4(k_1-k_0)/|\gamma|} \, \la t \ra^{-2   (k_1-k_0)/|\gamma|},
$$
from which we easily obtain $\Gamma_{m_1,m_0}(t) \lesssim \Theta_{m_1,m_0}(t)$.

\medskip\noindent
{\it Case 2: $m_0 = e^{\kappa_0 \la v \ra^s}$ and $m_1 = e^{\kappa_1 \la v \ra^s}$ with $\kappa_0 < \kappa_1$.}
We have
$$
\Gamma_{m_1,m_0}^2 (t) = \inf_{R>0} \left( e^{- \lambda \la R \ra^{\gamma + s} t} + e^{2 (\kappa_0 - \kappa_1) \la R \ra^s} \right).
$$
We take $\la R \ra = t^{1 / |\gamma|}$ and we get 
$$
\Gamma_{m_1,m_0}^2 (t) \le  e^{- \lambda t^{s/ |\gamma|}} + e^{-2 (\kappa_1 - \kappa_0) t^{s/|\gamma|}} ,
$$
which is nothing but $ \Theta_{m_1,m_0}^2(t)$. 
The general case $m_1 \succ m_0$ follows from that estimate, and the proof of \eqref{eq:SB-L2} is complete.

\medskip\noindent
{\it Case 3: $m_0 = \la v \ra^{k_0}$ and $m_1 = e^{\kappa_1 \la v \ra^s}$.} We define $ m = e^{\kappa \la v \ra^s}$ with $\kappa < \kappa_1$ so that $m_0 \prec  m \prec m_1$.  
Using Case 2 above with $m$ and $m_1$ we obtain 
$$
\| f_\BB (t) \|_{X(m_0)} \le \| f_\BB (t) \|_{X(m)} \lesssim \Theta_{m_1,m}(t) \| f_0 \|_{X(m_1)} \lesssim e^{- \lambda t^{s/|\gamma|}} \| f_0 \|_{X(m_1)}, 
$$
 and conclude with the estimate of Case 2 above.

\medskip\noindent
{\it Case 4: $m_0 = e^{\kappa_0 \la v \ra^{s_0}}$ and $m_1 = e^{\kappa_1 \la v \ra^{s}}$ with $s_0 < s$.}
 We first define $m = e^{\kappa \la v \ra^s}$ with $\kappa < \kappa_1$, so that $m_0 \prec  m \prec m_1$, and we argue as in Case 3. 
%\Red We then write $\| f_\BB (t) \|_{X(m_0)} \le \| f_\BB (t) \|_{X(m)}$ and conclude with the estimate of Case 2 above. ??? EXPLIQUER. Using Case 2 above with $m$ and $m_1$ we obtain $\| f_\BB (t) \|_{X(m_0)} \le \| f_\BB (t) \|_{X(m)} \lesssim \Theta_{m_1,m}(t) \| f_0 \|_{X(m_1)}\lesssim e^{- \lambda t^{s/|\gamma|}} \| f_0 \|_{X(m_1)} $. 

\smallskip

Estimate \eqref{eq:SB*-L2} can be proven similarly as above by using the estimates of Lemma~\ref{lem:BB2}, where we remark that in this case we have $\Theta_{\omega_1, \omega_0} (t) = \Theta_{m_1,m_0}(t)$, because $m_0, m_1, m$ are polynomial weight functions and $\omega_1 = m/m_0$, $\omega_0 = m/ m_1$.
\end{proof}

%-------------- Regularity for $S_\BB$ ------------------

\subsection{Regularity properties of $S_\BB$}

We now prove that the semigroup $S_\BB$ enjoys some regularization properties.
 
\begin{lem}\label{lem:SBreg}
Let $m_1,m $ be admissible polynomial weight functions such that $ \langle v \rangle^{3/2} \prec m_1 \prec m$. 
Then the following regularization estimate holds 
\be\label{eq:SH-1toL2}
\| S_\BB(t)  \|_{H^{-1}_{*}(m) \to L^2 (m_1 \la v \ra^{\gamma/2})}  \lesssim   \frac{\Theta_{m,m_1}(t)}{ t^{1/2} \wedge 1}  , \quad \forall \, t >  0.
\ee 
%with $\omega:= m/m_0$ and $\omega_1 := \langle v \rangle^{|\gamma|/2}$.

\end{lem}

\begin{proof}
We define $\omega_0 := 1$, $\omega_1 := \la v \ra^{|\gamma|/2}$ and $\omega := m/(m_1 \la v \ra^{\gamma/2})$, so that $1 \prec \omega \prec m \langle v \rangle^{- (\gamma+3)/2}$. We write $\phi_t :=  S_{\BB^*_m}(t) \phi$ for a giving function $\phi \in L^2(\omega_1)$. 
 Thanks to \eqref{eq:B*mL2} and \eqref{eq:B*mH1} together with $H^1_*(\omega_1) \subset H^1(\omega_0)$, we have for some constant $\lambda>0$
$$
{d \over dt} \Bigl( \| \phi_t \|_{L^2(\omega_1)} ^2 + \eta t \| \phi_t \|_{\tilde H^1(\omega_0)} ^2 \Bigr) 
 \le - \lambda \| \phi_t \|_{H^1_*(\omega_1)}^2 
+ \eta  \| \phi_t \|_{\tilde H^1(\omega_0)}^2 \le 0, 
$$
for $\eta > 0$ small enough. We deduce that 
\be\label{eq:SBH-1toL2smalltime}
 \eta t \, \| \phi_t \|_{H^1(\omega_0)} ^2 \lesssim  \| { \phi} \|_{L^2(\omega_1)} ^2, \quad \forall \, t \ge 0.
\ee
For large values of time $t \ge 1$, we can use \eqref{eq:SBH-1toL2smalltime} and \eqref{eq:SB*-L2} to obtain
$$
\| \phi_t \|_{H^1(\omega_0)} \lesssim \| \phi_{t-1} \|_{L^2(\omega_1)} \lesssim \Theta_{m,m_1}(t-1) \| \phi \|_{L^2(\omega)} \lesssim \Theta_{m,m_1}(t) \| { \phi} \|_{L^2(\omega)}.
$$
Both estimates together with $H^1(\omega_0) \subset H^1_*(\omega_0)$, we have proved 
\bean
\| S_{\BB_m^*} (t) \phi  \|_{H^1_*(\omega_0)} \lesssim {\Theta_{m,m_1}(t) \over t^{1/2} \wedge 1} \|  \phi  \|_{L^2(\omega)} \quad \forall \, t >  0.
\eean
We then get \eqref{eq:SH-1toL2} by duality. 
 More precisely, recalling that that 
$$
\forall\, t \ge 0, \quad m S_{B}(t) f = S_{B_m}(t) g, \quad
\la S_{B_m}(t) g , \phi \ra_{L^2} = \la g , S_{B^*_m}(t) \phi \ra_{L^2} ,
$$
we first have $\| S_{B}(t) f \|_{L^2(m_1 \la v \ra^{\gamma/2})} 
= \| \omega^{-1} S_{B_m}(t) g \|_{L^2}$ and then we can compute
$$
\ba
\| \omega^{-1} S_{B_m}(t) g \|_{L^2} 
&= \sup_{ \| \psi \|_{L^2} \le 1} \la  S_{B_m}(t) g , \omega^{-1}  \psi \ra_{L^2}  \\
&= \sup_{ \| \phi \|_{L^2(\omega)} \le 1} \la   g ,  S_{B^*_m}(t)\phi \ra_{L^2} \\
&\le \sup_{ \| \phi \|_{L^2(\omega)} \le 1} \| g  \|_{H^{-1}_*(\omega_0)} \, \| S_{B^*_m}(t)\phi \|_{H^1_*(\omega_0)} \\
&\lesssim \sup_{ \| \phi \|_{L^2(\omega)} \le 1}  \frac{\Theta_{m,m_1}(t)}{t^{1/2} \wedge 1}  \| g  \|_{H^{-1}_*(\omega_0)} \, \| \phi \|_{L^2(\omega)}, 
\ea
$$
which completes the proof of \eqref{eq:SH-1toL2} by coming back to the function $f = m^{-1} g$.
\end{proof}

%--------------- Uniform boundedness of $S_\LL$ -------------------------------
\subsection{Decay estimates for $S_\LL$}
We first prove decay estimates in a family of small reference spaces included in  $L^2(\mu^{-1/2})$.

\begin{prop}\label{prop:decay-small}
For any admissible weight $\nu$ such that $\mu^{-1/2} \prec \nu \prec \mu^{-1}$, there holds
$$
\forall\, t \ge 0, \quad
\| S_\LL (t) \Pi \|_{L^2(\nu) \to L^2(\mu^{-1/2})} 
\lesssim \Theta_{\nu,\mu^{-1/2}}(t) = C e^{-\lambda  t^{\frac{2}{|\gamma|}}}.
$$
\end{prop}

\begin{proof}
Let us denote for simplicity ${ E_0 = L^2(\mu^{-1/2}) \supset E_1 = L^2(\nu)}$.
We already know from \eqref{eq:LLsg} and \eqref{eq:BL2} that
$$
t \mapsto \| S_\LL (t) \Pi \|_{  E_0 \to E_0}  , \; 
t \mapsto \| S_\BB (t) \|_{E_1 \to E_1} \in L^\infty(\R_+).
$$
We then write, thanks to Duhamel's formula,
$$
S_{\LL}\Pi = S_{\BB} \Pi + S_{\BB} \AA  * S_{\LL}\Pi,
$$
and using Lemma \ref{lem:AA} and Lemma \ref{lem:SBdecay}, we obtain that $t \mapsto \| S_\BB \AA(t) \|_{ E_0 \to E_1} \in L^1 (\R_+)$, whence 
\be\label{eq:LinftytL2omega}
\| S_{\LL} (t) \Pi \|_{E_1 \to E_1} \lesssim \| S_{\BB}(t) \|_{E_1 \to E_1} + \| S_{\BB} \AA(t) \|_{  E_0 \to E_1} * \| S_{\LL}(t) \Pi \|_{E_1 \to  E_0 } \in L^\infty_t (\R_+).
\ee

%We conclude the proof by using the same interpolation argument as in the proof of Lemma~\ref{lem:SBdecay}. More precisely, 
Defining  $\Pi f_L (t) = \Pi S_L(t) f_0$ and using \eqref{eq:LLsg},  \eqref{eq:LinftytL2omega} and  the same interpolation argument as in the proof of Lemma~\ref{lem:SBdecay}, we obtain
$$
\ba
\frac{d}{dt} \| \Pi f_L(t)  \|_{E_0}^2 
&\le - \lambda \|   \la v \ra^{(\gamma + 2)/2} \,  \Pi f_L(t)  \|_{E_0 }^2 \\
&\le - \lambda \eps_R \| \Pi f_L(t)  \|_{E_0 }^2  + C\theta_R \| \Pi S_\LL(t) f_0  \|_{E_1}^2  
\\
&\le - \lambda \eps_R \| \Pi f_L(t)  \|_{E_0 }^2 + C\theta_R \| \Pi f_0 \|_{E_1}^2  ,
\ea 
$$
%$$
%\frac{d}{dt} \| \Pi f_L(t) \|_E^2 \lesssim - \epsilon_R \| \Pi f_L(t) \|_{E}^2 + \theta_R \| \Pi f_0 \|_{E_1}^2,  
%$$
with $\eps_R = \la R \ra^{\gamma+2}$ and $\theta_R/\eps_R = \mu^{-1/2}(R) / \nu(R)$. We conclude as   in the proof of Lemma~\ref{lem:SBdecay}. 
\end{proof}

As an immediate consequence, we prove uniform in time bounds for the semigroup $S_\LL$ in large spaces.

\begin{lem}\label{lem:SLinftyBIS}
For any admissible weight function $m \succ \la v \ra^{\frac{\gamma+3}{2}}$, there holds 
$$
t \mapsto \| S_{\LL}(t) \Pi  \|_{L^2(m) \to L^2(m)} \in L^\infty(\R_+).
$$
\end{lem}

\begin{proof}
Let us denote $E = L^2 (\mu^{-1/2})$, $E_1 = L^2(\nu)$ and $X = L^2(m)$, with $\mu^{-1/2} \prec \nu \prec \mu^{-1}$.
We only need to treat the case $ \la v \ra^{\frac{\gamma+3}{2}} \prec m \prec \mu^{-1/2}$ so that $E \subset X$ (the other cases have already been treated in \eqref{eq:LinftytL2omega}). We first write
$$
S_{\LL} \Pi = \Pi S_{\BB} + S_{\LL} \Pi * \AA S_{\BB},
$$
and observe that $t \mapsto \| S_{\BB} (t) \|_{X \to X} \in L^\infty(\R_+) $ from \eqref{eq:BL2} and $t \mapsto \| S_{\LL}(t) \Pi \|_{E_1 \to E} \in L^1 (\R_+)$ from Proposition~\ref{prop:decay-small}. Moreover, Lemma \ref{lem:AA} and Lemma \ref{lem:SBdecay} yield $t \mapsto \| \AA S_\BB(t) \|_{X \to E_1} \in L^\infty (\R_+)$, so that
$$
\| S_{\LL}(t) \Pi \|_{X \to X} \lesssim \| S_{\BB}(t) \|_{X \to X} + \| S_{\LL}(t) \Pi \|_{E_1 \to E \to X} * \| \AA S_{\BB} (t)  \|_{X \to E_1} \in L^\infty_t (\R_+),
$$
and the proof is complete.
\end{proof}

We can now prove that $S_\LL $ inherits the decay and regularity estimates already established for the semigroup $S_\BB$.

\begin{thm}\label{thm:S_L}
Let $m_0,m_1$ be two admissible weight functions such that $ \langle v \rangle^{(\gamma+3)/2} \prec m_0 \prec m_1$ and $m_0 \preceq \mu^{-1/2}$. There holds
\be\label{eq:LLdecayL2}
\| S_{\LL}(t) \Pi \|_{L^2(m_1) \to L^2(m_0)} \lesssim \Theta_{m_1,m_0}(t),  \quad \forall\, t \ge 0.
\ee
Let $m_0,m_1$ be two admissible polynomial weight functions such that $ \langle v \rangle^{3/2} \prec m_0 \prec m_1$. There holds
\be\label{eq:LLregH-1}
\| S_{\LL}(t) \Pi \|_{H^{-1}_*(m_1) \to L^2(m_0 \la v \ra^{\gamma/2})} \lesssim \frac{\Theta_{m_1,m_0}(t)}{t^{1/2} \wedge 1},  \quad \forall\, t > 0.
\ee

\end{thm}

\begin{proof}
We fix an admissible weight function $\nu$ such that $\mu^{-1/2} \prec \nu \prec \mu^{-1}$ and $\nu \succ m_1$, and we split the proof into two steps.

\smallskip\noindent
{\it Step 1.} We denote $X_0 = L^2(m_0)$, $X_1=L^2(m_1)$, $E_0 = L^2(\mu^{-1/2})$ and $E_1 = L^2(\nu)$.
We write the factorization identity
$$
S_\LL \Pi = \Pi S_\BB  +  S_\LL \Pi * \AA S_\BB,
$$ 
which implies
$$
\ba
\Theta_{m_1,m_0}^{-1} \, \| S_\LL  \Pi\|_{X_1 \to X_0} 
&\lesssim  \Theta_{m_1,m_0}^{-1} \, \| S_\BB  \Pi \|_{X_1 \to X_0}  + \left( \Theta_{m_1,m_0}^{-1} \,\| \Pi S_\LL \|_{E_1 \to X_0} * \Theta_{m_1,m_0}^{-1} \,\| \AA S_\BB \|_{X_1  \to E_1}  \right) .
\ea
$$
Thanks to Lemma~\ref{lem:SBdecay}, Proposition~\ref{prop:decay-small} and  Lemma \ref{lem:AA}, we have 
$$
\ba
& t \mapsto \Theta_{m_1,m_0}^{-1}(t) \, \| S_\BB (t) \Pi\|_{X_1 \to X_0} \in L^\infty (\R_+), \\
& t \mapsto \Theta_{m_1,m_0}^{-1}(t) \, \| \Pi S_\LL(t) \|_{E_1 \to E_0 \to X_0}  \in L^1(\R_+), \\
& t \mapsto \Theta_{m_1,m_0}^{-1}(t) \, \| \AA S_\BB (t)\|_{X_1  \to X_0 \to E_1} \in L^\infty (\R_+),
\ea
$$
which concludes the proof of \eqref{eq:LLdecayL2}.

\smallskip\noindent
{\it Step 2.}  Denote $Z_1 = H^{-1}_{*}(m_1)$ and $\widetilde X_0 = L^2(m_0 \la v \ra^{\gamma/2})$. Writing the factorization identity as in Step 1 and denoting $\widetilde \Theta_{m_1,m_0}(t) = \Theta_{m_1,m_0}(t) / (t^{1/2} \wedge 1)$, we have
$$
\ba
\widetilde \Theta_{m_1,m_0}^{-1} \| S_\LL  \Pi  \|_{Z_1 \to \widetilde X_0} 
&\lesssim \widetilde \Theta_{m_1,m_0}^{-1} \| S_\BB  \|_{Z_1 \to \widetilde X_0} + 
\left(  \widetilde \Theta_{m_1,m_0}^{-1} \| S_\LL  \Pi \|_{E_1  \to \widetilde X_0} *  \widetilde \Theta_{m_1,m_0}^{-1} \| \AA S_\BB \|_{Z_1  \to E_1} \right).
\ea
$$
Thanks to Lemma~\ref{lem:AA}, Lemma~\ref{lem:SBreg}, and Proposition~\ref{prop:decay-small}, we deduce
$$
\ba
& t \mapsto \widetilde \Theta_{m_1,m_0}^{-1}(t) \, \| S_\BB (t) \Pi\|_{Z_1 \to \widetilde X_0} \in L^\infty (\R_+), \\
& t \mapsto \widetilde \Theta_{m_1,m_0}^{-1}(t) \, \| \Pi S_\LL(t) \|_{E_1 \to E_0 \to \widetilde X_0}  \in L^1(\R_+), \\
& t \mapsto \widetilde \Theta_{m_1,m_0}^{-1}(t) \, \| \AA S_\BB (t)\|_{Z_1 \to \widetilde X_0 \to E_1} \in L^\infty (\R_+),
\ea
$$
which implies \eqref{eq:LLregH-1}. 
\end{proof}

\subsection{Weak dissipativity of $\LL$}
As a final step, we establish that $\LL$ is weakly dissipative in some appropriate spaces. 
In order to do that, we define the spaces 
\beqn\label{eq:defXYZ}
X := L^2(m), \quad Y := H^1_{*}(m), \quad  Z := H^{-1}_{*}(m), \quad
X_0 := L^2,
\eeqn
where we recall that $H^1_{*}(m)$ and $H^{-1}_{*}(m)$ have been introduced in \eqref{def:H1*} and \eqref{def:H-1*}. 
For any $\eta > 0$, we also define the norm $\Nt \cdot \Nt_X$  on $\Pi X$  by 
\be\label{eq:NormT-X}
\Nt f \Nt_{X}^2 := \eta \| f \|_{X}^2 + \int_0^\infty \| S_{\LL} (\tau) f \|_{X_0}^2 \, d\tau,
\ee
and we denote by $\la\!\la \cdot, \cdot \ra\!\ra_X$ the associated duality product.

\begin{prop}\label{prop:Nt}
Let $m$ be an admissible weight function such that $ m \succ \la v \ra^{\frac{3}{2}} $. The norm $\Nt \cdot \Nt_X$ is equivalent to $\| \cdot \|_X$  on $\Pi X$ , and, moreover, there exists $\eta>0$ small enough such that
\be\label{eq:LLdissip}
\frac{d}{dt} \Nt S_\LL(t) f \Nt_{X}^2 \lesssim -  \| S_\LL(t) f \|_{Y}^2, \quad \forall \, f \in \Pi X.
\ee

\end{prop}

\begin{proof}
We easily observe that, thanks to Theorem~\ref{thm:S_L},
$$
\int_0^\infty \| S_{\LL} (\tau) f \|_{X_0}^2 \, d\tau \lesssim \| f \|_X^2 \int_0^\infty \Theta^2(\tau) \, d\tau,
$$
for some decay function $\Theta \in L^2(\R_+)$ under the condition $m \succ \la v \ra^{3/2}$, thus $\Nt \cdot \Nt_X$ is equivalent to $\| \cdot \|_X$  on $\Pi X$ .
Now denote $f_\LL(t) = S_\LL(t) f_0$, $f_0 \in \Pi X$, so that $f_\LL(t) \in  \Pi X$ for any $t \ge 0$, recall that $\LL = \AA + \BB$ and write
$$
\frac12 \frac{d}{dt} \Nt f_\LL(t) \Nt_{X}^2 = \eta \la \BB f_\LL(t) , f_\LL(t) \ra_{X}
+ \eta \la \AA f_\LL(t) , f_\LL(t) \ra_{X} +  \frac12  \int_0^\infty \frac{d}{d\tau} \| S_\LL(\tau) f_\LL(t) \|_{X_0}^2 \, d\tau.
$$
Thanks to Lemma~\ref{lem:BB} and Lemma~\ref{lem:AA}, we have
$$
\eta \la \BB f_\LL(t) , f_\LL(t) \ra_{X} \le - \eta K' \| f_\LL(t) \|_{Y}^2, \quad
\eta \la \AA f_\LL(t) , f_\LL(t) \ra_{X} \le \eta C \| f_{\LL}(t) \|_{X_0}^2.
$$
Moreover, for the last term, we have
$$
\int_0^\infty \frac{d}{d\tau} \| S_\LL(\tau) f_\LL(t) \|_{X_0}^2 \, d\tau
= \lim_{\tau \to \infty} \| S_\LL(\tau) f_\LL(t) \|_{X_0}^2 - \|  f_\LL(t) \|_{X_0}^2
= - \|  f_\LL(t) \|_{X_0}^2,
$$
where we have used 
%that if $f_0 \in {X} $ then
%$$
%\| f_\LL(t) \|_{X} \le C \| f_0 \|_{X} \quad \forall\, t \ge 0,
%$$
%and
$$
\forall \, t \ge 0, \quad 
\| S_\LL(\tau) f_\LL(t) \|_{X_0} \le C \Theta_{m} (\tau) \| f_0 \|_{X} 
\quad \text{with}\quad
\lim_{\tau \to\infty} \Theta_{m} (\tau)=0,
$$
thanks to Lemma~\ref{lem:SLinftyBIS} and Theorem~\ref{thm:S_L}. We conclude the proof of \eqref{eq:LLdissip} gathering previous estimates and taking $\eta>0$ small enough. 
\end{proof}

%%% ----------------------------------------------------------------------------------------

\subsection{Summarizing the decay and dissipativity estimates}

We summarize the set of information we have established in this section and that we will use in order to get our main existence, uniqueness and stability result for the nonlinear equation in Section~\ref{sec:stabNL} (in the spatially homogeneous case). 
Consider the spaces defined in \eqref{eq:defXYZ}.

\begin{cor}\label{cor:S_L-homBIS} 
Consider an admissible weight function $m$ such that $m \succ \langle v \rangle^{2+3/2}$. With the above assumptions and notation, there exists $\eta > 0$ such that  the norm $\Nt \cdot \Nt_X$ defined in \eqref{eq:NormT-X}
is equivalent to the initial norm on $ \Pi X$ and 
\bear
\label{eq:LLdissipL2-homBIS}
&&\la\!\la \LL \Pi f  , \Pi f  \ra\!\ra_X \lesssim -  \|  \Pi f \|_{Y}^2, \quad \forall\, f \in X^\LL_1, 
\\
\label{eq:LLregH1-homBIS}
&&    
t \mapsto \| S_{\LL}(t) \Pi \|_{Y \to X_0} \, \| S_{\LL}(t) \Pi   \|_{Z \to X_0} \in L^1 (\R_+),
\eear
 where we recall that $X^\LL_1$ is the domain of $\LL$ when acting on $X$.

\end{cor}

 It is worth observing again that the polynomial decay rate \eqref{eq:LLregH-1} in Theorem~\ref{thm:S_L} has been established in polynomial weighted { Sobolev} spaces and thus  immediately extends with same decay rate to exponential weighted { Sobolev} spaces. That remark is used in the proof of the second estimate in \eqref{eq:LLregH1-homBIS} which is valid for any (polynomial or not) admissible weight function.

% functio only for polynomial  weight functions and this estimate will be used in the proof below, estimate \eqref{eq:LLregH1-homBIS} is valid for any admissible weight $m$. The reason for this is that we only need to treat the worst case scenario corresponding to the worst decay function $\Theta (t)$ associated to polynomial  functions.}

\begin{proof}  Using the identity 
$$
  \frac12   \frac{d}{dt} \Nt S_\LL(t) \Pi f \Nt_X^2 = \la\!\la \LL \Pi f , \Pi f \ra\!\ra_X,
$$
we see that  estimate \eqref{eq:LLdissipL2-homBIS} is just a reformulation of \eqref{eq:LLdissip} in Proposition~\ref{prop:Nt}.
%,  indeed $\frac{d}{dt} \Nt S_\LL(t) \Pi f \Nt_X^2 = \la\!\la \LL \Pi f , \Pi f \ra\!\ra_X$.
  
We now prove estimate \eqref{eq:LLregH1-homBIS}. We fix admissible polynomial weight functions $m_0$ and $m_1$ such that $\la v \ra^{(\gamma+3)/2}\prec m_0 \prec m_1 \preceq \la v \ra^{\gamma/2} m$. Then estimate \eqref{eq:LLdecayL2} in Theorem~\ref{thm:S_L} and the embeddings $L^2(m_0) \subset X_0$ and $Y \subset L^2(m_1)$ imply 
$$
\| S_{\LL}(t) \Pi \|_{Y \to X_0} \lesssim \Theta_{m_1,m_0}(t), \quad \forall t \ge 0.
$$
Now consider admissible polynomial weight functions $m'_0$ and $m'_1$ so that $\la v \ra^{3/2}\prec m'_0 \prec m'_1 \preceq  m$. Thanks to estimate \eqref{eq:LLregH-1} in Theorem~\ref{thm:S_L} together with the embeddings $L^2(m'_0 \la v \ra^{\gamma/2}) \subset X_0$ and $Z \subset H^{-1}_*(m'_1)$, we obtain
$$
\| S_{\LL}(t) \Pi  f \|_{Z \to X_0} \lesssim \frac{\Theta_{m'_1, m'_0}(t)}{t^{1/2} \wedge 1}, 
\quad \forall \, t >0.
$$
 We finally obtain \eqref{eq:LLregH1-homBIS} by observing that $t \mapsto \la t \ra^{-(2k-3)/|\gamma|} \, (t \wedge 1)^{-1/2} \in L^1(\R_+)$ for any $k > 2+3/2 $ and that we may thus choose $m_0,m_1, m'_0$ and $m'_1$ adequately in such a way that $ t \mapsto \Theta_{m_1,m_0}(t) \,  \Theta_{m'_1, m'_0}(t) \, (t \wedge 1)^{-1/2} \in L^1(\R_+)$.  
%
%\Red
%KC : tu as chang\'e "for some $ k > 2 + 3/2 $" en "for any $ k > 2 + 3/2 $", mais le dernier n'est pas vrai car les poids $m_1,m_0, m'_1, m'_0$ sont des polynomes. 
%
%that $m \succ \la v \ra^{2+3/2}$ and then 
%$$
%\Theta_{m_1,m_0}(t) \, \frac{\Theta_{m'_1, m'_0}(t)}{t^{1/2} \wedge 1} \lesssim \frac{\la t \ra^{-(2k-3)/|\gamma|}}{t^{1/2} \wedge 1} \in L^1(\R_+),
%$$
%\Red by choosing because   $k > 2+3/2 $. %> |\gamma|/2 + 3/2$.% and because  $m \succ \la v \ra^{2+3/2}$.
%%
%%
%%\Red
%%KC : tu as chang\'e "for some $ k > 2 + 3/2 $" en "for any $ k > 2 + 3/2 $", mais le dernier n'est pas vrai car les poids $m_1,m_0, m'_1, m'_0$ sont des polynomes. 
\end{proof}

%%% ------------------- %%% ------------------- %%% ------------------- %%% -------------------

\section{Nonlinear estimates}\label{sec:nonL}

In this section, we present some estimates on the nonlinear Landau operator $Q$. We start with two auxiliary results. 

\begin{lem}$($\cite[Lemma 3.2]{CTW}$)$\label{lem:Aalpha}
Let $-3<\alpha<0$ and $\theta>3$. Then
$$
A_\alpha(v) := \int_{\R^3} |v-v_*|^\alpha \, \la v_* \ra^{-\theta} \, dv_* \lesssim \la v \ra^\alpha.
$$
\end{lem}

\begin{lem}\label{lem:abc*f}
There holds 
\begin{enumerate}[(i)]

\item For any $3/(3+\gamma+2) < p \le \infty$ and $\theta > 2 + 3(1-1/p)$
$$
|(a_{ij}*f)(v) \, v_i v_j|
+ |(a_{ij}*f)(v) \, v_i |
+ |(a_{ij}*f)(v)  |   \lesssim \la v \ra^{\gamma+2} \, \| f \|_{L^p (\la v \ra^{\theta})}.
$$

\item For any $3/(3+\gamma+1) < p \le \infty$ and any $\theta'>3(1-1/p)$
$$
|(b_{j}*f)(v) |  \lesssim \la v \ra^{\gamma+1} \, \| f \|_{L^p(\la v \ra^{\theta'})}.
$$

\end{enumerate}

\end{lem}

\begin{proof}
(i) Recall that $0$ is an eigenvalue of the matrix $a_{ij}(z)$ so that $a_{ij}(v-v_*) v_i = a_{ij}(v-v_*) v_{*i} $ and $a_{ij}(v-v_*) v_i v_j = a_{ij}(v-v_*) v_{*i} v_{*j}$. 
Thanks to Holder's inequality and using Lemma \ref{lem:Aalpha}, we obtain for any $3/(3+\gamma+2) < p \le \infty$ and any $\bar \theta > 3(1-1/p)$,
$$
\ba
| (a_{ij} * f)(v) \, v_i v_j | 
&= |\int_{v_*} a_{ij}(v-v_*) v_{*i} v_{*j} f_*| \\
&\lesssim \int_{v_*} |v-v_*|^{\gamma+2} \, \la v_* \ra^{-\bar \theta}   \, \la v_* \ra^{\bar\theta+2}  |f_*| \\
&\lesssim \left(\int_{v_*} |v-v_*|^{(\gamma+2)\frac{p}{p-1}} \, \la v_* \ra^{-\bar\theta \frac{p}{p-1}}  \right)^{(p-1)/p}   \| f \|_{L^p(\la v \ra^{\bar\theta+2})} \\ 
&\lesssim \la v \ra^{\gamma+2} \| f \|_{L^p(\la v \ra^{\bar\theta+2})}.
\ea
$$
We can get the estimates for $(a_{ij} * f)(v) \, v_i$ and $(a_{ij} * f)(v)$ in a similar way.
Remark that we can choose $p=2$ since $\gamma \in [-3,-2)$.

\smallskip

(ii) For the term $(b*f)$ we recall that $b_i(z) = -2 |z|^\gamma z_i$. Thanks to Holder's inequality and Lemma~\ref{lem:Aalpha}, we obtain for any $3/(3+\gamma+1) < p \le \infty$ and any $\theta'>3(1-1/p)$,
$$
\ba
| (b_{i} * f)(v) | 
&\lesssim \int_{v_*} |v-v_*|^{\gamma+1} \, \la v_* \ra^{-\theta'}   \, \la v_* \ra^{\theta'}  |f_*| \\
&\lesssim \left(\int_{v_*} |v-v_*|^{(\gamma+1)\frac{p}{p-1}} \, \la v_* \ra^{-\theta' \frac{p}{p-1}}  \right)^{(p-1)/p}   \| f \|_{L^p(\la v \ra^{\theta'})} \\ 
&\lesssim \la v \ra^{\gamma+1} \| f \|_{L^p(\la v \ra^{\theta'})}.
\ea
$$
Remark now that we have $3/(3+\gamma+1) \in (3/2,3]$, thus we can choose $p=4$ for any $\gamma \in [-3,-2)$.
\end{proof}

\smallskip

We establish our main estimate on the Landau collision operator. 

\begin{lem}\label{lem:Q}
Consider any admissible weight function $m \succeq 1$. Then, for any $\theta>2+3/2$ and $\theta'>9/4 $, there holds
\be\label{eq:Qfgh0}
\la Q(f,g), h \ra_{L^2(m)} \lesssim \Big( \| f \|_{L^2(\la v \ra^\theta)} \, \| g \|_{H^1_{*}(m)} + \| f \|_{H^1( \la v \ra^{\theta'})} \, \| g \|_{L^2(m)} \Big) \, \| h \|_{H^1_{*}(m)}.
\ee
\end{lem}

\begin{proof}
Let us denote $G = mg$ and $H = mh$. We write
$$
\ba
\la Q(f,g) , h \ra_{L^2(m)} 
&= \int \partial_j \{ (a_{ij} * f) \partial_{i} g  - (b_j*f) g \} \, h \, m^2 \\
&= \int \partial_j  \{ (a_{ij} * f) \partial_{i} (m^{-1} G) \} \, H \, m
- \int \partial_j \{ (b_j*f) \, m^{-1} G \} \, H \, m =: A + B.
\ea
$$
Performing an integration by parts and developing terms, we easily get $A = A_1 + A_2+A_3 + A_4$ and $B = B_1 + B_2$, with
$$
A_1 := -\int (a_{ij} * f) \, \partial_i G \, \partial_j H, \quad
A_2 := -\int (a_{ij} * f) \, \frac{\partial_j m}{m} \, \partial_i G \, H,
$$
$$
A_3 := \int (a_{ij} * f) \, \frac{\partial_i m}{m} \,  G \,\partial_j H, \quad
A_4 := \int (a_{ij} * f) \,  \frac{\partial_i m}{m} \, \frac{\partial_j m}{m} \, G \, H,
$$
$$
B_1 := \int (b_j * f)  \, G \, \partial_j H, \quad
B_2 := \int (b_j * f) \, \frac{\partial_j m}{m} \, G \, H. 
$$
We then estimate each term separately.

\medskip\noindent
{\it Step 1. Term $A_1$.} 
We only consider the case $|v|>1$, since the estimate for $|v|\le 1$ is evident. 
We decompose $\partial_i G = P_v \partial_i G + (I-P_v) \partial_i G =: \partial_i^{\|} G + \partial_i^{\perp} G$, and similarly for $\partial_j H = \partial_j^{\|} H +  \partial_j^{\perp} H$. We  write
$$
\ba
{ A_1^+ } &:= \int_{|v|>1} (a_{ij}*f) \, \{  \partial_i^{\|} G \,  \partial_j^{\|} H +  \partial_i^{\|} G \,  \partial_j^{\perp} H +  \partial_i^{\perp} G \,  \partial_j^{\|} H
 +  \partial_i^{\perp} G \, \partial_j^{\perp} H \} \\
 &=: T_{1} + T_{2} + T_{3} + T_{4}.
\ea
$$
Using Lemma \ref{lem:abc*f}-(i) with $p=2$, for any $\theta>2+3/2$, we have
$$
\ba
T_{1} 
& = \int_{|v|>1} (a_{ij}*f) v_i v_j \, \frac{(v \cdot \nabla_v G)}{|v|^2} \, \frac{(v \cdot \nabla_v H)}{|v|^2}  \\
&\lesssim \| f \|_{L^2(\la v \ra^\theta)} \int_{|v|>1} \la v \ra^{\gamma+2}  |v|^{-2}\, |\nabla_v G| \, |\nabla_v H|   \\
&\lesssim \| f \|_{L^2(\la v \ra^\theta)}  \, \| \la v \ra^{\frac{\gamma}{2}} \nabla_v (mg) \|_{L^2}
\, \| \la v \ra^{\frac{\gamma}{2}} \nabla_v (m h) \|_{L^2}.
\ea
$$
On the other hand, we have 
$$
\ba
T_{2} 
& = \int_{|v|>1} (a_{ij}*f) v_i  \, \frac{(v \cdot \nabla_v G)}{|v|^2} \, \partial_j^{\perp} h\\
&\lesssim \| f \|_{L^2(\la v \ra^\theta)} \int_{|v|>1} \la v \ra^{\gamma+2}  |v|^{-1}\, |\nabla_v G| \, |\nabla_v^{\perp} H|   \\
&\lesssim \| f \|_{L^2(\la v \ra^\theta)} \, \| \la v \ra^{\frac{\gamma}{2}} \nabla_v (mg) \|_{L^2}
\, \| \la v \ra^{\frac{\gamma+2}{2}} \nabla_v^{\perp} (m h) \|_{L^2},
\ea
$$
and similarly
$$
\ba
T_{3} 
&\lesssim \| f \|_{L^2(\la v \ra^\theta)}  \, \| \la v \ra^{\frac{\gamma+2}{2}}  \nabla_v^{\perp} (mg) \|_{L^2}
\, \| \la v \ra^{\frac{\gamma}{2}} \nabla_v (mh) \|_{L^2}.
\ea
$$
For the term $T_{4}$, we have
$$
\ba
T_{4} 
&\lesssim \| f \|_{L^2(\la v \ra^\theta)} \int \la v \ra^{\gamma+2}   |\nabla_v^{\perp} G| \, |\nabla_v^{\perp} H|   \\
&\lesssim \| f \|_{L^2(\la v \ra^\theta)}  \, \| \la v \ra^{\frac{\gamma+2}{2}} \, \nabla_v^{\perp}  (mg) \|_{L^2}
\, \| \la v \ra^{\frac{\gamma+2}{2}} \, \nabla_v^{\perp} (mh) \|_{L^2}.
\ea
$$
All in all, we obtain 
$$
{ A_1^+} \lesssim \| f \|_{L^2(\la v \ra^\theta)}  \, \| g \|_{H^1_*(m)} \, \| h \|_{H^1_*(m)} .
$$

\medskip\noindent
{\it Step 2. Term $A_2$.}
Recall that $\partial_j m^2 = C v_j \la v \ra^{\sigma-2} m^2$. The case $|v|\le 1$ is evident so we only consider $|v|>1$. The same argument as for $A_1$ gives us
$$
\ba
{ A_2^+} &:= C \int_{|v|>1} (a_{ij}*f) \, v_j \la v \ra^{\sigma-2} \big\{  \partial_i^{\|} G + \partial_i^{\perp} G  \big\} \, H \\
&\lesssim \| f \|_{L^2(\la v \ra^\theta)}  \, \int \Big\{ \la v \ra^{\gamma+\sigma-1} |\nabla_v G| + \la v \ra^{\gamma+\sigma} |\nabla_v^{\perp} G| \Big\} \, |H| \\
&\lesssim \| f \|_{L^2(\la v \ra^\theta)} \, \Big\{ \| \la v \ra^{\frac{\gamma}{2}} \, \nabla_v (mg) \|_{L^2}  + \| \la v \ra^{\frac{\gamma+2}{2}} \, \nabla_v^{\perp} (mg)   \|_{L^2}   \Big\} \, \| \la v \ra^{\frac{\gamma+2\sigma-2}{2}}
\, h \|_{L^2(m)} \\
%&\lesssim \| f \|_{L^2(\la v \ra^\theta)}  \, \| g \|_{H^1_*(m)} \, \| \la v \ra^{\frac{\gamma+\sigma}{2}} h \|_{L^2(m)}
&\lesssim \| f \|_{L^2(\la v \ra^\theta)}  \, \| g \|_{H^1_*(m)} \, \| h \|_{H^1_*(m)}.
\ea
$$

\medskip\noindent
{\it Step 3. Term $A_3$.} In a similar way as for the term $A_2$, we also have 
$$
A_3 \lesssim \| f \|_{L^2(\la v \ra^\theta)}  \, \| \la v \ra^{\frac{\gamma+\sigma}{2}} g \|_{L^2(m)} \, \| h \|_{H^1_*(m)}
\lesssim  \| f \|_{L^2(\la v \ra^\theta)}  \, \| g \|_{H^1_*(m)} \, \| h \|_{H^1_*(m)}.
$$

\medskip\noindent
{\it Step 4. Term $A_4$.} 
Arguing as before, we easily get
$$
\ba
{ A_4^+} &:= C \int_{|v| >1} (a_{ij} * f) \, v_i \, v_j \la v \ra^{2\sigma-4} \, G\, H \, \\
&\lesssim \| f \|_{L^2(\la v \ra^\theta)}  \, \int \la v \ra^{\gamma+2\sigma-2} \, |G| \, |H| \\
& \lesssim \| f \|_{L^2(\la v \ra^\theta)}  \, \| \la v \ra^{\frac{\gamma+2\sigma-2}{2}} \, g \|_{L^2(m)} \, \| \la v \ra^{\frac{\gamma+2\sigma-2}{2}} \, h \|_{L^2(m)} \\
%& \Green \lesssim \| f \|_{L^2(\la v \ra^\theta)}  \, \| \la v \ra^{\frac{\gamma+\sigma}{2}} \, g \|_{L^2(m)} \, \| \la v \ra^{\frac{\gamma+\sigma}{2}} \, h \|_{L^2(m)} \\
&\lesssim  \| f \|_{L^2(\la v \ra^\theta)}  \, \| g \|_{H^1_*(m)} \, \| h \|_{H^1_*(m)}.
\ea
$$

\medskip\noindent
{\it Step 5. Term $B_1$.}
Thanks to Lemma~\ref{lem:abc*f}-(ii) with $p=4$, for any $\theta' >9/4$, it follows
$$
\ba
B_1 &\lesssim \| f \|_{L^4(\la v \ra^{\theta'})} \, \int \la v \ra^{\gamma+1} \, |G| \, |\nabla_v H| \\
& \lesssim \| f \|_{H^1(\la v \ra^{\theta'})} \, \| \la v \ra^{\frac{\gamma+2}{2}} \, g \|_{L^2(m)} \, \| \la v \ra^{\frac{\gamma}{2}} \, \nabla_v (mh) \|_{L^2} \\
& \lesssim \| f \|_{H^1(\la v \ra^{\theta'})} \, \| \la v \ra^{\frac{\gamma+2}{2}} \, g \|_{L^2(m)} \, \| h \|_{H^1_*(m)}, 
\ea
$$
where we have used the embedding $H^1 (\la v \ra^{\theta'}) \subset L^4(\la v \ra^{\theta'})$.

\medskip\noindent
{\it Step 6. Term $B_2$.} Using $\partial_j m =C v_j \la v \ra^{\sigma-2} m$, we have 
$$
\ba
B_2 
&\lesssim \| f \|_{L^4(\la v \ra^{\theta'})} \, \int \la v \ra^{\gamma+\sigma} \, |G| \, | H|\\
&\lesssim \| f \|_{H^1(\la v \ra^{\theta'})} \, \| \la v \ra^{\frac{\gamma+\sigma}{2}} \, g \|_{L^2(m)} \, \| \la v \ra^{\frac{\gamma+\sigma}{2}} \, h \|_{L^2(m)} \\
&\lesssim \| f \|_{H^1(\la v \ra^{\theta'})} \, \| \la v \ra^{\frac{\gamma+\sigma}{2}} \, g \|_{L^2(m)} \, \| h \|_{H^1_*(m)}.
\ea
$$

\medskip\noindent
{\it Step 7. Conclusion.}
Gathering previous estimates and using that $\| \la v \ra^{\frac{\gamma+\sigma}{2}} g \|_{L^2(m)}$ and $\| \la v \ra^{\frac{\gamma+2}{2}} g \|_{L^2(m)}$ can be controlled by 
$\| g \|_{L^2(m)}$, we obtain, for any $\theta >2 + 3/2$ and $\theta' >9/4$,
$$
\la  Q(f,g),h \ra_{L^{2}(m)}
\lesssim \| f \|_{L^2(\la v \ra^{\theta})} \, \| g \|_{H^1_{*}(m)} \, \| h \|_{H^1_{*}(m)} 
 +  \| f \|_{H^1(\la v \ra^{\theta'})}  \, \| g \|_{L^2(m)} \, \|  h \|_{H^1_{*}(m)} ,
$$
which concludes the proof of \eqref{eq:Qfgh0}. 
\end{proof}

\begin{cor}\label{cor:Q}
Consider an admissible weight function $m$ such that $m \succ \la v \ra^{2+3/2} $. With the notation \eqref{eq:defXYZ}, there holds 
\be\label{eq:QfghBIS}
\la Q(f,g), h \ra_X \lesssim \Big( \| f \|_X\, \| g \|_Y + \| f \|_Y \, \| g \|_X \Big) \, \| h \|_Y,
\ee
and in particular 
\beqn\label{eq:QfgH-1BIS}
\| Q(f,g) \|_Z  \lesssim    \| f \|_X \, \| g \|_Y + \| f \|_Y \, \| g \|_X. 
\eeqn
\end{cor}

\begin{proof}  The proof of \eqref{eq:QfghBIS} easily follows from \eqref{eq:Qfgh0} observing that, since $m \succ \la v \ra^{2+3/2}$, we can choose $\theta$ and $\theta'$ in Lemma~\ref{lem:Q}  such that $L^2(m) \hookrightarrow L^2(\la v \ra^{\theta})$ and 
$H^1_{*}(m) \hookrightarrow H^1(\la v \ra^{\theta'})$ (see \eqref{def:H1*}). The proof of \eqref{eq:QfgH-1BIS} 
is then straightforward by the definition of $Z=H^{-1}_{*}(m)$ (see \eqref{def:H-1*}).  
\end{proof}

%%%%%%%%%%%%%%%%%%%%%%%%%%%
\section{Nonlinear stability}\label{sec:stabNL}

This section is devoted to the proof of the spatially homogeneous version of Theorem~\ref{thm:stabNL-inhom}.

\medskip

Consider a solution $F$ to the homogeneous Landau equation \eqref{eq:landau} and define the variation $f= F- \mu$, which satisfies, 
\be\label{eq:pert}
\left\{
\ba
& \partial_t f = \LL f + Q(f,f) \\
& f_{|t=0} = f_0 = F_0 - \mu.
\ea
\right.
\ee
We observe that, $\Pi_0 f_0 = 0$ and therefore, thanks to the conservation laws,
$$
\Pi_0 f(t) = \Pi_0 Q(f(t),f(t)) = 0 \quad\hbox{for any} \,\,  t >0.
$$

\medskip

Hereafter in this section,
 we fix an admissible weight function $m$ satisfying $ m \succ \la v \ra^{2+3/2}$ and consider the spaces $X,Y,Z$ and $X_0$ defined in \eqref{eq:defXYZ}. We also recall the norm $\Nt \cdot \Nt_X$ defined in \eqref{eq:NormT-X}, which is equivalent to $\| \cdot \|_X$.

\medskip

We first prove a stability estimate.

\begin{prop}\label{prop:stab}
There exist some constants $C,K \in (0,\infty)$ such that any solution $f$ to \eqref{eq:pert} satisfies, at least formally, the following differential inequality  
$$
\frac{d}{dt} \Nt f \Nt_X^2 \le  ( C \Nt f \Nt_X - K ) \, \| f \|_{Y}^2.
$$
\end{prop}

\begin{proof}
We write
$$
\ba
\frac12\frac{d}{dt} \Nt f \Nt_X^2 
&= \la\!\la \LL f,f \ra\!\ra_X  + \eta\la Q(f,f), f \ra_X + \int_0^\infty \la S_\LL(\tau) \Pi Q(f,f), S_\LL(\tau) \Pi f \ra_{X_0} \, d\tau \\
&=: T_1+T_2+T_3.
\ea
$$
On the one hand, thanks to \eqref{eq:LLdissipL2-homBIS} in Corollary~\ref{cor:S_L-homBIS} and to Corollary~\ref{cor:Q}, there exist $K, C' > 0$ such that
$$
T_1+T_2 \le -K \| f \|_{Y}^2 +  C' \| f \|_{X} \, \| f \|_{Y}^2.
$$
On the other hand, we have 
$$
\ba
&\int_0^\infty \la S_\LL(\tau) \Pi Q(f,f), S_\LL(\tau) \Pi f \ra_{X_0} \, d\tau  \\
&\qquad 
\le \int_0^\infty \| S_\LL(\tau) \Pi Q(f,f) \|_{X_0} \, \| S_\LL(\tau) \Pi f \|_{X_0} \, d\tau \\
&\qquad 
\lesssim   \| Q(f,f) \|_{Z}  \, \| f \|_{Y} \int_0^\infty \| S_\LL(\tau) \Pi \|_{Z \to X_0} \, \| S_{\LL} (\tau) \Pi \|_{Y \to X_0}\, d\tau  \\
&\qquad \lesssim \| f \|_{X} \, \| f \|_{Y}^2,
\ea
$$
where we have used \eqref{eq:LLregH1-homBIS} in Corollary~\ref{cor:S_L-homBIS} as well as  Corollary~\ref{cor:Q} again in the last line. 
We conclude the proof by gathering theses two estimates. 
\end{proof}

A consequence of the stability estimate in Proposition~\ref{prop:stab} we obtain the spatially homogeneous version of Theorem~\ref{thm:stabNL-inhom}.

\begin{proof}[Proof of Theorem~\ref{thm:stabNL-inhom}. The spatially homogeneous case]
We split the proof into three steps.

\medskip\noindent
{\it Step 1. Uniqueness.} We still denote by $K$ and $C$ the constants 
exhibited in Proposition~\ref{prop:stab} and we set $\eps := (2-\sqrt{2}) K/C$. Consider two solutions $f_1$ and $f_2$ to \eqref{eq:pert} with same initial data such that
\beqn\label{eq:UniqAB}
\forall \, i = 1, 2, \quad 
\Nt f_i \Nt_{L^\infty(0,\infty ; X)}^2+  K \, \| f_i \|_{L^2(0,\infty;Y)} ^2 <  2 \eps^2 .
\eeqn
The difference $ \rho := f_1 - f_2$ satisfies
$$
\partial_t \rho  = \LL \rho  + Q(f_1, \rho) + Q(\rho, f_2), \quad \rho(0) = 0.
$$
 Repeating the same computation as in Proposition~\ref{prop:stab}, we get
$$
\frac{d}{dt} \Nt \rho \Nt_X^2 \le   - K \, \| \rho \|_{Y}^2 + \frac{C}2 \Bigl( (\Nt f_1 \Nt_X + \Nt f_2 \Nt_X)  \| \rho \|_Y^2 +  (\| f_1 \|_Y+\| f_2 \|_Y)  \Nt \rho \Nt_X  \| \rho \|_Y\Bigr).
$$
Integrating in time the above differential inequality and using the Cauchy-Schwarz inequality, we obtain 
\bean
A&:=&   \Nt \rho \Nt_{L^\infty(0,\infty ; X)}^2 + K \| \rho \|_{L^2(0,\infty ; Y)}^2 
\\
&\le& \Nt \rho(0) \Nt_{X}^2 + {C \over 2} \, \Bigl( \Nt f_1  \Nt_{L^\infty(0,\infty ; X)} 
+  \Nt f_2  \Nt_{L^\infty(0,\infty ; X)} \Bigr) \, \| \rho \|_{L^2(0,\infty ; Y)}^2 
\\
&&
+  \frac{C}2\Bigl(  \| f_1 \|_{L^2(0,\infty ; Y)}  +  \| f_2 \|_{L^2(0,\infty ; Y)}  \Bigr)  \Nt \rho \Nt_{L^\infty(0,\infty ; X)} \,   \| \rho \|_{L^2(0,\infty ; Y)}  .
\eean
We assume by contradiction that $\rho \not\equiv 0$. Thanks to estimate \eqref{eq:UniqAB} and the Young inequality, we deduce
\bean
A
&<&  C \,  \bigl( 2\eps^2   \bigr)^{1/2}  \| \rho \|_{L^2(0,\infty ; Y)}^2 
+ C  \, \bigl( { 2\eps^2 \over  K} \bigr)^{1/2} \, \Nt \rho \Nt_{L^\infty(0,\infty ; X)} \,   \| \rho \|_{L^2(0,\infty ; Y)}
\\
&\le& \Nt \rho \Nt_{L^\infty(0,\infty ; X)}^2
+ \Bigl\{ C \,  \sqrt{2} \, \eps + {C^2 \over 2 K} \, \eps^2    \Bigr\} \| \rho \|_{L^2(0,\infty ; Y)}^2
\le A,
\eean
and a contradiction. We conclude that $f_1 = f_2$.  

\medskip\noindent
{\it Step 2. Existence.} The proof follows a classical argument based on an iterative scheme that approximates \eqref{eq:pert} 
(see e.g.\ \cite{MR0363332,Guo} or \cite[Proof of Theorem 5.3]{GMM}) that we sketch for the sake of completeness. 
We consider the iterative scheme
$$
\left\{
\ba
 \partial_t f^n  &= \LL f^n + Q(f^{n-1},f^n) \\
 {f^n}_{|t=0} &= f_0
\ea
\right.
\quad \forall \, n \in \N,
$$
with the convention $f^{-1} = Q(f^{-1},f^0) = 0$ when $n=0$. 
We claim that for $\eps_0 := \Nt f_0 \Nt_X < \eps$, with $\eps$ defined as in Step 1, we may build by an induction argument a sequence $(f^n)_{n \ge 0}$ of solutions of the above scheme such that 
\be\label{eq:schema-stab}
\forall \, n \in \N, \quad A^n :=  \sup_{t \ge 0} \Nt f^n(t) \Nt_{X}^2 + K \int_{0}^\infty \| f^n(t) \|_Y^2 \, dt \le 2 \eps_0^2.
\ee
We only prove the a priori estimate \eqref{eq:schema-stab} by an induction argument, the construction  at each step of the solution of the above linear equation being very classical. We assume that $f^{n-1}$ satisfies \eqref{eq:schema-stab}. 
Repeating the same argument as in Step 1, we have 
%computation as in Proposition~\ref{prop:stab}, we get
%$$
%\frac{d}{dt} \Nt f^n \Nt_X^2 \le   - K \, \| f^n \|_{Y}^2 + \frac{C}2 \Bigl( \Nt f^{n-1} \Nt_X  \| f^n \|_Y^2 +  \| f^{n-1} \|_Y  \Nt f^n \Nt_X  \| f^n \|_Y\Bigr).
%$$
%Integrating in time the above differential inequality and using the Cauchy-Schwarz inequality, we obtain 
\bean
A^n 
&\le& \Nt f_0 \Nt_{X}^2 + {C \over 2} \,  \Nt f^{n-1} \Nt_{L^\infty(0,\infty; X)}  \, \| f^n \|_{L^2(0,\infty;Y)}^2 
\\
&&
+  \frac{C}2 \| f^{n-1}\|_{L^2(0,\infty;Y)}  \, \Nt f^n \Nt_{L^\infty(0,\infty;X)} \, \| f^n \|_{L^2(0,\infty;Y)}.
\eean
Thanks to estimate \eqref{eq:schema-stab} at rank $n-1$ and the Young inequality, as in Step 1 again, we deduce
\bean
A^n 
%&\le& \eps_0^2 +  {C \over 2} \,  \bigl( 2\eps_0^2   \bigr)^{1/2}  \int_0^t \| f^n \|_Y^2 d\tau
%+ \frac{C}2  \, \bigl( { 2\eps_0^2 \over  K} \bigr)^{1/2} \, \sup_{\R_+}  \Nt f^n \Nt_X \Bigl( \int_0^t \| f^n \|_Y^2 d\tau \Bigr)^{1/2}
%\\
&\le& \eps_0^2 +  \frac{1}2  \,   \Nt f^n \Nt_{L^\infty(0,\infty;X)}^2 
+ \Bigl\{ {C \over \sqrt{2} K} \, \eps_0 + {C^2 \over 4 K^2} \, \eps_0^2    \Bigr\} K  \| f^n \|_{L^2(0,\infty;Y)}^2
\\
&\le& \eps_0^2 +  \frac{1}2  \,A^n, 
\eean
from what $f^n$ satisfies \eqref{eq:schema-stab} and the stability of the scheme is proven. 
We now turn to the convergence of the scheme and we define $\rho^n := f^{n+1} - f^n$, for all $n \in \N$, which satisfies
$$
\left\{
\ba
\partial_t \rho^0 &= \LL \rho^0 + Q(f^0, f^1); \\
\partial_t \rho^n &= \LL \rho^n + Q(f^n,\rho^n) + Q(\rho^{n-1}, f^n), \quad \forall\, n \in \N^*; 
\ea
\right.
$$
with ${\rho^n}_{|t=0} = 0$. We define 
$$
\forall \, n \in \N, \quad
B^n  := \sup_{ t \ge 0} \Nt \rho^n(t) \Nt_{X}^2 + K \int_0^\infty \| \rho^n(t) \|_Y^2 \, dt, 
$$
so that in particular $B^0 \le A^1 + A^0 \le (2 \eps_0)^2$. For $n \ge 1$, we compute as in the previous steps 
\bean
B^n   
&\le&   {C \over 2} \,  \Nt f^n  \Nt_{L^\infty(0,\infty;X)} \, \| \rho^n \|_{L^2(0,\infty;Y)}^2  +  \frac{C}2  \| f^n \|_{L^2(0,\infty;Y)}  \,  \Nt \rho^n \Nt_{L^\infty(0,\infty;X)} \,  \| \rho^n \|_{L^2(0,\infty;Y)}
\\
&&
+ {C \over 2} \Bigl\{ \Nt f^n  \Nt_{L^\infty(0,\infty;X)}  \,  \| \rho^{n-1} \|_{L^2(0,\infty;Y)} +   \Nt \rho^{n-1}  \Nt_{L^\infty(0,\infty;X)} \,   \| f^n \|_{L^2(0,\infty;Y)}  \Bigr\} \, \| \rho^n \|_{L^2(0,\infty;Y)}.
\eean
Arguing similarly as in the previous steps by using the Young inequality, estimate \eqref{eq:schema-stab} and choosing $\eps_0 < \sqrt{2} K / (3C)$, 
we easily get 
\bean
B^n   
&\le&   {C_1^2  \over 2} \, \eps_0^2  \, B_{n-1} + {1 \over 2} \, B_n,
\eean
where the constant $C_1 := 3 C / (\sqrt{2} K) $ only depends on $C$ and $K$. That readily implies that 
$$
B^n \le   (C_1 \eps_0)^{2n} \, B_0, \quad \forall \, n \ge 1,
$$
with $C_1 \eps_0 <1$. It then follows that $(f^n)_{n \in \N}$ is a Cauchy sequence in $L^\infty( 0,\infty ; X)$, its limit $f$ is a weak solution to \eqref{eq:pert} and, passing to the limit $n \to \infty$ in \eqref{eq:schema-stab}, $f$ also satisfies \eqref{eq:schema-stab}, from which one deduces \eqref{eq:bound-g-inhom}.

\medskip\noindent
{\it Step 3. Decay.} 
Let $\tilde m$ be an admissible weight function such that $\la v \ra^{2+3/2} \prec \tilde m \prec m$, 
and denote $\tilde X = L^2(\tilde m)$ and $\tilde Y = H^1_*(\tilde m)$.
Thanks to the estimate \eqref{eq:schema-stab} (or \eqref{eq:bound-g-inhom}) and Proposition~\ref{prop:stab} in both spaces $X$ and $\tilde X$, it follows 
$$
\ba
\frac{d}{dt} \Nt f \Nt_X^2 &\le (C \sqrt{2} \eps_0 - K) \| f \|_{Y}^2 
\le -K' \| f \|_{Y}^2 \le 0, \\
\frac{d}{dt} \Nt f \Nt_{\tilde X}^2 &\le (C \sqrt{2} \eps_0 - K) \| f \|_{\tilde Y}^2
\le - K' \| f \|_{\tilde Y}^2 .
\ea
$$
These two estimates together imply (see the proof of Lemma~\ref{lem:SBdecay}) the decay
$$
\Nt f(t) \Nt_{\tilde X} \lesssim \Theta_{m, \tilde m}(t) \, \Nt f_0 \Nt_X.
$$
We hence obtain
$$
\| f(t) \|_{X_0} \lesssim \Theta_m(t) \, \| f_0 \|_X,
$$
where we recall that $\Theta_m$ is defined in \eqref{eq:Theta_m}, and that completes the proof.
\end{proof}

We conclude the section by presenting a proof of our improvement of the speed of convergence to the equilibrium for solutions to the spatially homogenous Landau equation in a non perturbative framework. 

\begin{proof}[Proof of Corollary~\ref{cor:SpeedHomogeneous}]
We claim that for some time $t_0>0$ (smaller than some explicit constant $T>0$) we have 
\be\label{eq:L2m1<eps}
\| f(t_0) \|_{L^2_v (m_1)} \le \eps_0,
\ee
where we denote $m_1 = m^{1/2}  \langle v \rangle^{-9/2}$ and $\eps_0>0$ is given in Theorem~\ref{thm:stabNL-inhom}. Indeed, thanks to \cite{D} there holds
$$
\forall\, t, T>0, \quad \int_{t}^{t+T} \| f(\tau) \|_{L^3_v (\la v \ra^{-3})} \, d \tau \lesssim 1+T,
$$
and from \cite[Theorem 2]{CDH} we have the convergence
$$
\| f(t) \|_{L^1(m)} \lesssim \theta(t), \quad \theta(t) {  = } e^{-\lambda \, t^{\frac{s}{s+|\gamma|}} \, (\log (1+t))^{-\frac{|\gamma|}{s+|\gamma|}}},
$$
for some constant $\lambda >0$. 
Thanks to the interpolation inequality
$$
\| f \|_{L^2_v(m_1)} \le \| f \|_{L^1_v (m)}^{1/4} \, \| f \|_{L^3_v (\la v \ra^{-3})}^{3/4},
$$
we obtain, for any $t > 0$,
$$
\theta(t)^{-1/4}\int_{t}^{t+1} \| f(\tau) \|_{L^2_v (m_1)} \, d \tau \lesssim 
\int_{t}^{t+1} \theta^{-1}(\tau) \| f(\tau) \|_{L^1_v(m)} \, d\tau + \int_{t}^{t+1} \| f(\tau) \|_{L^3_v (\la v \ra^{-3})}\, d \tau \lesssim 1,
$$
which proves \eqref{eq:L2m1<eps}. 
Therefore, observing that $m^{1/3} \prec m_1$ and $m^{1/3}$ is an exponential weight satisfying \eqref{eq:m}, we can apply Theorem~\ref{thm:stabNL-inhom} with $m^{1/3}$ starting from $t_0>0$ and we deduce the convergence
$$
\| f(t) \|_{L^2_v} \lesssim \Theta_{m^{1/3}} (t).
$$
The proof is then complete by remarking that, since $m$ is an exponential weight, 
$\Theta_{m^{1/3}} $ and $\Theta_{m} $ have the same type of asymptotic behaviour (up to a change in the constants in \eqref{eq:Theta_m}).
\end{proof}

%%%%%%%%%%%%%%%%%%%%%%%%%%%%%%%%%%%%%%
\section{The spatially inhomogeneous case}\label{sec:nonH}

In this section, we explain how we may adapt  to the spatially inhomogeneous case the arguments presented in the previous sections.
The novelties come from the facts that: 

\smallskip

(1) We establish a first weak hypocoercivity estimate in the (small) space $ \HH^1_{x,v}(\mu^{-1/2})$ (see \eqref{eq:HH1xv} below); 

\smallskip

(2) We prove a set of weak dissipativity estimates on an appropriate operator $\bar\BB$ and of regularization results on the time functions $(\AA S_{\bar \BB})^{(*n)}$ and $(S_{\bar \BB} \AA)^{(*\ell)}$ in order to transfer the above information to the space $H^2_xL^2_v(m)$, which is suitable for establishing our existence, uniqueness and stability results.

\subsection{The linearized inhomogeneous operator}
We denote by $\bar \LL$ the inhomogeneous linearized Landau operator given by
\be\label{eq:barLL}
\bar \LL := \LL - v \cdot \nabla_x,
\ee
where we recall that $\LL$ is defined in \eqref{eq:lin-op}. We have
$$
\mathrm{ker}(\bar \LL) = \mathrm{span} \{ \mu, v_1 \mu, v_2 \mu, v_3 \mu, |v|^2 \mu   \}
$$
and the projection $\bar \Pi_0$ onto $\mathrm{ker}(\bar \LL)$ is given by
$$
\bar \Pi_0(f) = \left(\int f \, dx \, dv \right) \mu + \sum_{j=1}^3 \left( \int v_j f \, dx\, dv \right) v_j \mu + \left(  \int \frac{|v|^2 - 3}{6} f \, dx\, dv \right) \frac{|v|^2 - 3}{6} \mu .
$$
Hereafter we denote $\bar \Pi := I - \bar \Pi_0$ the projection onto the orthogonal of $\mathrm{ker}(\bar \LL)$.
Recall the factorization for the homogeneous operator $\LL = \AA + \BB$ in \eqref{eq:AB}, then we write 
$$
\bar\LL = \AA + \bar \BB, \quad \bar \BB := \BB - v \cdot \nabla_x  .
$$

\subsection{Functional spaces}
We denote by $L^2_{x,v} = L^2_{x,v} (\T^3_x \times \R^3_v)$ the standard Lebesgue space on $\T^3_x \times \R^3_v$. For a velocity weight function $m=m(v) : \R^3_v \to \R_+$, we then define the weighted Lebesgue spaces $L^2_{x,v}(m)$ and weighted Sobolev spaces $H^n_x  L^2_v(m)$, $n \in \N$, associated to the norms
$$
\| f \|_{L^2_{x,v} (m)} = \| m f \|_{L^2_{x,v}} , \quad
\| f \|_{H^n_x L^2_v(m)}^2 := \sum_{0 \le j \le n} \| \nabla^j_x ( m f) \|_{L^2_{x,v}}^2.
$$
We similarly define the weighted Sobolev space $H^n_{x,v}(m)$, $n\in \N$, through the norm
\be\label{eq:H1xv}
\| f \|_{H^n_{x,v}(m)} : = \| m f \|_{H^n_{x,v}} ,
%:= \| m f \|_{L^2_{x,v}}^2 + \| \nabla_x (m f) \|_{L^2_{x,v}}^2 + \| \nabla_v (m f) \|_{L^2_{x,v}}^2.
\ee
where $H^n_{x,v} = H^n_{x,v} (\T^3_x \times \R^3_v)$ denotes the usual Sobolev space on $\T^3_x \times \R^3_v$.
We also define the space $\HH^1_{x,v} (m) $, { for an admissible weight $m$}, as the space associated to the norm defined by 
\be\label{eq:HH1xv}
\| f \|_{\HH^1_{x,v} (m)}^2 := \| m f \|_{L^2_{x,v}}^2 + \|  \nabla_x (m f) \|_{L^2_{x,v}}^2 
+ \| \la v \ra^{\alpha} \,  \nabla_v ( m f ) \|_{L^2_{x,v}}^2 , 
\ee
with
\be\label{eq:alpha(m)}
\alpha := \alpha(m) := \max\left\{ \gamma+\sigma , \frac{\gamma}{2} + \frac{\sigma}{4}   \right\} <0.
\ee
We easily observe that 
\be\label{eq:H1xv<HH1xv}
H^{1}_{x,v} ( m) \subset \HH^1_{x,v}(m) \subset H^{1}_{x,v} (\la v \ra^{\alpha} m) ,
\ee
and also that, for any $\gamma \in [-3,-2)$,
$$
\alpha = \frac{\gamma}{2} + \frac{\sigma}{4} \; \text{ if }\; \sigma \in [0,4/3], \quad
\alpha = \gamma+2 \; \text{ if }\; \sigma=2,
$$
where we recall that $\sigma$ has been defined at the beginning of Section~\ref{subsec:prelim}. 
We remark that we shall use the spaces $\HH^1_{x,v} (m)$ (instead of $H^1_{x,v} (m)$) in order to obtain weakly dissipative estimates for $\bar \BB$, and the reason for that will be explained in Lemma~\ref{lem:barBB1}.

Recall the space $H^1_{v,*}(m)$ defined in \eqref{def:H1*}, then we define the space $H^2_x (H^1_{v,*}(m))$ associated to the norm
\be\label{eq:def-HnxH1v*}
\ba
\| f \|_{H^2_x (H^1_{v,*}(m))}^2 
&:= \sum_{0\le j \le 2} \| \nabla_x^j f \|_{L^2_x (H^1_{v,*}(m))}^2 
:= \sum_{0\le j \le 2} \int_{\T^3_x} \| \nabla_x^j f \|_{ H^1_{v,*}(m)}^2 .
\ea
\ee
 When furthermore $m$ is a polynomial weight function,  we also define the negative weighted Sobolev space $H^2_x (H^{-1}_{v,*}(m))$ in duality with $H^2_x (H^1_{v,*}(m))$ with respect to the $H^2_x L^2_v(m)$ duality product, more precisely 
\bean
\| f \|_{H^2_x (H^{-1}_{v,*} (m))}
&:=& \sup_{\| \phi \|_{H^2_x (H^1_{v,*}(m))} \le 1} \la f , \phi \ra_{H^2_x L^2_v (m)}
\\
&:=& \sup_{\| \phi \|_{H^2_x (H^1_{v,*}(m))} \le 1} \, \sum_{0 \le j \le 2}  \la \nabla^j_x (m f) , \nabla^j_x  (m\phi) \ra_{L^2_{x,v} },
\eean
and observe that $\| f \|_{H^2_x (H^{-1}_{v,*} (m))} = \| m f \|_{H^2_x (H^{-1}_{v,*})}$.

\subsection{Weak coercivity estimate of $\bar\LL$} 
Starting from the weak coercivity estimate \eqref{eq:LLsg} for the homogeneous linearized operator $\LL$ in $L^2_v(\mu^{-1/2})$, we can exhibit an equivalent norm to the usual norm in $\HH^1_{x,v} (\mu^{-1/2})$ such that $\bar \LL$ is weakly coercive related to that norm. Our method of proof follows the method developed in \cite{MouNeu} for proving (strong) coercivity estimate and then spectral gap estimate in the case of the linearized Landau equation for harder potentials. We also refer to \cite{Guo,Vi-hypo} where related arguments have been introduced.

\begin{lem}\label{lem:GenCoerc}
There exists a Hilbert norm $\| \cdot \|_{\widetilde{\HH}^1_{x,v} (\mu^{-1/2})}$ (which associated scalar product is denoted by $\la \cdot , \cdot \ra_{\widetilde{\HH}^1_{x,v} (\mu^{-1/2})}$) equivalent to $\| \cdot \|_{\HH^1_{x,v} (\mu^{-1/2})}$ such that, for any $f \in \HH^1_{x,v} (\mu^{-1/2})$, there holds
\beqn\label{eq:GenCoerc}
\la \bar\LL f , f \ra_{\widetilde{\HH}^1_{x,v} (\mu^{-1/2})} 
\lesssim -  \| \bar \Pi f \|_{\widetilde{\HH}^1_{x,v} ( \la v \ra^{(\gamma+2)/2} \mu^{-1/2}) }^2 .
\eeqn

\end{lem}

\begin{proof} 
We only sketch the proof presenting the main steps, and we refer to \cite{MouNeu} for more details. We define 
$$
L h = \mu^{-1/2} \, \LL (\mu^{1/2} h). 
$$
Observe that $f = \mu^{1/2} h$ satisfies  $L h = \mu^{-1/2} \LL f$ and $\la L h, h \ra_{L^2_v} = \la \LL f, f \ra_{L^2_v(\mu^{-1/2})}$. Following \cite[Section 2]{Guo} we can decompose $L = A +K$ such that the following properties holds:

\medskip\noindent
$(i)$ Generalized coercivity estimate (see \eqref{eq:LLsg}): there holds, for some constant $\lambda >0$,
$$
\la L h, h \ra_{L^2_v} \le - \lambda \| h - \Pi_L h \|_{H^1_{v,**}}^2,
$$
where $\Pi_L$ is the projection onto $\mathrm{ker}(L)$ in $L^2_v$, and we denote 
$$
\| h  \|_{H^1_{v,**}(\omega)}^2 := \| \la v \ra^{\frac{\gamma+2}{2}} \, h \|_{L^2_v(\omega)}^2 +
\| \la v \ra^{\frac{\gamma}{2}} \, \widetilde \nabla_v h \|_{L^2_v(\omega)}^2.
$$

\medskip\noindent
$(ii)$ \cite[Lemma 5]{Guo}: For $\theta \in \R$ and $\delta>0$,  there holds
$$
\la \la v \ra^{2 \theta} K h, h \ra_{L^2_v} 
\lesssim \delta \| h \|_{H^1_{v,**}(\la v \ra^\theta)}^2
+C(\delta) \| h \|_{L^2_v(\la v \ra^\theta)}^2,
$$
and also
$$
\la \la v \ra^{2\theta} L h_1, h_2 \ra_{L^2_v} \lesssim \| h_1 \|_{H^1_{v,**}(\la v \ra^\theta)} \, \| h_2 \|_{H^1_{v,**}(\la v \ra^\theta)}.
$$

\medskip\noindent
$(iii)$ \cite[Lemma 6]{Guo}: For $\theta \in \R$ and $\eta>0$, there holds ({ for some $\lambda,C>0$})
$$
\la \la v \ra^{2\theta} \nabla_v (Ah), \nabla_v h \ra_{L^2_v}
\le - \lambda \| \nabla_v h \|_{H^1_{v,**}(\la v \ra^\theta)}^2 + \eta C\| h \|_{H^1_{v,**}(\la v \ra^\theta)}^2 + \eta^{-1} C \| \mu h \|_{L^2_v}^2,
$$
and also
$$
\la \la v \ra^{2\theta} \nabla_v (K h), \nabla_v h \ra_{L^2_v}
\lesssim \eta \| h \|_{L^2_x H^1_{v,**} (\la v \ra^{\gamma+2})}^2
+ \eta \| \nabla_v h \|_{L^2_x H^1_{v,**} (\la v \ra^{\gamma+2})}^2
+ \eta^{-1} \| \mu h \|_{L^2_{x,v}}^2.
$$

\smallskip
We now consider the inhomogeneous operator $\bar L := L - v \cdot \nabla_x$, we denote $\Pi_{\bar L}$ the projection onto $\mathrm{ker} (\bar L)$ in $L^2_{x,v}$ and we consider a solution $h$ to the evolution equation $\partial_t h = \bar L h$ with initial datum $h(0) = h_0 \in \mathrm{ker}(\bar L)^\perp$. Thanks to $(i)$ and the fact that $\nabla_x$ commutes with $\bar L$, we immediately have
$$
\frac12 \frac{d}{dt}\Big( \| h \|_{L^2_{x,v}}^2 
+  \| \nabla_x h \|_{L^2_{x,v}}^2 \Big) \le - \lambda \| h - \Pi_L h \|_{L^2_x (H^1_{v,**})}^2
 - \lambda \| \nabla_x h - \Pi_L (\nabla_x h) \|_{L^2_x (H^1_{v,**})}^2.
$$
We next look to the $v$-derivative. 
 
We first compute
$$
\ba
\frac12 \frac{d}{dt} \| \la v \ra^{\gamma + 2}  \, \nabla_v h \|_{L^2_{x,v}}^2
& = \la \la v \ra^{2(\gamma+2)} \nabla_v (Kh), \nabla_v h \ra_{L^2_{x,v}}
+ \la \la v \ra^{2(\gamma+2)} \nabla_v (Ah), \nabla_v h \ra_{L^2_{x,v}}\\
&\quad - \la \la v \ra^{2(\gamma+2)} v \cdot \nabla_x (\nabla_v h), \nabla_v h \ra_{L^2_{x,v}}
-\la \la v \ra^{2(\gamma+2)} \nabla_x h, \nabla_v h \ra_{L^2_{x,v}} \\
& =: T_1 + T_2 + T_3 + T_4.
\ea
$$
Terms $T_1$ and $T_2$ satisfy estimates of point $(iii)$ above, moreover, we easily observe that $T_3=0$ and we also get 
$$
T_4 \lesssim \eta \| \nabla_v h \|_{L^2_{x,v} (\la v \ra^{3\gamma/2 + 3})}^2 
+ \eta^{-1} \| \nabla_x h \|_{L^2_{x,v}(\la v \ra^{\gamma/2 + 1})}^2 .
$$
We know observe that
$$
\ba
&\| \mu h \|_{L^2_{x,v}} \lesssim \| h \|_{L^2_{x,v} (\la v \ra^{3\gamma/2 + 3})}, \quad
\| \nabla_v h \|_{L^2_{x,v} (\la v \ra^{3\gamma/2 + 3})} 
\lesssim \| \nabla_v h \|_{L^2_x H^1_{v,**} (\la v \ra^{\gamma+2})} , \\
&\quad
\| h \|_{L^2_x H^1_{v,**} (\la v \ra^{\gamma+2})} \lesssim 
\| h \|_{L^2_{x,v} (\la v \ra^{3\gamma/2 + 3})}
+ \| \nabla_v h \|_{L^2_x H^1_{v,**} (\la v \ra^{\gamma+2})}.
\ea
$$
Therefore, putting together previous estimates and taking $\eta >0$ small enough, we already obtain, for (other) constants $\lambda, C >0$,
$$
\ba
\frac{d}{dt} \| \la v \ra^{\gamma + 2}  \, \nabla_v h \|_{L^2_{x,v}}^2 
&\le - \lambda \| \nabla_v h \|_{L^2_x (H^1_{v,**}(\la v \ra^{\gamma+2}))}^2 
+ \eta^{-1} C  \|  h \|_{L^2_{x,v} (\la v \ra^{3\gamma/2+3})}^2 
+ \eta^{-1} C \| \nabla_x h \|_{L^2_{x,v}(\la v \ra^{\gamma/2+1})}^2.
\ea
$$
We also compute the evolution of the mixed term
$$
\ba
\frac{d}{dt} \la \la v \ra^{\gamma+2} \, \nabla_x h , \nabla_v h \ra_{L^2_{x,v}} 
&= - \| \la v \ra^{\frac{\gamma+2}{2}} \, \nabla_x h \|_{L^2_{x,v}}^2 
+ 2 \la \la v \ra^{\gamma+2}  \, \nabla_x L h , \nabla_v h \ra_{L^2_{x,v}} \\
&\quad + \la (\nabla_v \la v \ra^{\gamma+2}) \nabla_x L h , h \ra_{L^2_{x,v}},
\ea
$$
Thanks to $(i)$ and $\nabla_x L h = L(\nabla_x h - \Pi_L(\nabla_x h))$, for any $\eta>0$,  it follows that 
$$
\ba
&\la \la v \ra^{\gamma+2} \, \nabla_x L h , \nabla_v h \ra_{L^2_{x,v}}
+ \la (\nabla_v \la v \ra^{\gamma+2}) \, \nabla_x L h , h \ra_{L^2_{x,v}} \\
&\qquad \lesssim \eta^{-1} \| \nabla_x h - \Pi_L(\nabla_x h) \|_{L^2_x (H^1_{v,**} (\la v \ra^{(\gamma+2)/2}))}^2 \\
&\qquad\quad + \eta \| \nabla_v h \|_{L^2_x ( H^1_{v,**} (\la v \ra^{(\gamma+2)/2}))}^2
+\eta \| h \|_{L^2_x ( H^1_{v,**} (\la v \ra^{(\gamma+2)/2}))}^2
\ea
$$
We finally introduce the norm 
$$
\Nt h \Nt^2 := \| h \|_{L^2_{x,v}}^2 + \alpha_1 \| \nabla_x h \|_{L^2_{x,v}}^2
+ \alpha_2 \| \la v \ra^{\gamma+2} \, \nabla_v h \|_{L^2_{x,v}}^2
+ \alpha_3 \la \la v \ra^{\gamma+2} \, \nabla_x h , \nabla_v h \ra_{L^2_{x,v}},
$$
for positive constants $\alpha_i$ with $\alpha_3< 2 \sqrt{\alpha_1 \, \alpha_2}$, so that $\Nt h \Nt^2$ is equivalent to 
$$
\| h \|_{L^2_{x,v}}^2 + \| \nabla_x h \|_{L^2_{x,v}}^2 + \| \la v \ra^{\gamma+2} \nabla_v h \|_{L^2_{x,v}}^2.
$$
Observe that $\Pi_L h$ has zero mean on the torus $\T^3$ hence Poincar\'e's inequality implies
$$
\| \Pi_L h \|_{L^2_{x,v}(\omega)}^2 + \| \Pi_L h \|_{L^2_{x}(H^1_{v,**}(\omega))}^2 
\lesssim \| \nabla_x h \|_{L^2_{x,v}(\omega)}^2,
$$
and splitting $h = (h - \Pi_L h) + \Pi_L h$ we get
$$
\ba
& \|  h \|_{L^2_{x,v} (\la v \ra^{3\gamma/2+3})}^2 \lesssim 
\|  h - \Pi_L h \|_{L^2_{x,v} (\la v \ra^{3\gamma/2+3})}^2 
+ \| \nabla_x h \|_{L^2_{x,v}(\la v \ra^{3\gamma/2+3})}^2 \\
& \| h \|_{L^2_x ( H^1_{v,**} (\la v \ra^{(\gamma+2)/2}))}^2 
\lesssim \| h  - \Pi_L h\|_{L^2_x ( H^1_{v,**} (\la v \ra^{(\gamma+2)/2}))}^2
+ \| \nabla_x h \|_{L^2_{x,v}(\la v \ra^{3\gamma/2+3})}^2 .
\ea
$$
Finally, gathering previous estimates we obtain
$$
\ba
\frac{d}{dt} \Nt h \Nt^2
&\le - \lambda \| h - \Pi_L h \|_{L^2_x (H^1_{v,**})}^2
- \alpha_1 \lambda \| \nabla_x h - \Pi_L (\nabla_x h) \|_{L^2_x (H^1_{v,**})}^2 \\
&\quad -\alpha_2 \lambda  \| \nabla_v h \|_{L^2_x (H^1_{v,**}(\la v \ra^{\gamma+2}))}^2 
- \alpha_3 \| \nabla_x h \|_{L^2_{x,v}(\la v \ra^{(\gamma+2)/2})}^2\\
&\quad + \alpha_2 \eta^{-1} C \|  h - \Pi_L h\|_{L^2_{x,v} (\la v \ra^{3\gamma/2+3})}^2
+ \alpha_2 \eta^{-1} C \| \nabla_x h \|_{L^2_{x,v}(\la v \ra^{3\gamma/2+3})}^2 \\
&\quad + \alpha_2 \eta^{-1} C \| \nabla_x h\|_{L^2_{x,v} (\la v \ra^{(\gamma+2)/2})}^2 
+ \alpha_3 \eta C\| h - \Pi_L h \|_{L^2_x (H^1_{v,**} (\la v \ra^{(\gamma+2)/2}))}^2
+ \alpha_3 \eta C\| \nabla_x h \|_{L^2_{x,v} (\la v \ra^{(\gamma+2)/2})}^2 \\
&\quad + \alpha_3 \eta C\| \nabla_v h \|_{L^2_x (H^1_{v,**} (\la v \ra^{(\gamma+2)/2}))}^2
+ \alpha_3 \eta^{-1} C \| \nabla_x h - \Pi_L(\nabla_x h) \|_{L^2_x (H^1_{v,**} (\la v \ra^{(\gamma+2)/2}))}^2 .
\ea
$$

We choose the constants $\alpha_i, \eta >0$ small enough, and we get
$$
\ba
\frac{d}{dt} \Nt h \Nt^2
&\lesssim - \| h - \Pi_L h \|_{L^2_x (H^1_{v,**})}^2
 -  \alpha_1 \| \nabla_x h - \Pi_L (\nabla_x h) \|_{L^2_x (H^1_{v,**})}^2 \\
&\quad  - \alpha_3 \| \nabla_x h \|_{L^2_{x,v}(\la v \ra^{(\gamma+2)/2})}^2 
-  \alpha_2 \| \nabla_v h \|_{L^2_x (H^1_{v,**}(\la v \ra^{\gamma+2}))}^2  .
\ea
$$
Because
$\Pi_{\bar L} h = 0$, the function $\Pi_L h$ has zero mean on the torus $\T^3_x$ and Poincar\'e's inequality implies
$$
\ba
\| h  \|_{L^2_{x,v} (\la v \ra^{(\gamma+2)/2})}^2 
&\lesssim \| h - \Pi_L h \|_{L^2_{x,v} (\la v \ra^{(\gamma+2)/2})}^2 + \frac{\alpha_3}{2} \| \nabla_x h \|_{L^2_{x,v} (\la v \ra^{(\gamma+2)/2})}^2.
\ea
$$
We put together the two last estimates and we get 
$$
\ba
\frac{d}{dt} \Nt h \Nt^2
&\lesssim - \| h  \|_{L^2_{x,v} (\la v \ra^{(\gamma+2)/2})}^2
-   \| \nabla_x h  \|_{L^2_{x,v} (\la v \ra^{(\gamma+2)/2})}^2 
-   \| \nabla_v h \|_{L^2_{x,v} (\la v \ra^{3(\gamma+2)/2})}^2 
\\
&\lesssim - \Nt \la v \ra^{(\gamma+2)/2} h \Nt^2 .
\ea
$$
Coming back to the function $f = \mu^{1/2} h$ and defining 
$$
\| f \|_{\widetilde{\HH}^1_{x,v} (\mu^{-1/2})} := \Nt \mu^{-1/2} f \Nt,
$$
 we have $\partial_t f = \bar \LL f$ and 
$$
\ba
\la \bar \LL f , f \ra_{\widetilde{\HH}^1_{x,v} (\mu^{-1/2})}
= \frac{d}{dt} \| f \|_{\widetilde{\HH}^1_{x,v} (\mu^{-1/2})}^2 
\lesssim - \| f \|_{\widetilde{\HH}^1_{x,v} (\la v \ra^{(\gamma+2)/2}  \mu^{-1/2})}^2,
\ea
$$
from which \eqref{eq:GenCoerc} immediately follows. 
\end{proof}

\subsection{Weak dissipativity properties on $\bar\BB$}
We prove in this section weak dissipativity properties of $\bar \BB$ using the analogous results already proven in Lemmas \ref{lem:BB} and \ref{lem:BB2} for the homogeneous operator $\BB$.

\begin{lem}\label{lem:barBB}
Let $m$ be an admissible weight function such that $ m \succ \la v \ra^{(\gamma+3)/2} $ and $n \in \N$. There exist $M,R>0$ large enough such that $\bar \BB$ is weakly dissipative in $H^n_x L^2_{v}(m)$ in the following sense:

\medskip

$\bullet$ If $m \prec \mu^{-1/2}$, there holds
\be\label{eq:barBB-L2}
\ba
\la \bar \BB f , f \ra_{H^n_x L^2_{v}(m)} &\lesssim 
-   \| \la v \ra^{\frac{\gamma}{2}} \,   \widetilde  \nabla_v  f  \|_{H^n_x L^2_{v}(m)}^2
-  \| \la v \ra^{\frac{\gamma}{2}} \,   \widetilde  \nabla_v ( m f ) \|_{H^n_x L^2_{v}}^2
-  \| \la v \ra^{\frac{\gamma+\sigma}{2}} f  \|_{H^n_x L^2_{v}(m)}^2.
\ea
\ee

\medskip

$\bullet$ If $\mu^{-1/2} \preceq m \prec \mu^{-1}$, there holds 
\be\label{eq:barBB-L2BIS}
\ba
\la \bar \BB f , f \ra_{H^n_x L^2_{v}(m)} &\lesssim 
-  \| \la v \ra^{\frac{\gamma}{2}} \,   \widetilde  \nabla_v ( m f ) \|_{H^n_x L^2_{v}}^2
-  \| \la v \ra^{\frac{\gamma+\sigma}{2}} f  \|_{H^n_x L^2_{v}(m)}^2.
\ea
\ee

\end{lem}

\begin{proof} 
Since the operator $\bar \BB$ commutes with $\nabla_x$ we only need to treat the case $n=0$.
The proof follows the same argument as for the homogeneous case in Lemma~\ref{lem:BB} thanks to the divergence structure of the transport operator.
\end{proof}

We define the operator 
\be\label{def:barBBm}
\bar\BB_m g = m \bar\BB(m^{-1} g) = \BB_m g - v \cdot \nabla_x g,
\ee
where we recall that $\BB_m$ is defined in \eqref{eq:BBm}, as well as its formal adjoint operator $\bar \BB^*_m$ that verifies
\be\label{def:barB*m}
\ba
\bar \BB^*_m \phi &= \BB^*_m \phi + v \cdot \nabla_x \phi   ,
\ea
\ee
with $\BB^*_m$ defined in \eqref{def:B*m}. Observe that if $f$ satisfies $\partial_t f = \bar \BB f$, then $g=mf$ satisfies $\partial_t g = \bar\BB_m g$ and $\la \bar \BB f, f \ra_{\HH^1_{x,v} (m)} = \la \bar\BB_m g , g \ra_{\HH^1_{x,v}}$. 
Moreover, we have by duality
$$
\forall \, t \ge 0, \quad 
\la S_{\bar \BB_m}(t) g , \phi \ra_{H^n_x L^2_v} = \la g , S_{\bar \BB^*_m}(t) \phi \ra_{H^n_x L^2_v}.
$$

\begin{lem}\label{lem:barBB2} 
Let $m, \omega$ be admissible { polynomial} weight functions such that $m \succ \la v \ra^{(\gamma+3)/2}$, $ 1 \preceq \omega \prec m \la v \ra^{-(\gamma+3)/2}$ and $n \in \N$. 
We can choose $M,R$ large enough such that $\bar \BB^*_m$ is weakly dissipative in $H^n_x L^2_{v}$ in the sense
$$ 
\la \bar \BB^*_m \phi, \phi \ra_{H^n_x L^2_{v}(\omega) } 
\lesssim -   \| \la v \ra^{\frac{\gamma}{2}} \widetilde \nabla_v \phi \|_{H^n_x L^2_{v} (\omega)}^2 
- \| \la v \ra^{\frac{\gamma+\sigma}{2}} \phi  \|_{H^n_x L^2_{v}(\omega)}^2.
$$

\end{lem}

\begin{proof}
The proof follows the same arguments as in the proof of Lemma \ref{lem:BB2}, thanks to the divergence structure of the transport operator and since $\nabla_x$ commutes with $\BB^*_m$. 
\end{proof}

We turn now to weakly dissipative properties of $\bar \BB$ in the spaces $\HH^1_{x,v} (m)$ defined in \eqref{eq:HH1xv}.

\begin{lem}\label{lem:barBB1}
Let $m$ be an admissible weight function such that $ m \succ \la v \ra^{(\gamma+3)/2} $. For any $\eta>0$, we define the norm 
$$
\| f \|_{\widetilde \HH^1_{x,v}(m)}^2 
:= \| m f \|_{L^2_{x,v}}^2 +  \|  \nabla_x (mf) \|_{L^2_{x,v}}^2 
+ \eta  \| \la v \ra^{\alpha} \,  \nabla_v (  mf) \|_{L^2_{x,v}}^2 ,
$$
and its associated scalar product $\la  \cdot , \cdot \ra_{\widetilde\HH^1_{x,v}(m)}$,
which is equivalent to the standard $\HH^1_{x,v}(m)$-norm defined in \eqref{eq:HH1xv}.
There exist $M,R,\eta>0$ such that $\bar \BB$ is weakly dissipative in $ \HH^1_{x,v}(m) $ in the sense
$$
\ba
\la \bar \BB f , f \ra_{\widetilde\HH^1_{x,v}(m)}
&\lesssim 
- \| f \|_{\widetilde\HH^1_{x,v} (m \la v \ra^{(\gamma+\sigma)/2})}^2 
- \| \la v \ra^{\frac{\gamma}{2}} \widetilde \nabla_v  (mf) \|_{L^2_{x,v}}^2 \\
&\quad
-    \| \la v \ra^{\frac{\gamma}{2}} \widetilde \nabla_v (  \nabla_x (mf) )\|_{L^2_{x,v}}^2 
-  \eta  \| \la v \ra^{\frac{\gamma}{2} + \alpha}  \, \widetilde \nabla_v (\nabla_v  (m f)) \|_{L^2_{x,v}}^2.
\ea
$$

\end{lem}

\begin{proof}
We remark that we have introduced the spaces \eqref{eq:HH1xv}, in which the term $\nabla_v (mf)$ has a weight $\la v \ra^{\alpha} $ with $\alpha<0$, in order to treat the terms coming from the derivative in the $v$-variable of the transport operator. { In what follows we shall denote $\lambda,C>0$ positive constants that can change from line to line.}

\smallskip

For the sake of simplicity, we shall equivalently prove that
$$
\ba
&\frac{d}{dt} \Big( \| g_{\bar \BB_m} \|_{L^2_{x,v}}^2 +  \| \nabla_x g_{\bar \BB_m} \|_{L^2_{x,v}}^2  +  \eta  \| \la v \ra^{\alpha} \nabla_v  g_{\bar \BB_m} \|_{L^2_{x,v}}^2  \Big) \\
&\quad \lesssim -  \Big( \| \la v \ra^{\frac{\gamma+\sigma}{2}} g_{\bar \BB_m} \|_{L^2_{x,v}}^2 +   \|\la v \ra^{\frac{\gamma+\sigma}{2}} \nabla_x g_{\bar \BB_m} \|_{L^2_{x,v}}^2  +  \eta  \|  \la v \ra^{\frac{\gamma+\sigma}{2} + \alpha} \, \nabla_v  g_{\bar \BB_m} \|_{L^2_{x,v}}^2  \Big)\\
&\quad\quad -  \| \la v \ra^{\frac{\gamma}{2}} \widetilde \nabla_v g_{\bar \BB_m} \|_{L^2_{x,v}}^2 
-  \| \la v \ra^{\frac{\gamma}{2}} \widetilde \nabla_v (\nabla_x g_{\bar \BB_m} ) \|_{L^2_{x,v}}^2
-  \eta \| \la v \ra^{\frac{\gamma}{2} + \alpha} \, \widetilde \nabla_v ( \nabla_v  g_{\bar \BB_m} )\|_{L^2_{x,v}}^2, 
\ea
$$
for any solution $g_{\bar \BB_m}$ to the equation $\partial_t g_{\bar \BB_m} = \bar \BB_m g_{\bar \BB_m}$, so that, with $g_{\bar \BB_m} = m f_{\bar \BB}$, $f_{\bar \BB}$ is a solution to $\partial_t f_{\bar \BB} = \bar \BB f_{\bar \BB}$. We now use the shorthand $g = g_{\bar \BB_m}$ and split the proof into three steps.

\medskip\noindent
{\it Step 1.}
We first obtain from Lemma~\ref{lem:barBB} (for $M,R>0$ large enough)
\be\label{eq:g}
\frac{d}{dt} \| g \|_{L^2_{x,v}}^2 
\lesssim -  \| \la v \ra^{\frac{\gamma}{2}}    \widetilde \nabla_v g \|_{L^2_{x,v}}^2 
-  \| \la v \ra^{\frac{\gamma+\sigma}{2}} g \|_{L^2_{x,v}}^2
\ee
and
\be\label{eq:gx}
\frac{d}{dt} \| \nabla_x g \|_{L^2_{x,v}}^2 
\lesssim -  \| \la v \ra^{\frac{\gamma}{2}} \widetilde \nabla_v  ( \nabla_x g) \|_{L^2_{x,v}}^2 
-  \| \la v \ra^{\frac{\gamma+\sigma}{2}} \nabla_x g \|_{L^2_{x,v}}^2 .
\ee

\medskip\noindent
{\it Step 2.}
We write 
$$
\frac12\frac{d}{dt} \| \la v \ra^{\alpha} \nabla_v g \|_{L^2_{x,v}}^2
= \int_{x,v} \nabla_v (\BB_m g ) \cdot \nabla_v g \, \la v \ra^{2 \alpha} 
- \int_{x,v}  \nabla_x g \cdot \nabla_v g \, \la v \ra^{2 \alpha}  .
$$
where we have
$$
\nabla_v (\BB_m g ) = \BB_m(\nabla_v g) + (\nabla_v \bar a_{ij}) \partial_{ij} g + (\nabla_v \beta_j) \partial_j g + (\nabla_v \delta - M \nabla_v \chi_R) g .
$$
We first compute 
$$
\ba
\int_{x,v} \nabla_v (\BB_m g) \cdot \nabla_v g \, \la v \ra^{2\alpha}
&= : T_1 + T_2 + T_3 + T_4 ,
\ea
$$
where
$$
\ba
&T_1 = \int (\BB_m \nabla_v g) \, \cdot \nabla_v g \, \la v \ra^{2\alpha}   , \quad
&T_2 =  \int (\nabla_v \bar a_{ij}) \, \partial_{ij} g \, \nabla_v g \, \la v \ra^{2\alpha},  \\
&T_3 =  \int (\nabla_v  \beta_j) \, \partial_j g \, \nabla_v g \, \la v \ra^{2\alpha}  , \quad
&T_4 =  \int (\nabla_v  \delta -M \nabla_v \chi_R) g \, \nabla_v g \, \la v \ra^{2\alpha}. \ea
$$
From Lemma~\ref{lem:barBB}, we have 
$$
\ba
T_1
&\le  - { \lambda} \| \la v\ra^{\frac{\gamma}{2} + \alpha} \, \widetilde \nabla_v (\nabla_v g) \|_{L^2}^2 
+  \int \{ \tilde\zeta_{m} - M\chi_R \}  \, |\nabla_v g|^{2} \, \la v \ra^{2 \alpha}  .
\ea
$$
Terms $T_{3}$ and $T_{4}$ are easy to estimate. As in the proof of Lemma~\ref{lem:varphi}, we can compute explicitly $\beta_j(v)$ and $\delta(v)$, thus we easily deduce
$$
|\nabla_v \beta_j (v)| + |\nabla_v \delta(v)| \lesssim \la v \ra^{\gamma+\sigma-1}.
$$
Therefore
$$
T_3 + T_4  \lesssim \int \{ \la v \ra^{\gamma+\sigma-1} +  \frac{M}{R} {\mathbf 1}_{R \le |v| \le 2R} \} \, |\nabla_v g|^{2} \, \la v \ra^{2 \alpha}
+ \int \{ \la v \ra^{\gamma+\sigma-1} +  \frac{M}{R} {\mathbf 1}_{R \le |v| \le 2R} \} \, g^2 \, \la v \ra^{2 \alpha}.
$$
Thanks to Lemma~\ref{lem:varphi}, for $M,R>0$ large enough, we have
$$%\be\label{eq:T1345}
T_1 + T_3 + T_4  \le - { \lambda}\| \la v\ra^{\frac{\gamma}{2} + \alpha} \, \widetilde \nabla_v (\nabla_v g) \|_{L^2_{x,v}}^2 
- { \lambda}\| \la v \ra^{\frac{\gamma+\sigma}{2} + \alpha} \, \nabla_v g \|_{L^2_{x,v}}^2
+ { C}\| \la v \ra^{\frac{\gamma+\sigma-1}{2} + \alpha} \, g \|_{L^2_{x,v}}^2 .
$$%\ee
Performing an integration by parts, we first obtain 
$$
\ba
T_2 &= - \int (\nabla_v \bar b_{j}) \, \partial_j g \, \nabla_v g \, \la v \ra^{2 \alpha}
- \int (\nabla_v \bar a_{ij}) \, \partial_j g \, \partial_i \nabla_v g \, \la v \ra^{2 \alpha}
- \int (\nabla_v \bar a_{ij}) \, \partial_j g \,  \nabla_v g \, \partial_i \la v \ra^{2 \alpha}\\
& =: U+V+W.
\ea
$$
Thanks to Lemma~\ref{lem:bar-aij}, we easily have
$$
U +W \lesssim \| \la v \ra^{\frac{\gamma}{2}+\alpha} \, \nabla_v g \|_{L^2_{x,v}}^2.
$$
We make another integration by parts for $V$ (now with respect to $\nabla_v$), we get
$$
\ba
V &= \int (\Delta_v \bar a_{ij}) \, \partial_i g \, \partial_j g \,\la v \ra^{2 \alpha}
+\int (\nabla_v \bar a_{ij}) \, \partial_i g \, \partial_j \nabla_v g \,\la v \ra^{2 \alpha}
+\int (\nabla_v \bar a_{ij}) \, \partial_i g \, \partial_j  g \, \nabla_v \la v \ra^{2 \alpha},
\ea
$$
and we recognize that the middle term is equal to $-V$, so that
$$
V = \frac12 \int (\Delta_v \bar a_{ij}) \, \partial_i g \, \partial_j g \,\la v \ra^{2 \alpha}
+ \frac12 \int (\nabla_v \bar a_{ij}) \, \partial_i g \, \partial_j  g \, \nabla_v \la v \ra^{2 \alpha}
\lesssim \| \la v \ra^{\frac{\gamma}{2} + \alpha} \, \nabla_v g \|_{L^2_{x,v}}^2.
$$
We finally obtain (for $M,R>0$ large enough)
\be\label{eq:gv}
\ba
\int_{x,v} \nabla_v (\BB_m g) \cdot \nabla_v g \, \la v \ra^{2\alpha}  
&\le 
- { \lambda} \| \la v \ra^{\frac{\gamma}{2} + \alpha} \, \widetilde \nabla_v (\nabla_v g) \|_{L^2_{x,v}}^2 
- { \lambda} \| \la v \ra^{\frac{\gamma+\sigma}{2} + \alpha} \, \nabla_v g \|_{L^2_{x,v}}^2\\ 
&\quad
+  { C} \| \la v \ra^{\frac{\gamma+\sigma-1}{2} + \alpha } \, g \|_{L^2_{x,v}}^2
+  { C} \| \la v \ra^{\frac{\gamma}{2} + \alpha} \, \widetilde \nabla_v  g \|_{L^2_{x,v}}^2 .
\ea
\ee
By Cauchy-Schwarz inequality, we also get
\be\label{eq:gxgv-alpha}
\int_{x,v}  \nabla_x g \cdot \nabla_v g \, \la v \ra^{2\alpha}   
\le  C \eta^{-1/2} \| \la v \ra^{\frac{\gamma+\sigma}{2}} \nabla_x g \|_{L^2_{x,v}}^2
 + C\eta^{1/2}  \| \la v \ra^{2\alpha - \frac{\gamma+\sigma}{2}} \, \nabla_v g \|_{L^2_{x,v}}.
\ee
Remark that the first term in the right-hand side of \eqref{eq:gxgv-alpha} can be controlled by the second term in the right-hand side of \eqref{eq:gx}, as well as
$$
2\alpha - \frac{\gamma+\sigma}{2} =
\frac{\gamma}{2} \quad \text{if} \quad  \frac{\gamma}{2} + \frac{\sigma}{4} \ge \gamma+\sigma ,
$$
$$
2\alpha - \frac{\gamma+\sigma}{2} 
= \frac{\gamma+\sigma}{2} + \alpha
=\frac32(\gamma+\sigma) \quad \text{if} \quad  \frac{\gamma}{2} + \frac{\sigma}{4} < \gamma+\sigma.
$$
As a consequence, the last term in \eqref{eq:gxgv-alpha} can be controlled by the first term in the right-hand side of \eqref{eq:g} or by the second term in the right-hand-side of \eqref{eq:gv}.

\medskip\noindent
{\it Step 3.}
Putting together previous estimates, it follows that for any $\eta >0$,
$$
\ba
\frac{d}{dt} \| g \|_{\widetilde\HH^1_{x,v} }^2
&\le  
- { \lambda} \| \la v \ra^{\frac{\gamma+\sigma}{2}} \,  g \|_{L^2_{x,v}}^2 
+ \eta { C} \| \la v \ra^{\frac{\gamma+\sigma-1}{2} + \alpha}  \, g \|_{L^2_{x,v}}^2 \\
&\quad
-  { \lambda} \| \la v \ra^{\frac{\gamma}{2}} \, \widetilde \nabla_v g \|_{L^2_{x,v}}^2
+ \eta { C} \| \la v \ra^{\frac{\gamma}{2}+\alpha} \, \widetilde \nabla_v g \|_{L^2_{x,v}}^2 \\
&\quad
- { \lambda} \| \la v \ra^{\frac{\gamma+\sigma}{2}} \, \nabla_x g \|_{L^2_{x,v}}^2 
+ \eta^{1/2} { C} \| \la v \ra^{\frac{\gamma+\sigma}{2}}  \, \nabla_x g \|_{L^2_{x,v}}^2 \\
&\quad
-  { \lambda} \| \la v \ra^{\frac{\gamma}{2}} \, \widetilde \nabla_v g \|_{L^2_{x,v}}^2
-\eta { \lambda} \| \la v \ra^{\frac{\gamma+\sigma}{2} + \alpha} \, \nabla_v g \|_{L^2_{x,v}}^2
+ \eta^{3/2} { C} \| \la v \ra^{2\alpha-\frac{\gamma+\sigma}{2}} \, \nabla_v g \|_{L^2_{x,v}}^2
\\
&\quad
- { \lambda} \| \la v \ra^{\frac{\gamma}{2}} \, \widetilde \nabla_v (\nabla_x g)  \|_{L^2_{x,v}}^2
-\eta { \lambda} \| \la v \ra^{\frac{\gamma}{2} + \alpha} \, \widetilde \nabla_v (\nabla_v g)  \|_{L^2_{x,v}}^2  ,
\ea
$$
and we conclude the proof by taking $\eta>0$ small enough.
\end{proof}

\begin{cor}\label{cor:SB-inhom}
Let $m_0,m_1$ be admissible weight functions such that $m_1 \succ m_0 \succ \la v \ra^{(\gamma+3)/2}$.  
There hold 
\be\label{eq:BBdecayL2-nonH}
\| S_{\bar \BB}(t) \|_{H^2_x L^2_{v}(m_1) \to H^2_x L^2_{v}(m_0)} \lesssim \Theta_{m_1,m_0}(t), \quad \forall \, t \ge 0 ,
\ee
\be\label{eq:BBdecayH1-nonH}
\| S_{\bar \BB}(t) \|_{\HH^1_{x,v}(m_1) \to  \HH^1_{x,v}(m_0)} \lesssim \Theta_{m_1,m_0}(t), \quad \forall \, t \ge 0  ,
\ee
Let $m_0,m_1,m$ be admissible { polynomial} weight functions such that $m \succeq m_1 \succ m_0 \succ \la v \ra^{(\gamma+3)/2}$. Then there holds
\be\label{eq:BB*decayL2-nonH}
\| S_{\bar \BB^*_m}(t) \|_{H^2_x L^2_{v}(\omega_1) \to H^2_x L^2_{v}(\omega_0)} \lesssim \Theta_{m_1,m_0}(t), \quad \forall \, t \ge 0  ,
\ee
where $\omega_1 := m/m_0$ and $\omega_0 := m/m_1$.

\end{cor}

\begin{proof}
The proof follows the same arguments of Lemma~\ref{lem:SBdecay}, using the weakly dissipative estimates of Lemmas~\ref{lem:barBB}, \ref{lem:barBB2} and \ref{lem:barBB1}.
\end{proof}

\subsection{Regularisation properties of $S_{\bar\BB}$ and $(\AA S_{\bar\BB})^{(*n)}$}
We start proving regularisation properties of the semigroup $S_{\bar \BB}$ in some large weighted Lebesgue and Sobolev spaces in the spirit of  H\'erau's quantitative version \cite{Herau} of the H\"ormander hypoellipticity property of the kinetic Fokker-Planck equation.

\begin{lem}\label{lem:reg}
Let $m, m_1$ be admissible polynomial weight functions such that $ \la v \ra^{3/2} \prec m_1 \prec m $ with $m_1 \prec \mu^{-1/2}$. Then the following regularity estimates hold:

\medskip
$(i)$ 
For any $n \in \N^*$, there holds 
\be\label{eq:BB-H1toH2-nonH}  
\forall\, t > 0, \qquad
\| S_{\bar \BB}(t)  \|_{L^{2}_{x,v}(m) \to H^n_{x,v} (m_1 \la v \ra^{\gamma/2})} \lesssim  \frac{\Theta_{m, m_1}(t)}{t^{3n/2} \wedge 1}.
\ee

\medskip
$(ii)$ 
For any $\ell \in \N$, there holds 
\be\label{eq:BB-H-1toL2-nonH}  
\forall\, t > 0, \qquad
\| S_{\bar \BB}(t)  \|_{H^\ell_x (H^{-1}_{v,*}(m)) \to H^\ell_{x} L^2_v (m_1 \la v \ra^{\gamma/2})} \lesssim  \frac{\Theta_{m, m_1}(t)}{t^{1/2} \wedge 1}.
\ee

\end{lem}

\begin{proof} We split the proof into two steps.

\medskip\noindent
{\it Step 1. Proof of (i).} 
We only prove \eqref{eq:BB-H1toH2-nonH} in the case $n=1$, the other cases can be obtained by iterating the case $n=1$.  In what follows we shall denote $\lambda,C>0$ positive constants that can change from line to line. 

Let us denote $m_0 := m_1 \la v \ra^{\gamma/2}$ and $f_t = \SS_{\bar \BB}(t) f$. Define $g^0_t = m_0 f_t$ and $g^1_t = m_1 f_t$, which verify $g^0_t = S_{\bar \BB_{m_0}}(t) g^0$ and $g^1_t = S_{\bar \BB_{m_1}} (t) g^1$. We define the functional
$$
\FF(t) := \| g^1_t \|_{L^2_{x,v}}^2 + \alpha_1\, t \,\| \nabla_v g^0_t \|_{L^2_{x,v}}^2 + \alpha_2 \,t^2 \, \la \nabla_x g^0_t, \nabla_v g^0_t \ra_{L^2_{x,v}} + \alpha_3\, t^3 \, \| \nabla_x g^0_t \|_{L^2_{x,v}}^2,
$$
and choose $\alpha_i$, $i=1,2,3$ such that $0<\alpha_3 \le \alpha_2 \le \alpha_1 \le 1$ and $\alpha_2^2 \le  \alpha_1 \alpha_3$. We already observe that we have the following lower bounds 
\be\label{eq:BddF1}
\forall\, t \in [0,1], \quad
 \FF(t) \gtrsim \| g^1_t \|_{L^2_{x,v}}^2 + \alpha_3 \, t^3 \,  \| \nabla_{x,v} g^0_t \|^2_{L^2_{x,v}}
 \gtrsim  \alpha_3 t^3 \,  \| g^0_t \|^2_{H^1_{x,v}},
\ee
and also
\be\label{eq:BddF2}
\forall\, t \in [0,1], \quad
 \FF(t) \gtrsim \| g^1_t \|_{L^2_{x,v}}^2 + \alpha_1 \, t \,  \| \nabla_{v} g^0_t \|^2_{L^2_{x,v}}
 \gtrsim  \alpha_1 t \,  \| g^0_t \|^2_{L^2_{x}(H^1_v)}
  \gtrsim  \alpha_1 t \,  \| g^0_t \|^2_{L^2_{x}(H^1_{v,*})},
\ee
where we have used the embedding $L^2_{x}(H^1_{v}) \subset L^2_{x}(H^1_{v,*}) $ in the last inequality.

We derive the functional $\FF$ in time to obtain
$$
\ba
\frac{d}{dt}\, \FF(t)
&= \frac{d}{dt} \| g^1_t \|_{L^2_{x,v}}^2
+ \alpha_1 \,\| \nabla_v g^0_t \|_{L^2_{x,v}}^2
+ \alpha_1\, t \,\frac{d}{dt}\| \nabla_v g^0_t \|_{L^2_{x,v}}^2 \\
&\quad
+ 2 \alpha_2\, t\,  \la \nabla_x g^0_t , \nabla_v g^0_t \ra_{L^2_{x,v}}
+ \alpha_2\, t^2\, \frac{d}{dt} \la \nabla_x g^0_t , \nabla_v g^0_t \ra_{L^2_{x,v}} \\
&\quad
+ 3\alpha_3 \,t^2 \, \| \nabla_x g^0_t \|_{L^2_{x,v}}^2
+ \alpha_3 \,t^3 \, \frac{d}{dt} \| \nabla_x g^0_t \|_{L^2_{x,v}}^2.
\ea
$$
Recall that $\bar \BB_m$ is defined in \eqref{def:barBBm}, so that we compute
$$
\ba
&\frac{d}{dt} \la \nabla_x g^0 , \nabla_v g^0 \ra_{L^2_{x,v}} 
= \int \nabla_x ({ \bar \BB_m} g^0) \cdot \nabla_v g^0 + \nabla_v ({ \bar \BB_m} g^0) \cdot \nabla_x g^0  \\ 
&\quad = \int \bar a_{ij} \partial_{ij} ( \nabla_x g^0) \, \nabla_v g^0
+ \beta_j \, \partial_j ( \nabla_x g^0) \, \nabla_v g^0
+ (\delta - M \chi_R) \nabla_x g^0 \, \nabla_v g^0  - v \cdot \nabla_x (\nabla_x g^0) \, \nabla_v g^0\\
&\quad\quad + \int \bar a_{ij} \partial_{ij} (\nabla_v g^0) \, \nabla_x g^0
+ \beta_j \, \partial_j (\nabla_v g^0) \, \nabla_x g^0
+ (\delta - M \chi_R) \nabla_v g^0 \, \nabla_x g^0  - v \cdot \nabla_x (\nabla_v g^0) \, \nabla_x g^0\\
&\quad\quad + \int (\nabla_v \bar a_{ij}) \partial_{ij} g^0 \, \nabla_x g^0
+ (\nabla_v \beta_j) \, \partial_j g^0 \, \nabla_x g^0
+ \nabla_v (\delta - M \chi_R) g^0 \, \nabla_x g^0  - |\nabla_x g^0|^2.
\ea
$$
Gathering terms and integrating by parts in last expression, we obtain (with the same type of arguments as in step 2 of Lemma~\ref{lem:barBB1})
$$
\ba
\frac{d}{dt} \la \nabla_x g^0 , \nabla_v g^0 \ra_{L^2_{x,v}}
&= -2 \int \bar a_{ij} \partial_i (\nabla_x g^0 ) \, \partial_j (\nabla_v g^0_v) 
+ \int \{ -\partial_j \beta_j + \bar c + 2 \delta - 2M\chi_R  \} \, \nabla_x g^0 \, \nabla_v  g^0   \\
&\quad + \int \nabla_v ( \beta_j - \bar b_j) \, \partial_j g^0 \, \nabla_x g^0
- \| \nabla_x g^0 \|_{L^2_{x,v}}^2 .
\ea
$$
From that equation, we deduce 
\be\label{eq:gxgv}
\ba
\frac{d}{dt} \la \nabla_x g^0_t , \nabla_v g^0_t \ra_{L^2_{x,v}}
&\le C \| \la v \ra^{\frac{\gamma}{2}} \widetilde \nabla_v (\nabla_x g^0_t) \|_{L^2_{x,v}}
\, \| \la v \ra^{\frac{\gamma}{2}} \widetilde \nabla_v (\nabla_v g^0_t) \|_{L^2_{x,v}}\\
&\quad + C \| \la v \ra^{\frac{\gamma}{2}}  \nabla_x g^0_t \|_{L^2_{x,v}}
\, \| \la v \ra^{\frac{\gamma }{2}}  \nabla_v g^0_t \|_{L^2_{x,v}}
 - \| \nabla_x g^0_t \|_{L^2_{x,v}}^2.
\ea
\ee
Recall that from Lemma \ref{lem:barBB}, we already have
\be\label{eq:g1}
\frac{d}{dt} \| g^1_t \|_{L^2_{x,v}}^2 \le - { \lambda} \| \la v \ra^{\frac{\gamma}{2}} \widetilde \nabla_v g^1_t \|_{L^2_{x,v}}^2 - { \lambda} \| \la v \ra^{\frac{\gamma }{2}} g^1_t \|_{L^2_{x,v}}^2 .
\ee
Moreover, thanks to the proof of Lemma~\ref{lem:barBB1}, we get
\be\label{eq:g0v}
\ba
\frac{d}{dt} \| \nabla_v g^0_t \|_{L^2_{x,v}}^2
& \le - { \lambda} \| \la v \ra^{\frac{\gamma}{2}} \widetilde \nabla_v (\nabla_v g^0_t) \|_{L^2_{x,v}}^2
- { \lambda} \| \la v \ra^{\frac{\gamma}{2}}  \nabla_v g^0_t \|_{L^2_{x,v}} \\
&\quad + { C}\| \la v \ra^{\frac{\gamma-1}{2}}  g^0_t \|_{L^2_{x,v}}^2
+ { C} \| \la v \ra^{\frac{\gamma}{2}} \widetilde \nabla_v  g^0_t\|_{L^2_{x,v}}^2 \\
&\quad + { C} \| \nabla_x g^0_t \|_{L^2_{x,v}} \, \| \nabla_v g^0_t \|_{L^2_{x,v}}.
\ea
\ee
Using Lemma \ref{lem:barBB} and the fact that $\nabla_x$ commutes with $\bar \BB$, we also have 
\be\label{eq:g0x}
\frac{d}{dt} \| \nabla_x g^0_t \|_{L^2_{x,v}}^2 \le - { \lambda} \| \la v \ra^{\frac{\gamma}{2}} \widetilde \nabla_v (\nabla_x g^0_t) \|_{L^2_{x,v}}^2 - { \lambda} \| \la v \ra^{\frac{\gamma}{2}}  \nabla_x g^0_t \|_{L^2_{x,v}}^2.
\ee

Let us denote $ D_1 := \lambda (\| \la v \ra^{\frac{\gamma}{2}} \widetilde \nabla_v g^1_t \|_{L^2_{x,v}}^2 + \| \la v \ra^{\frac{\gamma }{2}} g^1_t \|_{L^2_{x,v}}^2) $ the absolute value of the dissipative terms in \eqref{eq:g1}, $ D_2 := \lambda (\| \la v \ra^{\frac{\gamma}{2}} \widetilde \nabla_v (\nabla_v g^0_t) \|_{L^2_{x,v}}^2+ \| \la v \ra^{\frac{\gamma}{2}}  \nabla_v g^0_t \|_{L^2_{x,v}})$ the absolute value of the dissipative 
terms in \eqref{eq:g0v}, $ D_3 := \| \nabla_x g^0_t \|_{L^2_{x,v}}^2$ the absolute value of the dissipative terms in \eqref{eq:gxgv}, and finally $ D_4 := \lambda(\| \la v \ra^{\frac{\gamma}{2}} \widetilde \nabla_v (\nabla_x g^0_t) \|_{L^2_{x,v}}^2 + \| \la v \ra^{\frac{\gamma}{2}}  \nabla_x g^0_t \|_{L^2_{x,v}}^2) $ the absolute value of the dissipative terms in \eqref{eq:g0x}. Observe that
$$
\|  \nabla_v g^0_t \|_{L^2_{x,v}}^2
+ \| \la v \ra^{\frac{\gamma}{2}} \widetilde \nabla_v g^0_t  \|_{L^2_{x,v}}^2
\lesssim D_1.
$$
Gathering estimates \eqref{eq:gxgv}, \eqref{eq:g1}, \eqref{eq:g0v} and \eqref{eq:g0x}, we obtain, for any $t \in (0,1]$,
$$
\ba
\frac{d}{dt} \FF(t) 
&\le (-1 + C\alpha_1 + C \alpha_1 t  ) D_1 + (C\alpha_1 t + C\alpha_2 t + C\alpha_2 t^2) \, D_1^{1/2} \, D_3^{1/2} \\
&\quad - \alpha_1 t D_2   - \alpha_2 t^2 \, D_3 
+ C \alpha_2 t^2 \, D_2^{1/2} \, D_4^{1/2} + C \alpha_3 t^2 \, D_3 - \alpha_3 t^3 \, D_4\\
&\le (-1 + C\alpha_1 ) D_1 + C \alpha_1 t  \, D_1^{1/2} \, D_3^{1/2} \\
&\quad - \alpha_1 t D_2   + (-\alpha_2 + C \alpha_3) t^2 \, D_3 
+ C \alpha_2 t^2 \, D_2^{1/2} \, D_4^{1/2}  - \alpha_3 t^3 \, D_4.
\ea
$$
Using Cauchy-Schwarz inequality we first get, for some $0 < \alpha_4 < \alpha_3$ to be chosen later, 
$$
\alpha_1 t  \, D_1^{1/2} \, D_3^{1/2} \lesssim \frac{\alpha_1^2}{\alpha_3} \, D_1 + \alpha_3 t^2 D_3, \quad
\alpha_2 t^2 \, D_2^{1/2} \, D_4^{1/2} \lesssim \frac{\alpha_2^2}{\alpha_4} \, t D_2 + \alpha_4 t^3 D_4,
$$
from which it follows, for $t \in (0,1]$,
$$
\frac{d}{dt} \FF(t) \le (-1 + C\alpha_1  + C\frac{\alpha_1^2}{\alpha_3} ) D_1 
+ t(-\alpha_1  + C\frac{\alpha_2^2}{\alpha_4} ) D_2
+ t^2 (-\alpha_2 + C\alpha_3   ) D_3 
+ t^3(- \alpha_3 +  C\alpha_4) D_4.
$$
Let $\e \in (0,1)$. We choose $\alpha_1 = \e > \alpha_2 = \e^{3/2} > \alpha_3 = \e^{5/3} > \alpha_4 = \e^{11/6}$ so that $\alpha_2^2 \le \alpha_1 \, \alpha_3 $. Taking $\e >0$ small enough, we easily conclude to
\be\label{eq:dF/dt<0}
\forall \, t \in (0,1], \quad \frac{d}{dt} \FF(t) \le 0.
\ee
This implies, coming back to the function $f_t = S_{\bar \BB}(t) f$ and using \eqref{eq:BddF1},
$$
\forall \, t \in (0,1],\quad   t^3  \,  \| S_{\bar \BB}(t) f  \|^2_{H^1_{x,v}(m_1 \la v \ra^{\gamma/2})} \lesssim  \| f \|^2_{L^2_{x,v}(m_1)},
$$
which already gives \eqref{eq:BB-H1toH2-nonH} for small times $t \in (0,1]$. For large times $t>1$ and $m \succ m_1$ (recall that $m_1 \la v \ra^{\gamma/2} \succ \la v \ra^{(\gamma+3)/2}$) we write, using first the last estimate and next \eqref{eq:BBdecayL2-nonH},
$$
\ba
\| S_{\bar \BB}(t) f \|_{H^1_{x,v} (m_1 \la v \ra^{\gamma/2})} 
&= \| S_{\bar \BB}(1) (S_{\bar \BB}(t-1) f) \|_{H^1_{x,v} (m_1 \la v \ra^{\gamma/2})} \\ 
& \lesssim \| S_{\bar \BB}(t-1) f \|_{L^2_{x,v} (m_1)} \\
& \lesssim \Theta_{m,m_1}(t) \| f \|_{L^2_{x,v}(m)} ,
\ea
$$
which completes the proof of \eqref{eq:BB-H1toH2-nonH}. 
In a similar way, using \eqref{eq:dF/dt<0} together with \eqref{eq:BddF2} (instead of \eqref{eq:BddF1}) and \eqref{eq:BBdecayL2-nonH}, we obtain
\be\label{eq;L2vtoH1v*}
\| S_{\bar \BB}(t) \|_{L^2_x L^2_v (m) \to L^2_x (H^1_{v,*}(m_1 \la v \ra^{\gamma/2}))} \lesssim \frac{\Theta_{m_1,m_0}(t)}{t^{1/2} \wedge 1}, \quad \forall\, t >0.
\ee

\medskip\noindent
{\it Step 2. Proof of (ii).} We only need to prove \eqref{eq:BB-H-1toL2-nonH} for $\ell=0$, since the operators $\nabla_x$ and $\bar \BB$ commute.

Define $\omega_0 := 1$, $\omega_1 := \la v \ra^{|\gamma|/2}$ and $\omega := m/(m_1 \la v \ra^{\gamma/2})$ so that $1 \prec \omega \prec m\la v \ra^{-(\gamma+3)/2}$.
Let us denote $f_t = S_{\bar \BB}(t) f$ and $\phi_t = S_{\bar \BB^*_m} \phi$. Arguing as in Step 1, we define the functional
$$
\RR(t) := \| \phi_t \|_{L^2_{x,v} (\omega_1)}^2 + a_1 t \| \nabla_v \phi_t \|_{L^2_{x,v} (\omega_0)}^2 + a_2 t^2 \la \nabla_x \phi_t , \nabla_v \phi_t \ra_{L^2_{x,v} (\omega_0)}^2 
+ a_3 t^3 \| \nabla_x \phi_t \|_{L^2_{x,v} (\omega_0)}^2,
$$
and we can choose appropriate constants $a_1,a_2,a_3 >0$ such that it follows
$$
\| S_{\bar \BB^*_m}(t) \|_{L^2_x L^2_v (\omega) \to L^2_x (H^1_{v,*}(\omega_1 \la v \ra^{\gamma/2}))}
\lesssim \frac{\Theta_{m_1,m_0}(t)}{t^{1/2} \wedge 1}, \quad \forall\, t >0,
$$
Last estimate implies by duality
$$
\| S_{\bar \BB}(t) \|_{L^2_x (H^{-1}_{v,*} (m)) \to L^2_x L^2_v(m_1 \la v \ra^{\gamma/2})}
\lesssim \frac{\Theta_{m_1,m_0}(t)}{t^{1/2} \wedge 1}, \quad \forall\, t >0,
$$
which completes the proof.
\end{proof}

As a consequence of Lemma~\ref{lem:AA}, we also obtain an analogous result for high-order Sobolev spaces.

\begin{cor}\label{cor:AA} 
For any  $\theta\in(0,1)$ and $n \in \N$, there hold
$ \AA \in \BBB (H^{n}_x L^2_v , H^{n}_x L^2_v(\mu^{-\theta}) )$ and $ \AA \in \BBB (H^{n}_{x,v}  , H^{n}_{x,v} (\mu^{-\theta}) )$.
\end{cor}

We finally obtain the following regularity properties, as a consequence of Corollary~\ref{cor:SB-inhom}, Lemma~\ref{lem:reg} and Corollary~\ref{cor:AA}.

\begin{cor}\label{cor:ASB*2}
Let $m,\nu$ be admissible weight functions such that $ \la v \ra^{(\gamma+3)/2}\prec  m \prec \nu$  and $ \mu^{-1/2} \prec \nu \prec \mu^{-1}$. There hold
\be\label{eq:ASB*2}
\forall\, t >0, \quad
\| (\AA S_{\bar \BB} )^{(*2)} (t) \|_{L^2_{x,v} (\nu) \to \HH^1_{x,v} (\nu)} 
\lesssim \frac{e^{- \lambda t^{2/|\gamma|}}}{t^{1/2} \wedge 1},
\ee
and
\be\label{eq:SBA*2}
\forall\, t >0, \quad
\| (S_{\bar \BB} \AA )^{(*4)} (t) \|_{L^2_{x,v} (m) \to H^2_{x,v} (m)}
\lesssim e^{- \lambda t^{2/|\gamma|}}.
\ee
\end{cor}

% It is worth emphasizing that in \eqref{eq:SBA*2} in Corollary~\ref{cor:ASB*2} we have indeed that very str strongest decay estimate we may indeed consider any admissible weight function $m$, even if in the estimates of Lemma~\ref{lem:reg} that shall be used in the proof below, only polynomial  functions are considered. The reason for this is that we can always recover any loss of weight thanks to the properties of the operator $\AA$ stated in Corollary~\ref{cor:AA}. REPHRASE !!!

\begin{proof}
We define $m_1 := m \la v \ra^{|\gamma|/2} \succ \la v \ra^{3/2}$ and observe that with the choice of the weight $\nu$ we have $\Theta_{{ \nu }, m_1} (t)  = e^{-\lambda t^{2/|\gamma|}}$.

\smallskip\noindent
{\it Step 1.}
Thanks to Corollary~\ref{cor:SB-inhom} and Corollary~\ref{cor:AA}, we already have, 
$$
\| \AA S_{\bar \BB} (t) \|_{L^2_{x,v}(\nu) \to L^2_{x,v}(\nu)}
\lesssim \| \AA  \|_{L^2_{x,v}(m) \to L^2_{x,v}(\nu)} \, \| S_{\bar \BB} (t) \|_{L^2_{x,v}(\nu) \to L^2_{x,v}(m)}
\lesssim \Theta_{\nu,m}(t)
$$
and
$$
\| S_{\bar \BB} \AA (t)  \|_{L^2_{x,v}(m) \to L^2_{x,v}(m)}
\lesssim  \| S_{\bar \BB} (t) \|_{L^2_{x,v}(\nu) \to L^2_{x,v}(m)} \, \| \AA  \|_{L^2_{x,v}(m) \to L^2_{x,v}(\nu)}
\lesssim \Theta_{\nu,m}(t)
$$
so that, for any $j \in \N^*$,
$$
\ba
\| (\AA S_{\bar \BB})^{(*j)} (t) \|_{L^2_{x,v}(\nu) \to L^2_{x,v}(\nu)} , \;
\| (S_{\bar \BB} \AA)^{(*j)} (t) \|_{L^2_{x,v}(m) \to L^2_{x,v}(m)}
\lesssim \Theta_{\nu,m}(t) ,
\ea
$$
and similarly
$$
\ba
\| (\AA S_{\bar \BB})^{(*j)} (t) \|_{\HH^1_{x,v}(\nu) \to \HH^1_{x,v}(\nu)} , \;
\| (S_{\bar \BB} \AA)^{(*j)} (t) \|_{\HH^1_{x,v}(m) \to \HH^1_{x,v}(m)}
\lesssim \Theta_{\nu,m}(t) .
\ea
$$

\smallskip\noindent
{\it Step 2.} We prove \eqref{eq:ASB*2}.
We first write
$$
\ba
&\| (\AA S_{\bar \BB} )^{(*2)} (t) \|_{L^2_{x,v} (\nu) \to \HH^1_{x,v} (\nu)} \\
&\lesssim \int_0^{t/2} \| \AA S_{\bar \BB}(t-s)  \AA S_{\bar \BB} (s) \|_{L^2_{x,v} (\nu) \to \HH^1_{x,v} (\nu)} \, ds
\\
&\quad+ \int_{t/2}^t \| \AA S_{\bar \BB}(t-s)  \AA S_{\bar \BB} (s) \|_{L^2_{x,v} (\nu) \to \HH^1_{x,v} (\nu)} \, ds =: I_1 + I_2.
\ea
$$
Using Corollary~\ref{cor:SB-inhom}, \eqref{eq:BB-H1toH2-nonH} of Lemma~\ref{lem:reg}, Corollary~\ref{cor:AA} and Step 1, we have 
$$
\ba
I_1 & \lesssim 
\int_0^{t/2} \| \AA \|_{H^1_{x,v} (m_1 \la v \ra^{\gamma/2}) \to \HH^1_{x,v} (\nu)} \, \| S_{\bar \BB}(t-s) \|_{L^2_{x,v} (\nu) \to H^1_{x,v} (m_1 \la v \ra^{\gamma/2})} \,  \| \AA  S_{\bar \BB} (s) \|_{L^2_{x,v} (\nu) \to  L^2_{x,v} (\nu)} \, ds \\
&\lesssim \int_0^{t/2} \frac{\Theta_{\nu, m_1} (t-s)}{(t-s)^{3/2} \wedge 1} \, \Theta_{\nu ,m} (s) \, ds 
\lesssim \Theta_{\nu, m_1} (t/2)  \int_0^{t/2}  \frac{\Theta_{\nu,m}(s)}{(t-s)^{3/2} \wedge 1} \, ds  \lesssim \frac{\Theta_{\nu, m_1} (t)}{t^{1/2} \wedge 1} .
\ea
$$
For the other term $I_2$, we use again Corollary~\ref{cor:SB-inhom}, \eqref{eq:BB-H1toH2-nonH} of Lemma~\ref{lem:reg}, Corollary~\ref{cor:AA} and Step 1, but in a different order, to obtain
$$
\ba
I_2 &\lesssim
\int_{t/2}^t \| \AA  S_{\bar \BB}(t-s) \|_{\HH^1_{x,v} (\nu) \to \HH^1_{x,v} (\nu)} \,
\| \AA \|_{H^1_{x,v} (m_1 \la v \ra^{\gamma/2}) \to \HH^1_{x,v} (\nu)} \, \| S_{\bar \BB} (s) \|_{L^2_{x,v} (\nu) \to H^1_{x,v} (m_1 \la v \ra^{\gamma/2})} \, ds \\
&\lesssim 
\int_{t/2}^t \Theta_{\nu,m} (t-s) \, \frac{\Theta_{\nu,m_1} (s)}{s^{3/2} \wedge 1} \, ds 
\lesssim \Theta_{\nu,m_1} (t/2)
\int_{t/2}^t \frac{\Theta_{\nu,m} (t-s)}{s^{3/2} \wedge 1}  \, ds
\lesssim \frac{\Theta_{\nu,m_1} (t)}{t^{1/2} \wedge 1} ,
\ea
$$
and the proof of \eqref{eq:ASB*2} is complete.

\smallskip\noindent
{\it Step 3.}
We now turn to the proof of \eqref{eq:SBA*2}. 
We claim that, for any $j \in \N$, there holds
\be\label{SBA*j+1}
\| (S_{\bar \BB} \AA )^{(*(j+1))} (t) \|_{L^2_{x,v} (m) \to H^n_{x,v}(m)} \lesssim \frac{\Theta_{\nu,m_1} (t) }{t^{3n/2 - j} \wedge 1},
\ee
so that we can conclude to \eqref{eq:SBA*2} by choosing $j=3$ when $n=2$. 

The case $j=0$ follows directly from Lemma~\ref{lem:reg} and Corollary~\ref{cor:AA}, and we prove the claim by induction. Suppose that \eqref{SBA*j+1} holds for some $j$ then we compute, splitting again the integral into two parts,
$$
\ba
\| (S_{\bar \BB} \AA )^{(*(j+2))} (t) \|_{L^2_{x,v} (m) \to  H^n_{x,v} (m)} 
&\lesssim \int_0^{t/2} \| (S_{\bar \BB} \AA)^{(*(j+1))} (t-s)  S_{\bar \BB} \AA (s)  \|_{L^2_{x,v} (m) \to  H^n_{x,v} (m)} \, ds \\
&\quad
+ \int_{t/2}^t \| S_{\bar \BB}\AA (t-s)   (S_{\bar \BB} \AA)^{(*(j+1))} (s)\|_{L^2_{x,v} (m) \to  H^n_{x,v} (m)} \, ds \\
& =: T_1 + T_2.
\ea
$$
In a similar way as in Step 2, using Corollary~\ref{cor:SB-inhom}, \eqref{eq:BB-H1toH2-nonH} of Lemma~\ref{lem:reg}, Corollary~\ref{cor:AA} and Step~1, we obtain
$$
\ba
T_1 
&\lesssim
\int_0^{t/2} \| (S_{\bar \BB} \AA)^{(*(j+1))} (t-s) \|_{L^2_{x,v} (m) \to H^n_{x,v}(m)}  \,
 \| S_{\bar \BB} \AA (s) \|_{L^2_{x,v} (m) \to L^2_{x,v} (m)}  \, ds \\
&\lesssim \int_0^{t/2} \frac{\Theta_{\nu, m_1 } (t-s)}{(t-s)^{3n/2-j} \wedge 1} \, \Theta_{\nu , m} (s) \, ds \lesssim \frac{\Theta_{\nu , m_1} (t)}{t^{3n/2-(j+1)} \wedge 1}.
\ea
$$ 
Moreover, 
$$
\ba
T_2 
&\lesssim
\int_{t/2}^t \| S_{\bar \BB} \AA (t-s) \|_{L^2_{x,v}(m)  \to L^2_{x,v} (m)} 
\,  \| (S_{\bar \BB} \AA )^{(*(j+1))} (s) \|_{L^2_{x,v} (m) \to H^n_{x,v} (m)} \, ds \\
&\lesssim \int_{t/2}^t \Theta_{\nu, m}(t-s) \, \frac{\Theta_{\nu, m_1} (s)}{s^{3n/2-j} \wedge 1} \,  ds 
\lesssim \frac{\Theta_{\nu , m_1} (t)}{t^{3n/2-(j+1)} \wedge 1},
\ea
$$
which completes the proof.
\end{proof}

\subsection{Decay of the semigroup $S_{\bar \LL}$}
With the results above we obtain the decay of the semigroup $S_{\bar \LL} \bar \Pi$ in large spaces as considered in the statement of Theorem~\ref{thm:stabNL-inhom}.

\medskip
We first write a semigroup factorization. 
Recall that $\bar \LL = \AA + \bar \BB$ and that $\bar \Pi$ commutes with $\bar \LL$. 
For any $\ell,n \in \N^*$, we can write the iterated Duhamel formulas
$$
\bar \Pi S_{\bar \LL}
=  \sum_{0\le j \le \ell -1} \bar \Pi S_{\bar \BB}  * (\AA S_{\bar \BB})^{(*j)}
+ \bar\Pi S_{\bar \LL} * (\AA S_{\bar \BB})^{(*\ell)}
$$
$$
{ S_{\bar \LL} }\bar\Pi = \sum_{0 \le i \le n-1} (S_{\bar \BB} \AA)^{(*i)}*S_{\bar \BB} \bar \Pi 
+ (S_{\bar \BB}\AA)^{(*n)} * S_{\bar \LL} \bar \Pi,
$$
and then deduce  
\be\label{eq:SL=SB+SL*ASB}
\ba
S_{\bar \LL} \bar \Pi 
&= \sum_{0\le j \le \ell -1}  \bar \Pi S_{\bar \BB}  * (\AA S_{\bar \BB})^{(*j)}
+  \sum_{0 \le i \le n-1} (S_{\bar \BB} \AA)^{(*i)} *S_{\bar \BB} \bar \Pi * (\AA S_{\bar \BB})^{(*\ell)}\\
&\quad
+ (S_{ \BB}\AA)^{(*n)} * S_{\bar \LL} \bar \Pi  *  (\AA S_{\bar \BB})^{(*\ell)}.
\ea
\ee

\begin{thm}\label{thm:S_L-inhom}
Let $m_0,m_1$ be two admissible weight functions such that $ \la v \ra^{(\gamma+3)/2} \prec m_0 \prec m_1$ and $m_0 \preceq \mu^{-1/2}$. Then we have the uniform in time bound
\be\label{eq:SLboundLinfty}
t \mapsto \| S_{\bar \LL}(t) \bar \Pi \|_{H^2_x L^2_v(m_0) \to H^2_x L^2_v(m_0)} \in L^\infty(\R_+),
\ee
as well as the decay estimate
\be\label{eq:LLdecayL2-inhom}
\| S_{\bar \LL}(t) \bar \Pi \|_{H^2_x L^2_{v}(m_1) \to H^2_x L^2_{v}(m_0)} \lesssim \Theta_{m_1,m_0}(t) \quad \forall\, t \ge 0.
\ee
Let $m_0,m_1$ be admissible polynomial weight functions such that $\la v \ra^{3/2} \prec m_0 \prec m_1$. Then the following regularity estimate holds
\be\label{eq:LLreg-H-1-inhom}
\| S_{\bar \LL}(t) \bar \Pi \|_{H^2_x ( H^{-1}_{v,*}(m_1)) \to H^2_x L^2_{v}(m_0 \la v \ra^{\gamma/2})} \lesssim \frac{\Theta_{m_1,m_0}(t)}{t^{1/2} \wedge 1}, \quad \forall \, t > 0 .
\ee

\end{thm}

\begin{proof}   
We fix an admissible weight function $\nu$ such that $\mu^{-1/2} \prec \nu \prec \mu^{-1}$ with $\nu \succ m_1$, and split the proof into five steps.

\medskip\noindent
{\it Step 1. Decay in the small function space.} 
Let us denote $E_0 = \HH^1_{x,v}(\mu^{-1/2})$ and $E_1 = \HH^1_{x,v}( \nu )$. Arguing exactly as in Proposition~\ref{prop:decay-small}, using Lemma~\ref{lem:GenCoerc} and Lemma~\ref{lem:barBB1} we obtain
$$
\forall\, t \ge 0, \quad
\| S_{\bar \LL} (t) \bar \Pi \|_{E_1 \to E_0}
\lesssim e ^{- \lambda  t^{\frac{2}{|\gamma|}}} .
$$

\medskip\noindent
{\it Step 2. Factorization.} We write the factorization identity thanks to \eqref{eq:SL=SB+SL*ASB}
\be\label{eq:SL=S1S3}
\ba
S_{\bar \LL} \bar \Pi 
&=  \sum_{0\le j \le 2} \bar \Pi S_{\bar \BB}  * (\AA S_{\bar \BB})^{(*j)} 
+ \sum_{0\le i \le 3}  (S_{\bar \BB} \AA)^{(*i)}      *  S_{\bar \BB} \bar \Pi  * (\AA S_{\bar \BB})^{(*3)}  \\
&\quad 
+  (S_{\bar \BB} \AA)^{(*4)}      *  S_{\bar \LL} \bar \Pi  * (\AA S_{\bar \BB})^{(*3)} \\
&=: \sum_{0 \le j \le 2} S_1^j + \sum_{0 \le i \le 3} S_2^i + S_3.
\ea
\ee

\medskip\noindent
{\it Step 3. Proof of \eqref{eq:SLboundLinfty}.} 
Let us denote $X_0 = H^2_x L^2_v(m_0)$ and $X_2 = H^2_x L^2_v (\nu)$. Thanks to Corollary~\ref{cor:ASB*2}, we already have
$$
t \mapsto \| (\AA S_{\bar \BB})^{(*2)}(t) \|_{X_2  \to E_1 } \in L^1 (\R_+) , \quad
t \mapsto \|  (S_{\bar \BB} \AA)^{(*4)}(t) \|_{E_0 \to X_0} \in L^1(\R_+).
$$
From Corollary~\ref{cor:SB-inhom} and Corollary~\ref{cor:AA}, it also holds, for any $ i,j \ge 1 $,
$$
t \mapsto \| S_{\bar \BB} (t) \|_{X_2 \to X_0} , \; 
t \mapsto \| (\AA S_{\bar \BB})^{(*j)} (t) \|_{X_2 \to X_2} , \; 
t \mapsto \| (S_{\bar \BB} \AA)^{(*i)} (t) \|_{X_0 \to X_0}  \in L^1 (\R_+),
$$
moreover
$$
t \mapsto \| S_{\bar \BB}(t) \|_{X_0 \to X_0}, \; 
t \mapsto \| \AA S_{\bar \BB} \|_{X_0 \to X_2} \in L^\infty(\R_+).
$$
Gathering these previous estimates and using the factorization \eqref{eq:SL=S1S3}, we first get
$$
\| S_1^0(t) \|_{X_0 \to X_0} \in L^\infty_t(\R_+),
$$
Moreover, for $1 \le j \le 2$ and $0 \le i \le 3$, we also have
$$
\ba
\| S_1^j(t) \|_{X_0 \to X_0} 
&\lesssim  \| S_{\bar \BB}(t) \|_{X_2 \to X_0} * \| (\AA S_{\bar \BB})^{(*(j-1))} (t) \|_{X_2 \to X_2} * \| \AA S_{\bar \BB} (t)\|_{X_0 \to X_2} \\
&\quad \in L^1_t(\R_+) * L^1_t (\R_+)* L^\infty_t (\R_+) \subset  L^\infty_t(\R_+),
\ea
$$
and
$$
\ba
\| S_2^i(t) \|_{X_0 \to X_0} 
&\lesssim  \| (S_{\bar \BB} \AA)^{(*i)}  (t) \|_{X_0 \to X_0}   * \| S_{\bar \BB}(t) \|_{X_2 \to X_0}   * \| (\AA S_{\bar \BB})^{(*2)}(t) \|_{X_2 \to X_2} *  \| \AA S_{\bar \BB}(t) \|_{X_0 \to X_2} \\
&\quad\in L^1_t(\R_+) * L^1_t (\R_+) * L^1_t (\R_+)* L^\infty_t (\R_+) \subset  L^\infty_t(\R_+).
\ea
$$
Finally, using Step 1, it follows
$$
\ba
\| S_3(t) \|_{X_0 \to X_0} 
&\lesssim  \| (S_{\bar \BB}(t) \AA)^{(*4)} \|_{E_0 \to X_0}   * \| S_{\bar \LL}(t)  \bar \Pi \|_{E_1 \to E_0} *  \| (\AA S_{\bar \BB})^{(*2)} (t) \|_{X_2 \to E_1} * \| \AA S_{\bar \BB} (t) \|_{X_0 \to X_2} \\
&\quad\in L^1_t(\R_+) * L^1_t (\R_+) * L^1_t (\R_+)* L^\infty_t (\R_+) \subset  L^\infty_t(\R_+),
\ea
$$
which completes the proof of \eqref{eq:SLboundLinfty}.

\medskip\noindent
{\it Step 4. Proof of \eqref{eq:LLdecayL2-inhom}.} 
Let us denote $X_0 = H^2_x L^2_v(m_0)$, $X_1 = H^2_x L^2_v(m_1)$ and $X_2 = H^2_x L^2_v (\nu)$.
From Corollary~\ref{cor:ASB*2} it follows
$$
t \mapsto \Theta_{m_1,m_0}^{-1} (t) \| (\AA S_{\bar \BB})^{(*2)}(t) \|_{X_2  \to E_1 } \in L^1 (\R_+) , \quad
t \mapsto \Theta_{m_1,m_0}^{-1} (t) \|  (S_{\bar \BB} \AA)^{(*4)}(t) \|_{E_0 \to X_0} \in L^1(\R_+).
$$
Thanks to Corollary~\ref{cor:SB-inhom} and Corollary~\ref{cor:AA}, it also holds, for any $ i,j \ge 1 $,
$$
\ba
& t \mapsto \Theta_{m_1,m_0}^{-1} (t) \| S_{\bar \BB} (t) \|_{X_2 \to X_0}  \in L^1 (\R_+), \\
& t \mapsto \Theta_{m_1,m_0}^{-1} (t) \| (\AA S_{\bar \BB})^{(*j)} (t) \|_{X_2 \to X_2} \in L^1 (\R_+),\\
& t \mapsto \Theta_{m_1,m_0}^{-1} (t) \| (S_{\bar \BB} \AA)^{(*i)} (t) \|_{X_0 \to X_0} 
 \in L^1 (\R_+) ,
\ea
$$
and also
$$
t \mapsto \Theta_{m_1,m_0}^{-1} (t) \| S_{\bar \BB}(t) \|_{X_1 \to X_0}, \; 
t \mapsto \Theta_{m_1,m_0}^{-1} (t) \| \AA S_{\bar \BB} \|_{X_1 \to X_2} \in L^\infty(\R_+).
$$
We deduce \eqref{eq:LLdecayL2-inhom} by writing the factorization \eqref{eq:SL=S1S3} and using the above estimates.
Indeed, with $\Theta := \Theta_{m_1,m_0}$, we have 
$$
\ba
&\Theta^{-1} \| S_{\LL} \Pi \|_{X_1 \to X_0} \\
&\lesssim \Theta^{-1}  \| S_{\bar \BB} \|_{X_1 \to X_0} \\
&\quad + \sum_{1 \le j \le 2}  (\Theta^{-1}  \| S_{\bar \BB} \|_{X_2 \to X_0}) * (\Theta^{-1}  \| (\AA S_{\bar \BB})^{*(j-1)} \|_{X_2 \to X_2}) * (\Theta^{-1} \| \AA S_{\bar \BB} \|_{X_1 \to X_2}) \\
&\quad + \sum_{0 \le i \le 3}  (\Theta^{-1}  \| (S_{\bar \BB} \AA)^{*i} \|_{X_0 \to X_0}) * (\Theta^{-1}  \|  S_{\bar \BB} \|_{X_2 \to X_0} ) * (\Theta^{-1} \| (\AA S_{\bar \BB})^{*2} \|_{X_2 \to X_2} )
* (\Theta^{-1} \| \AA S_{\bar \BB} \|_{X_1 \to X_2}) \\
&\quad + (\Theta^{-1}  \| (S_{\bar \BB} \AA)^{*4} \|_{E_0 \to X_0}) * (\Theta^{-1}  \|  S_{\bar \LL}  \|_{E_1 \to E_0} )* (\Theta^{-1} \| (\AA S_{\bar \BB})^{*2} \|_{X_2 \to E_1} )
* (\Theta^{-1} \| \AA S_{\bar \BB} \|_{X_1 \to X_2} ).
\ea
$$

\medskip\noindent
{\it Step 5. Proof of \eqref{eq:LLreg-H-1-inhom}.}
Let us denote $Z_1 = H^2_x (H^{-1}_{v,*} (m_1))$, $\widetilde X_0 = H^2_x L^2_v(m_0 \la v \ra^{\gamma/2})$, and also $\widetilde \Theta_{m_1,m_0}(t) = \Theta_{m_1,m_0}(t)  / (t^{1/2} \wedge 1)$. 
From Corollary~\ref{cor:ASB*2} it follows
$$
t \mapsto \widetilde \Theta_{m_1,m_0}^{-1} (t) \| (\AA S_{\bar \BB})^{(*2)}(t) \|_{X_2  \to E_1 } \in L^1 (\R_+) , \quad
t \mapsto \widetilde \Theta_{m_1,m_0}^{-1} (t) \|  (S_{\bar \BB} \AA)^{(*4)}(t) \|_{E_0 \to \widetilde X_0} \in L^1(\R_+).
$$
Thanks to Corollary~\ref{cor:SB-inhom} and Corollary~\ref{cor:AA}, it also holds, for any $ i,j \ge 1 $,
$$
\ba
& t \mapsto \widetilde \Theta_{m_1,m_0}^{-1} (t) \| S_{\bar \BB} (t) \|_{X_2 \to \widetilde X_0}  \in L^1 (\R_+), \\
& t \mapsto \widetilde \Theta_{m_1,m_0}^{-1} (t) \| (\AA S_{\bar \BB})^{(*j)} (t) \|_{X_2 \to X_2} \in L^1 (\R_+),\\
& t \mapsto \widetilde \Theta_{m_1,m_0}^{-1} (t) \| (S_{\bar \BB} \AA)^{(*i)} (t) \|_{\widetilde X_0 \to \widetilde X_0} 
 \in L^1 (\R_+),
\ea
$$
and also, using Lemma~\ref{lem:reg}-(ii),
$$
t \mapsto \widetilde \Theta_{m_1,m_0}^{-1}(t)  \| S_{\bar \BB} (t) f \|_{Z_1 \to \widetilde X_0} \in L^\infty(\R_+) , \quad
t \mapsto \widetilde \Theta_{m_1,m_0}^{-1}(t)  \| \AA S_{\bar \BB} (t) \|_{Z_1 \to X_2} \in L^\infty (\R_+).
$$
We deduce \eqref{eq:LLreg-H-1-inhom} by writing the factorization \eqref{eq:SL=S1S3} and using the above estimates similarly as in Step 4.
%{\Red
%Indeed we have
%$$
%\ba
%&\widetilde \Theta_{m_1,m_0}^{-1} \| S_{\LL} \Pi \|_{Z_1 \to \widetilde X_0} \\
%&\lesssim \widetilde \Theta_{m_1,m_0}^{-1}  \| S_{\bar \BB} \|_{Z_1 \to \widetilde X_0} \\
%&\quad + \sum_{1 \le j \le 2}  (\widetilde \Theta_{m_1,m_0}^{-1}  \| S_{\bar \BB} \|_{X_2 \to \widetilde X_0}) * (\widetilde \Theta_{m_1,m_0}^{-1}  \| (\AA S_{\bar \BB})^{*(j-1)} \|_{X_2 \to X_2}) * (\widetilde \Theta_{m_1,m_0}^{-1} \| \AA S_{\bar \BB} \|_{Z_1 \to X_2}) \\
%&\quad + \sum_{0 \le i \le 3}  ( \widetilde\Theta_{m_1,m_0}^{-1}  \| (S_{\bar \BB} \AA)^{*i} \|_{\widetilde X_0 \to \widetilde X_0}) * (\widetilde \Theta_{m_1,m_0}^{-1}  \|  S_{\bar \BB} \|_{X_2 \to \widetilde X_0} ) * (\widetilde \Theta_{m_1,m_0}^{-1} \| (\AA S_{\bar \BB})^{*2} \|_{X_2 \to X_2} )
%* (\widetilde \Theta_{m_1,m_0}^{-1} \| \AA S_{\bar \BB} \|_{Z_1 \to X_2}) \\
%&\quad + (\widetilde \Theta_{m_1,m_0}^{-1}  \| (S_{\bar \BB} \AA)^{*4} \|_{E_0 \to \widetilde X_0}) * (\widetilde \Theta_{m_1,m_0}^{-1}  \|  S_{\bar \LL}  \|_{E_1 \to E_0} )* (\widetilde \Theta_{m_1,m_0}^{-1} \| (\AA S_{\bar \BB})^{*2} \|_{X_2 \to E_1} )
%* (\widetilde \Theta_{m_1,m_0}^{-1} \| \AA S_{\bar \BB} \|_{Z_1 \to X_2} )
%\ea
%$$
%}
\end{proof}

\subsection{Summary of the decay and dissipativity results for $\bar \LL$}
We introduce the appropriate functional spaces and we summarize the decay and dissipativity properties 
of the semigroup $S_{\bar\LL}$ which will be useful in the next section. 

\medskip
From now on, for a given admissible weight function $m$ such that $m \succ \langle v \rangle^{2+3/2}$, we define  
\be\label{eq:XX-YY-ZZ}
\ba
& \XX := H^2_x L^2_v(m) ,\quad \YY := H^2_x (H^1_{v,*}(m)), \quad \ZZ := H^2_x (H^{-1}_{v,*}(m)), 
\quad \XX_0 := H^2_x L^2_v .
\ea
\ee
We also define the norm $\Nt \cdot \Nt_\XX$ on $ \bar\Pi\XX$, and its associated scalar product $\la\!\la \cdot , \cdot \ra\!\ra_\XX$, given by 
\be\label{eq:NormT-XX}
\Nt g \Nt_{\XX}^2 := \eta \| g \|_{\XX}^2 + \int_0^\infty \| S_{\bar\LL} (\tau) g \|_{\XX_0}^2 \, d\tau,
\ee
for $\eta>0$ small enough.

\begin{thm}\label{thm:S_L-inhomBIS} 
Consider an admissible weight function $m$ such that $ m \succ \langle v \rangle^{2 + 3/2} $.
With the above assumptions and notations, the norm $\Nt \cdot \Nt_{\XX}$ is equivalent to the initial norm $\| \cdot \|_\XX$ on $ \bar \Pi \XX$, and moreover, 
there exists $\eta >0$ small enough such that 
\bear\label{eq:LLdecayL2-inhomBIS}
&& \la\!\la {\bar \LL} \bar \Pi f , \bar \Pi f \ra\!\ra_\XX \lesssim -  \|  \bar \Pi f \|_{\YY}^2, 
\quad \forall\, f \in \XX^{\bar \LL}_1, 
\\
\label{eq:LLregH-1-inhomBIS}
&&  t \mapsto \| S_{\bar\LL}(t) \bar \Pi   \|_{\YY \to \XX_0} \, \| S_{\bar\LL}(t) \bar \Pi  \|_{\ZZ \to \XX_0} \in L^1(\R_+) ,
\eear
 where we recall that $\XX^{\bar \LL}_1$ is the domain of $\bar \LL$ when acting on $\XX$.
\end{thm}

{ The same  remark as for  Corollary~\ref{cor:S_L-homBIS} also works here.}

\begin{proof}
The proof follows exactly the same arguments as in Proposition~\ref{prop:Nt} and Corollary~\ref{cor:S_L-homBIS}. First of all, the equivalence of the norms follows as in Proposition~\ref{prop:Nt} since $m \succ \la v \ra^{3/2}$.

Let us prove \eqref{eq:LLregH-1-inhomBIS}.
We fix admissible polynomial weight functions $m_0, m_1$ such that $\la v \ra^{(\gamma+3)/2} \prec m_0  \prec m_1 \preceq \la v \ra^{\gamma/2} m$. 
Thanks to estimate \eqref{eq:LLdecayL2-inhom} in Theorem~\ref{thm:S_L-inhom} together with the embeddings $H^2_x L^2_v(m_0) \subset \XX_0$ and $\YY \subset H^2_x L^2_v(m_1)$ we first obtain
$$
\| S_{\bar \LL}(t) \bar \Pi \|_{\YY \to \XX_0} \lesssim \Theta_{m_1,m_0}(t), 
\quad \forall\, t \ge 0.
$$ 
We know consider admissible polynomial weight functions $m'_0, m'_1$ so that $\la v \ra^{3/2} \prec m'_0 \prec m'_1 \preceq m$. Thanks to \eqref{eq:LLreg-H-1-inhom} in Theorem~\ref{thm:S_L-inhom} and the embeddings $H^2_x L^2_v(m'_0 \la v \ra^{\gamma/2}) \subset \XX_0$ and $\ZZ \subset H^2_x ( H^{-1}_{v,*}(m'_1))$, it follows
$$
\| S_{\bar\LL}(t) \bar \Pi  \|_{\ZZ \to \XX_0} \lesssim \frac{\Theta_{m'_1, m'_0}(t)}{t^{1/2} \wedge 1}, \quad \forall \, t >0.
$$
 We then deduce \eqref{eq:LLregH-1-inhomBIS} by arguing similarly as in the proof of Corollary~\ref{cor:S_L-homBIS}. 
%by observing that, under the conditions $m \succ \la v \ra^{2+3/2}$, we have
%$$
%\Theta_{m_1,m_0}(t) \, \frac{\Theta_{m'_1, m'_0}(t)}{t^{1/2} \wedge 1} 
%\lesssim \frac{\la t \ra^{- (2k-3)/|\gamma|}}{t^{1/2} \wedge 1} \in L^1(\R_+),
%$$
%for \Red some  $k > 2+3/2$. % > |\gamma|/2 + 3/2$. 
%
%\Red
%KC : tu as chang\'e "for some $ k > 2 + 3/2 $" en "for any $ k > 2 + 3/2 $", mais le dernier n?est pas vrai car les poids $m_1,m_0, m'_1, m'_0$ sont des polynomes. 
\end{proof}

\subsection{Nonlinear estimate}
From the nonlinear estimate for the homogeneous case established in Lemma~\ref{lem:Q} and Corollary~\ref{cor:Q}, we deduce the following estimate.

\begin{lem}\label{lem:nonLnonHestimQ} 
Let $m$ be an admissible weight function such that $m\succ \la v \ra^{2+3/2}$. Then 
\beqn\label{eq:nonLnonHestimQ}
\la Q(f,g), h \ra_{\XX} \lesssim \Big( \| f \|_{\XX} \, \| g \|_{\YY} 
+ \| f \|_{\YY} \, \| g \|_{\XX} \Big) \| h \|_{\YY}. 
\eeqn
As a consequence 
\be\label{eq:nonLnonHestimQinZ}
\| Q(f,g) \|_{\ZZ} \lesssim   \| f \|_{\XX} \, \| g \|_{\YY} + \| f \|_{\YY} \, \| g \|_{\XX} .
\ee

\end{lem}

\begin{proof} We proceed similarly as in \cite[Lemma 3.5]{CTW} and thus only sketch the proof.  We remark however that the estimates here are somewhat simpler than in \cite{CTW}, where the authors considered different spaces (with $3$ derivatives in $x$ and different weight functions in the $x$-derivatives) because there the weight function coming from the gain term of the linearized operator was weaker than the weight function appearing here in the loss term coming from the nonlinear estimates. 
For the most difficult term, we have thanks to Lemma~\ref{lem:Q},
\bean
&&\la \nabla^2_x Q(f,g), \nabla^2_x h \ra_{L^2_{x,v}(m)}
= \la Q(\nabla^2_xf,g) + 2 Q(\nabla_xf,\nabla_x g) + Q(f,\nabla^2_x g), \nabla^2_x h \ra_{L^2_{x,v}(m)}
\\
&&\qquad \lesssim \int_{\T^3} \Big( \| \nabla^2_x f \|_{L^2(m)} \, \| g \|_{H^1_{*}(m)} + \| \nabla^2_x f \|_{H^1_{*}(m)} \, \| g \|_{L^2(m)}  
\\
&&\qquad   + \| \nabla_x f \|_{L^2(m)} \, \| \nabla_x g \|_{H^1_{*}(m)} + \| \nabla_x f \|_{H^1_{*}(m)} \, \| \nabla_x g \|_{L^2(m)}  
\\
&&\qquad  + \| f \|_{L^2(m)} \, \| \nabla^2_x g \|_{H^1_{*}(m)} + \| f \|_{H^1_{*}(m)} \, \| \nabla^2_x g \|_{L^2(m)} \Big) \, \| \nabla^2_x h \|_{H^1_{*}(m)} \, dx.
\eean
Using first the Cauchy-Schwarz inequality in the $x$ variable and next the two Sobolev embeddings $ H^2_x \subset L^\infty_x $ and 
$ H^1_x \subset L^4_x$, we straightforwardly obtain that the above RHS term is bounded by the RHS term in \eqref{eq:nonLnonHestimQ}. The proof of \eqref{eq:nonLnonHestimQinZ} is then straightforward.
\end{proof}

\subsection{Proof of the main result}
For a solution $F$ to the inhomogeneous Landau equation \eqref{eq:landau-inhom}, we
consider the perturbation $f = F-\mu$ that verifies
\be\label{eq:pert2}
\left\{
\ba
& \partial_t f = \bar\LL f + Q(f,f) \\
&  f_0 = F_0 - \mu.
\ea
\right.
\ee
Observe that, thanks to the conservation laws, there holds $\bar\Pi_0 f(t) = \bar\Pi_0 f_0 = 0 $ and also that $\bar\Pi_0 Q(f(t),f(t)) = 0$ for any $t \ge 0$.

\begin{proof}[Proof of Theorem \ref{thm:stabNL-inhom}]
Consider the spaces and norms defined in \eqref{eq:XX-YY-ZZ} and \eqref{eq:NormT-XX}. 
The proof then follows the same arguments as in the proof of the spatially homogeneous version of Theorem~\ref{thm:stabNL-inhom} presented in Section~\ref{sec:stabNL}, by using the dissipative, decay and regularity estimates of Theorem~\ref{thm:S_L-inhomBIS} and the nonlinear estimates in Lemma~\ref{lem:nonLnonHestimQ}. 

For the sake of clarity we sketch the proof below.

Let $f$ satisfy \eqref{eq:pert2}. Thanks to Theorem~\ref{thm:S_L-inhomBIS} and Lemma~\ref{lem:nonLnonHestimQ}, arguing as in the proof of Proposition~\ref{prop:stab}, we obtain the following uniform in time a priori estimate
$$
\frac{d}{dt} \Nt f \Nt_{\XX}^2 \le (C \Nt f \Nt_{\XX} - K) \| f \|_{\YY}^2,
$$
for some constants $C,K >0$. For $\eps_0 >0$ small enough, the existence and uniqueness of  a solution $f$ for equation \eqref{eq:pert2} such that \eqref{eq:bound-g-inhom} holds are then a consequence of this last estimate together with standard arguments (as already presented in the proof of the spatially homogeneous version of Theorem~\ref{thm:stabNL-inhom} in Section~\ref{sec:stabNL}). Moreover, using the above estimate for different weight functions $\la v \ra^{2+3/2} \prec \tilde m \prec m$, the proof of the decay result \eqref{eq:decay-g-inhom} follows exactly as in the spatially homogeneous version of Theorem~\ref{thm:stabNL-inhom}.
\end{proof}

We conclude the section by presenting a proof of our improvement of the speed of convergence to the equilibrium for solutions to the spatially inhomogeneous Landau equation in a non perturbative framework. 

\begin{proof}[Proof of Corollary~\ref{cor:SpeedInhomogeneous}]
Under the assumptions \eqref{hyp:SpeedInhomo1} and \eqref{hyp:SpeedInhomo2}, \cite[Theorem 2 \& Section I.5]{DV-boltzmann} implies that
$$
\| f(t) \|_{L^1_{x,v}} \lesssim \la t \ra^{- \theta},
$$
for some explicit constant $\theta>0$. We then write the interpolation inequality
$$
\| f \|_{H^2_{x,v} (m^{\alpha})} \lesssim \| f \|_{H^3_{x,v}}^{\beta_1} \, \| f \|_{L^1_{x,v}}^{\beta_2} \, \| f \|_{L^1_{x,v}(m)}^{1-\beta_1-\beta_2},
$$
where $\alpha,\beta_1, \beta_2 \in (0,1)$ are explicit constants.  
We conclude taking $t_0 >0$ large enough so that $\| f(t_0) \|_{H^2_x L^2_v (m^\alpha)} \le \eps_0$, applying Theorem~\ref{thm:stabNL-inhom} and observing that $\Theta_{m^\alpha}(t) \simeq \Theta_m (t)$ (up to changing the constants in \eqref{eq:Theta_m}).
\end{proof}

%%%%%%%%%%%%

\bibliographystyle{acm}
\bibliography{./bib-Landau}

\end{document}